\documentclass[preprint,3p,times]{elsarticle}
\usepackage{amsmath}
\usepackage{mathrsfs}
\usepackage{tikz}
\usepackage{slashbox}
\usepackage{hyperref}
\usepackage{listings}
\usepackage{multirow}
\usepackage{amssymb}
\usepackage{amsthm}
\usepackage{epstopdf}
\usepackage{algorithm}
\usepackage{algpseudocode}
\usepackage{subfigure}
\usepackage{tabularx,booktabs}
\usepackage{caption}
\usepackage{makecell}
\usepackage{verbatim}
\usepackage[figuresright]{rotating}
\usepackage{CJK}

\biboptions{sort&compress}
\theoremstyle{plain}
  
  \newtheorem{lemma}{Lemma}[section]
  \newtheorem{proposition}{Proposition}[section]
  \newtheorem{corollary}{Corollary}[section]
\theoremstyle{definition}
  \newtheorem{definition}{Definition}[section]
  
  \newtheorem{example}{Example}[section]
  \newtheorem{remark}{Remark}[section]

\makeatletter
\newif\if@borderstar
\def\bordermatrix{\@ifnextchar*{%
\@borderstartrue\@bordermatrix@i}{\@borderstarfalse\@bordermatrix@i*}%
}
\def\@bordermatrix@i*{\@ifnextchar[{\@bordermatrix@ii}{\@bordermatrix@ii[()]}}
\def\@bordermatrix@ii[#1]#2{%
\begingroup
\m@th\@tempdima8.75\p@\setbox\z@\vbox{%
\def\cr{\crcr\noalign{\kern 2\p@\global\let\cr\endline }}%
\ialign {$##$\hfil\kern 2\p@\kern\@tempdima &\thinspace %
\hfil $##$\hfil &&\quad\hfil $##$\hfil\crcr\omit\strut %
\hfil\crcr\noalign{\kern -\baselineskip}#2\crcr\omit %
\strut\cr}}%
\setbox\tw@\vbox{\unvcopy\z@\global\setbox\@ne\lastbox}%
\setbox\tw@\hbox{\unhbox\@ne\unskip\global\setbox\@ne\lastbox}%
\setbox\tw@\hbox{%
$\kern\wd\@ne\kern -\@tempdima\left\@firstoftwo#1%
\if@borderstar\kern2pt\else\kern -\wd\@ne\fi%
\global\setbox\@ne\vbox{\box\@ne\if@borderstar\else\kern 2\p@\fi}%
\vcenter{\if@borderstar\else\kern -\ht\@ne\fi%
\unvbox\z@\kern-\if@borderstar2\fi\baselineskip}%
\if@borderstar\kern-2\@tempdima\kern2\p@\else\,\fi\right\@secondoftwo#1$%
}\null \;\vbox{\kern\ht\@ne\box\tw@}%
\endgroup
}
\makeatother
\allowdisplaybreaks
\UseRawInputEncoding
\begin{document}
\begin{frontmatter}

\title{Some neighborhood-related fuzzy covering-based rough set models and their applications for decision making}
\author{Gongao Qi\fnref{label1}}
\author{Bin Yang\fnref{label1}\corref{cor1}}
\ead{binyang0906@whu.edu.cn,binyang0906@nwsuaf.edu.cn}
\author{Wei Li\fnref{label1}}
\address[label1]{College of Science, Northwest A \& F University, Yangling 712100, PR China}
\cortext[cor1]{Corresponding author.}

\begin{abstract}	
 Fuzzy rough set (FRS) has a great effect on data mining processes and the fuzzy logical operators play a key role in the development of FRS theory. In order to further generalize the FRS theory to more complicated data environments, we firstly propose four types of fuzzy neighborhood operators based on fuzzy covering by overlap functions and their implicators in this paper. Meanwhile, the derived fuzzy coverings from an original fuzzy covering are defined and the equalities among overlap function-based fuzzy neighborhood operators based on a finite fuzzy covering are also investigated. Secondly, we prove that new operators can be divided into seventeen groups according to equivalence relations, and the partial order relations among these seventeen classes of operators are discussed, as well. Go further, the comparisons with $ t$-norm-based fuzzy neighborhood operators given by D'eer et al. are also made and two types of neighborhood-related fuzzy covering-based rough set models, which are defined via different fuzzy neighborhood operators that are on the basis of diverse kinds of fuzzy logical operators proposed. Furthermore, the groupings and partially order relations are also discussed. Finally, a novel fuzzy TOPSIS methodology is put forward to solve a  biosynthetic nanomaterials select issue, and the rationality and enforceability of our new approach is verified by comparing its results with nine different methods.
\end{abstract}


\end{frontmatter}

\section{Introduction}\label{section0}

In order to deal with incomplete and indistinguishable information in the real world, Pawlak \cite{pawlak1982rough} established a useful mathematical tool that is rough set theory (RST) in 1982. In the original theory, given an arbitrary subset of a universe $\mathbb{U}$, which can be approximately characterized via two definable sets that are called lower and upper approximations \cite{degang2006rough}.

In Pawlak's model \cite{pawlak1982rough}, an equivalence class $\mathbb{E}$ is used to define to express the indiscernible relations among pairs of elements in the universe $\mathbb{U}$. However, when this theory applies to practical problems, the equivalence relation seems to be so restricted and limited.
During recent decades, to solve more complicated problems, scholars have generalized the original model via replacing the equivalence class $\mathbb{E}$ with other more sophisticated but more general concepts \cite{pomykala1987approximation}. By means of extending the partition to a covering, the notion of covering-based rough set (CRS) was firstly proposed which is investigated by Zakowski in \cite{zakowski1983approximations}. Furthermore, scholars have also explored other forms of CRS models\cite{yao2012covering}. In particular, through further researched, via the concepts of neighborhood and granularity, two couples of dual approximation operators given by Pomykala \cite{pomykala1987approximation,pomykala1988definability}, Yao \cite{yao1998relational} put forward a novel concept named neighborhood-related CRS model.

Although RST, such as CRS, is an efficient and effective tool for dealing with discrete data, it is weak in processing real-valued data sets on the applications because of the values of the attributes, which can be described as symbolic and real-valued \cite{jensen2004fuzzy}. As an effective method to do with the vague issues and indiscernibility in the real world, fuzzy set theory (FST) has been proposed by Zadeh \cite{zadeth1965fuzzy} in 1965, and this theory is a useful tool for overcoming the above problems. These two theories are related but different and complementary \cite{yao1998comparative}, and nowadays, RST and FST are the two major methods that are used to deal with uncertainty and incomplete data in information systems. During the past two decades, scholars have taken great interest in the connection between rough sets and fuzzy sets, and they have made attempts to research in different and relevant mathematical fields\cite{yao1998comparative,wang2007learning}.
Dubois and Prade firstly introduced the notion of fuzzy rough sets \cite{dubois1990rough} in 1990. They argued to combine these two models of uncertainty (that is vagueness and coarseness) instead of competing with each other on the same issues. As a generalization of Pawlak's rough set, the CRS was also introduced into fuzzy environment. However, the early works have not defined clearly what a fuzzy covering-based fuzzy rough set (FCRS) is, such as  De Cock et al. \cite{de2004fuzzy} and  Deng \cite{deng2007novel}, the fuzzy coverings are studied as special cases or special fuzzy binary relations. Through the efforts of scholars, the original definition of fuzzy covering was summarized and proposed (see \cite{feng2012reduction,li2008generalized}):

Let $\mathbb{U}$ be an an arbitrary universal set, and $\mathscr{F}(\mathbb{U})$ be the fuzzy power set of $\mathbb{U}$. We call $\widehat{\mathbf{C}}=\{C_1, C_2, ..., C_m\}$, with $C_i\in \mathscr{F}(\mathbb{U})(i=1,2, ..., m),$ a fuzzy covering of $\mathbb{U}$, if $(\bigcup^{m}_{i=1}C_i)(x)=1$ for each $x\in \mathbb{U}$.

In line with this definition, Ma \cite{ma2016two} put forward a new generalization named fuzzy $\beta$-covering via substituting a parameter $\beta$ $(0<\beta\leq 1)$ for $1$. Besides, Yang and Hu \cite{yang2017some} also
proposed some new models based on fuzzy $\beta$-coverings and discussed their properties. Especially, the neighborhood operators in \cite{yao2012covering} were introduced into fuzzy environment by D'eer et al. \cite{d2017fuzzy} via $t$-norms and their implicators.

On the other hand, the applications of fuzzy rough set theory are very extensive which is always used to handle decision-making issues \cite{jiang2018covering}, attribute reductions \cite{wang2019fuzzy}, data mining, etc. For decision-making theory \cite{pedrycz2011risk}, multi-attribute decision-making (MADM) issues are important branches which are usually handled by the aggregation operator method \cite{yager1988ordered} , the TOPSIS method \cite{hwang1981methods}, the TODIM method \cite{llamazares2018analysis}, and so on. Introducing conventional TOPSIS method into fuzzy environment \cite{chen2000extensions} greatly expands the range of data that can be processed. Zhang et al. \cite{zhang2019topsis} even constructed the TOPSIS method based on the FCRS models with the help of the fuzzy neighborhood operators proposed by D'eer et al. \cite{d2017fuzzy}.

Nevertheless, with the development of society, the information to be disposed of has become even more sophisticated. For some data that cannot be randomly aggregated, the normal methods do not get the desired result.

In order to deal with several actual applications issues, such as image processing, decision making classifications, and so on. Bustince et al. \cite{bustince2010overlap} summarized and put forward the axiomatic definition for a kind of special binary aggregation function which is named overlap function. Comparing overlap function with $t$-norm, as two different types of aggregation functions, they are similar but different. Overlap function is not strongly required associativity property because of its birth background, which is mainly applied in classification issues, image processing, and some special decision making problems, the above features make overlap function more complicated but more flexible than $t$-norm. When an overlap function $\mathbb{O}$ satisfies the exchange principle and the boundary condition satisfies $\mathbb{O}(x,1)=x$ for all $x\in [0,1]$, it can be seen as a $t$-norm, certainly, which cannot be seen as the general condition. The  Venn diagram (Figure \ref{fig:ee}) intuitively depicts the connection between overlap functions and $t$-norms.
 \begin{figure}[h]
  \centering
  \includegraphics[width=12cm]{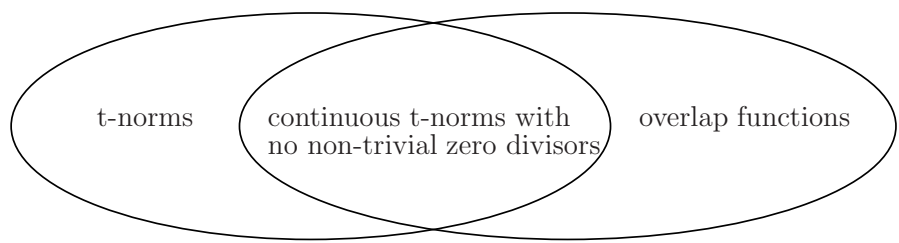}\\
\caption{Relationship between $t$-norms and overlap functions
 \cite{qiao2019distributive}}\label{fig:ee}
\end{figure}

Owing to the peculiarity and the practicability of overlap functions, scholars have researched both actual applications and theory of the overlap function. In terms of theoretical extension, Bedregal et al. \cite{bedregal2013new} put forward some conclusions on overlap functions and grouping function such as migrativity, homogeneity, idempotency and the existence of generators. In particular, Zhou and Yan \cite{zhou2021migrativity} also studied migrativity of overlap functions but over uninorms. Dimuro et al. \cite{dimuro2014archimedean} have introduced the notion of Archimedean overlap function. In \cite{wang2019modularity}, Wang and Liu studied the modularity condition of overlap and grouping functions and discussed the modularity equation among different kinds of aggregation functions. In addition, the derived concepts, such as implication operators \cite{dimuro2015residual} and additive generators \cite{dimuro2016additive} related to overlap functions and their dual functions have been investigated in recent years. Furthermore, the expanding concepts have also been researched which are induced from the original notion and the investigations include the interval overlap functions \cite{cao2018interval}, the binary
relation derived from overlap functions \cite{qiao2019binary}, distributive equations \cite{zhang2021distributivity}, and other different kinds of overlap functions \cite{da2020general}.
The practical applications of overlap functions are the other research hot topics. For instance, these concepts play key roles in image processing \cite{jurio2013some}, information classification \cite{elkano2016fuzzy}, decision making analysis \cite{bustince2011grouping}, and fuzzy community \cite{gomez2016new} and so on.

On the basis of the comparison between overlap function and $t$-norm, we conclude that the overlap functions without associative law are more complex but also more flexible. However, although FCRS theory has led to many effective methods and models in decision making problems, most of they are on the basis of $t$-norms which may be useless for the information without the exchange principle, in addition, the overlap function-based models which can deal with the above data efficiently are far less than $t$-norm-based ones. Therefore, to investigate the related generalization is necessary. The more details of research motivations for this paper are illustrated as below:
\begin{enumerate}
  \item[(1)] A continuation for the work of D'eer et al. \cite{d2017fuzzy}. The $t$-norm is one of the most common fuzzy logical operators in fuzzy environment, according to the comparison between overlap functions and $t$-norms, the overlap functions have many similar characteristics to $t$-norms and they can be transformed into each other under certain cases. Yet, for normal overlap functions, most of them do not satisfy the exchange principle, which leads their structure to the more sophisticated but more in line with the actual situation. For the practical application, such as image process, the overlap function-based operators are more suitable than $t$-norm-based ones. In addition to considering practical application aspects, we are interested in replacing the basic fuzzy logical operators of neighborhood operators, which may cause tremendous changes for the properties of fuzzy neighborhood operators and even the neighborhood-related FCRS models.
  \item[(2)]A generation for the fuzzy TOPSIS method. Like the above we discussed, along with the development of society, the data to be handled become even more complex. For great majority of TOPSIS methods, which are unable to efficaciously deal with the MADM issues with the data that cannot be random aggregation. However, the overlap function based methods can process this type of information effectively while preserving its characteristics, but it is not common to have the models based on overlap functions. To offset the lack of this field, we construct a series of fuzzy rough set models which are based on overlap function and its implicator to build up the new fuzzy TOPSIS method.
  \item[(3)] In order to certificate the universality of the models based on overlap function. Besides the models based on the neighborhood operators defined by us, we also propose a different type of neighborhood-related FCRS models which are established by the operators given in \cite{d2017fuzzy} to compare and illustrate that for conventional MADM problems, the overlap function-based models can solve as well as the traditional models based on $t$-norm.
\end{enumerate}

The following text is the outline of this paper. In Section \ref{Section 2}, we recall some basic definitions and conclusions used in this paper. This section introduces the concepts of fuzzy coverings and fuzzy neighborhood system. And then, the definitions and properties of the overlap function and its implication are discussed as well. In Section \ref{Section 3}, we give the definitions of four original overlap function-based fuzzy neighborhood operators and discuss some properties of them. The six derived fuzzy coverings are also introduced. In Section \ref{Section 4}, combining four original fuzzy neighborhood operators with six derived fuzzy coverings, we study the equalities among all twenty-four overlap function-based fuzzy neighborhood operators and group them into seventeen groupings. On the basis of the groupings, the partial order relations among new fuzzy neighborhood operators are discussed and the differences between $t$-norm-based fuzzy neighborhood operators and overlap function-based ones are studied by comparing these two kinds of operators. In Section \ref{5c}, two types of neighborhood-related FCRS models are defined and we classify them by the newly proposed partially order relations. Furthermore, the properties of approximation operators are also investigated. In light of the models put forward in the previous section, a new TOPSIS method is proposed  for the MADM problem without known weights in Section \ref{dm}. The detailed decision process and related algorithms are given in this part, as well. Finally, we give the conclusions and state future work in Section \ref{Section 6}.

\section{Preliminaries}\label{Section 2}
    In this section, we recall some preliminary concepts and results which
    are necessary for the research in this paper.

\subsection{Fuzzy covering and fuzzy neighborhood operator}
Throughout this paper, let $\mathbb{U}$ be a non-empty set called universe, and $\mathscr{F}(\mathbb{U})$ be a collection of fuzzy subsets on universe $\mathbb{U}$.
   \begin{definition}(\cite{d2017fuzzy})\label{D1}
   	Let $\mathbb{U}$ be a universe and $I$ be an (infinite) index set. For a family of fuzzy sets $$\mathscr{C}=\{K_{i}\in \mathscr{F}(\mathbb{U}):K_{i} \neq \emptyset,i\in I \},$$
   if there exists a $K\in \mathscr{C}$ such as $K(x)=1$ for all $x\in \mathbb{U}$, then $\mathscr{C}$ is called a fuzzy covering on $\mathbb{U}$.
\end{definition}

The above definition shows that for an infinite coverings, $\forall x\in \mathbb{U}$, there exists $K\in \mathscr{C}$ such that $x\in K$.
\begin{definition}(\cite{d2017fuzzy})
Let $\mathbb{U}$ be a universe. Then a mapping $N:\mathbb{U}\rightarrow \mathscr{F}(\mathbb{U})$ is called a fuzzy neighborhood operator.
\end{definition}

For each $x\in \mathbb{U}$, $N(x)\in \mathscr{F}(\mathbb{U})$ is associated with a fuzzy neighborhood operator.

In equivalent conditions, every fuzzy neighborhood operator $N$ on $\mathbb{U}$ is in one-to-one correspondence with a fuzzy binary relation $R$ on $\mathbb{U}$, i.e., $N(x)(y)=R(x,y)$ for all $x,y\in \mathbb{U}$. Hence, we can consider some properties of fuzzy neighborhood operator which are associated with the similar properties of the corresponding fuzzy binary relation.
\begin{definition}(\cite{d2017fuzzy})
	Let $\mathbb{U}$ be a universe and $N$ be a fuzzy neighborhood operator on $ \mathbb{U}$.
\begin{enumerate}
\item[(1)]
$N$ is reflexive $\Longleftrightarrow$ $N(x)(x)=1$ for all $x \in \mathbb{U}$;
\item[(2)]
$N$ is symmetric $\Longleftrightarrow$ $N(x)(y)=N(y)(x)$ for all $x, y \in \mathbb{U}$.

\end{enumerate}

  \end{definition}

\begin{definition}(\cite{d2017fuzzy})
Let $\mathscr{C}$ be a fuzzy covering on $\mathbb{U}$. For all $x \in \mathbb{U}$,$$\mathbb{C}(\mathscr{C},x)=\{K\in \mathscr{C}:K(x)>0\}$$ is called the $\mathit{fuzzy}$ $\mathit{neighborhood}$ $\mathit{system}$ of $x$;
$$md(\mathscr{C},x)=\{K\in \mathbb{C}(\mathscr{C},x):(\forall S\in \mathbb{C}(\mathscr{C},x))(S(x)=K(x),S\subseteq K\Rightarrow S=K)\}.$$
is called the $\mathit{fuzzy}$ $\mathit{minimal}$ $\mathit{description}$ of $x$;
	$$MD(\mathscr{C},x)=\{K\in \mathbb{C}(\mathscr{C},x):(\forall S\in \mathbb{C}(\mathscr{C},x))(S(x)=K(x),S\supseteq  K\Rightarrow S=K)\}$$ is called the $\mathit{fuzzy}$ $\mathit{maximal}$ $\mathit{description}$ of $x$.

\end{definition}

Based on the definition of fuzzy covering, it is easy to find that $\mathbb{C}(\mathscr{C},x)\neq \emptyset$. Meanwhile, we can see that $\mathbb{C}(\mathscr{C},x)$, $md(\mathscr{C},x)$ and $MD(\mathscr{C},x)$ are all collections of fuzzy sets. Especially, $md(\mathscr{C},x)\subseteq \mathbb{C}(\mathscr{C},x)$ and $MD(\mathscr{C},x)\subseteq \mathbb{C}(\mathscr{C},x)$ are hold for any $x\in \mathbb{U}$.
    \begin{lemma}(\cite{d2017fuzzy})\label{2.1}
    	Let $\mathscr{C}$ be a $\mathit{fuzzy}$ $\mathit{covering}$ on $\mathbb{U}$, which such that any $\mathit{descending}$ ($\mathit{resp}.$ $\mathit{ascending}$) chain is $\mathit{closed}$ under $\mathit{infimum}$ ($\mathit{resp}$. $\mathit{supremum}$), i.e., for any set $\{K_{i}\in \mathscr{C}:i\in I\}$ with $K_{i+1}\subseteq K_{i}$ ($\mathit{resp}$. $K_{i+1}\supseteq K_{i}$). Then$$\underset{i\in I}{\inf} K_{i}=\underset{i\in I}{\bigcap}K_i\in \mathscr{C}\left ( \mathit{resp}. \underset{i\in I}{\sup} K_{i}=\underset{i\in I}{\bigcup} K_{i}\in \mathscr{C} \right ).$$
    	Let $K$ be a fuzzy set which is in the fuzzy neighborhood system $\mathbb{C}(\mathscr{C},x)$. Then there exist $K_{1}\in md(\mathscr{C},x)$($\mathit{resp}$. $K_{2}\in MD(\mathscr{C},x)$) such that $K_{1}(x)=K(x)$ and $K_{1}\subseteq K $ ($\mathit{resp}$. $K_{2}(x)=K(x)$ and $K\supseteq K_{2}$).

    \end{lemma}

    	Note that, when the fuzzy covering $\mathscr{C}$ is finite, the Lemma \ref{2.1} always holds.

\subsection{Overlap function and $R_{O}$-implication}
In this part, we recall some concepts of overlap function and $R_{O}$-implication that are necessary for this paper.

\begin{definition}(\cite{bustince2010overlap})\label{D2.6}
	If a bivariate function $O$: $[0, 1]^{2}\longrightarrow[0, 1]$ satisfies the following conditions:
	\begin{flushleft}
		(O1) $O$ is commutative;\\
		(O2) $O(x, y)=0$ $\Longleftrightarrow$ $xy=0$ for any $x,y\in [0,1]$;\\
		(O3) $O(x, y)=1$ $\Longleftrightarrow$ $xy=1$ for any $x,y\in [0,1]$;\\
		(O4) $O$ is increasing;\\
		(O5) $O$ is continuous.
	\end{flushleft}
	then $O$ is called an overlap function.

Based on above definition, it is easy to find an overlap function $O$: $[0,1]^2\rightarrow [0,1]$ is associative if only if $O$ satisfies the exchange principle, i.e.,
	\begin{flushleft}

(O6)\label{O6} $\forall x,y,u\in [0,1]: O(x,O(y,u))=O(y,O(x,u)).$
	\end{flushleft}

\end{definition}

\begin{definition}(\cite{dimuro2015residual})
	An overlap function $O$: $[0, 1]^{2}\rightarrow[0, 1]$ satisfied the Property 1-section deflation if:
	\begin{flushleft}
		(O7)\label{O7} $O(x,1)\leq x$ for any $x\in [0,1]$;
	\end{flushleft}
and the Property 1-section inflation if
\begin{flushleft}
	(O8) $O(x,1)\geq x$ for any $x\in [0,1]$.
\end{flushleft}
\begin{remark}
	An overlap function $O$ satisfies (O7) and (O8) if and only if $O$ has 1 as neutral element. However, there exist overlap functions satisfying (O7)((O8)) that do not have 1 as neutral element. This is the case of the overlap function $O_{mp}(x,y)=\min\{x^{p},y^{p}\}$, with $p>0$ and $p\neq 1$. Whenever $p>1$, $O_{mp}$ satisfies (O7)(but not (O8)). On the other hand, when $p<1$, $O_{mp}$ satisfies (O8) and not (O7).
\end{remark}
\end{definition}

\begin{definition}(\cite{dimuro2015residual})
	A function $I:[0,1]^{2}\rightarrow [0,1]$ is a fuzzy implicator if, for each $x, y, z\in [0,1]$ it holds that:
	\begin{flushleft}
		(I1) First place antitonicity: if $x\leq y$, then $I(y,z)\leq I(x,z)$;\\
		(I2) Second place isotonicity: if $y\leq z$, then $I(x,y)\leq I(x,z)$;\\
		(I3) Boundary condition: $I(0,0)=1$; $I(1,1)=1$; $I(1,0)=0$.\\
	\end{flushleft}
\end{definition}

\begin{definition} (\cite{dimuro2015residual})
	Let $O:[0,1]^{2}\rightarrow [0,1] $ be an overlap function. Then $I_O:[0,1]^{2}\rightarrow [0,1]$ is a residual implication of $O$ (i.e., $R_O$-implicator), where
	$$\forall x,y \in [0,1], I_O(x,y)=\max\{z\in [0,1]:O(x,z)\leq y\}.$$
	\begin{remark}
		The function $I_O$ is a fuzzy implicator. $O$ and $I_O$ form an adjoint pair, that is, they satisfy the property $\forall x,y,u\in [0,1], O(x,u)\leq y \Longleftrightarrow I_O(x,y)\geq u.$
	\end{remark}
	\end{definition}
\begin{example}(\cite{dimuro2015residual})
	The following three overlap functions with different conditions do not satisfy (O6).
	\begin{enumerate}[(1)]
		\item The overlap function $O_2^V(x,y)$ that satisfies (O7):
		$$O_2^V(x,y)=\left\{\begin{matrix}
		\frac{1+(2x-1)^2(2y-1)^2}{2}, & \text{if}$\ $ x,y\in (0.5,1];\\
	\min\{x,y\}, & \text{otherwise}.
		\end{matrix}\right.$$
		And the $R_O$-implicator:
		$$I_{O_2^V}(x,y)=\left\{\begin{matrix}
			\min \left\{1, \frac{\sqrt{2y-1}}{2(2x-1)}+\frac{1}{2}\right\}, & \text{if} $\ $ x\in (0.5,1], y\in [0.5,1];\\
			y, & \text{if}$\ $ y\in [0,0.5), x>y; \\
			1, & \text{if}$\ $ x\in [0,0.5], x\leq y.
		\end{matrix}\right.$$
	\item The overlap function $O_{m\frac{1}{2}}(x,y)$ that satisfies (O8):
	$$O_{m\frac{1}{2}}=\min\left\{\sqrt{x},\sqrt{y}\right\}.$$
	And the $R_O$-implicator:
	$$I_{O_{m\frac{1}{2}}}=\left\{\begin{matrix}
		1, & \text{if}$\ $ \sqrt{x}\leq y;\\
		y^2, & \text{if}$\ $ \sqrt{x}>y.
	\end{matrix}\right.$$
\item The overlap function $O^V_{mM}$ that satisfies (O7) and (O8):
$$ O^V_{mM}(x,y)=\left\{\begin{matrix}
	\frac{1+\min\left\{2x-1,2y-1\right\}\max\left\{(2x-1)^2,(2y-1)^2\right\}}{2}, & \text{if}$\ $ x,y\in (0.5,1];\\
	\min\{x,y\}, & \text{otherwise}.
\end{matrix}\right.$$
And the $R_O$-implicator:
$$I_{O_{mM}^V}(x,y)=\left\{\begin{matrix}
	\min \left\{1, \min\left\{\frac{\sqrt{2y-1}}{2\sqrt{2x-1}},\frac{2y-1}{2(2x-1)^2}\right\}+\frac{1}{2}\right\}, & \text{if}$\ $ x\in (0.5,1], y\in [0.5,1];\\
	y, & \text{if}$\ $ y\in [0,0.5), x>y; \\
	1, &\text{if}$\ $ x\in [0,0.5], x\leq y.
\end{matrix}\right.$$
		 \end{enumerate}
\end{example}

\section{Fuzzy neighborhood operators with overlap functions based on fuzzy coverings}\label{Section 3}
  In this section, we give the definitions of four fuzzy neighborhood operators based on fuzzy covering by using overlap functions and their implicators. Considering similarities and differences between $t$-norms and overlap functions, and we discuss some properties of these fuzzy neighborhood operators. And in the final part of this section, five new fuzzy coverings are derived from the original fuzzy covering $\mathscr{C}$.

  \subsection{Fuzzy neighborhood operators based on overlap functions and some properties of them}
  Now, on the basis of the definitions of fuzzy neighborhood operators based on $t$-norm proposed by D'eer et al.\cite{d2017fuzzy}, we give the definitions of four neighborhood operators $N_{1}^{\mathscr{C}}$, $N_{2}^{\mathscr{C}}$, $N_{3}^{\mathscr{C}}$, $N_{4}^{\mathscr{C}}$ with overlap function based on fuzzy covering $\mathscr{C}$.

 Considering the fuzzy neighborhood operator $\mathbb{N}_{1}^{\mathscr{C}}$ (see \cite{d2017fuzzy}), for which the fuzzy neighborhood $\mathbb{N}_{1}^{\mathscr{C}}(x)$ is defined by $$\mathbb{N}_1^{\mathscr{C}}(x):\mathbb{U}\rightarrow [0,1]:y\mapsto \underset{K\in \mathscr{C}}{\inf} \mathscr{I}(K(x),K(y)).$$ Then we can consider a natural extension of this definition by replacing the implicator $\mathscr{I}$ with $I_O$, which is based on an overlap function $O$.
 \begin{definition}
 	Let $\mathscr{C}$ be a fuzzy covering on $\mathbb{U}$ and $I_O$ an $R_O$-implicator. Then $N_{1}^{\mathscr{C}}: \mathbb{U}\rightarrow \mathscr{F}(\mathbb{U}):x \mapsto N_{1}^{\mathscr{C}}$ is an overlap function- based fuzzy neighborhood operator, for which the overlap function based fuzzy neighborhood $N_{1}^{\mathscr{C}}(x)$ is defined by
 	$$N_1^{\mathscr{C}}(x):\mathbb{U}\rightarrow [0,1]:y\mapsto \underset{K\in \mathscr{C}}{\inf}I_{O}(K(x),K(y)).$$

 	\end{definition}
 	If the overlap function, which defines the overlap function-based fuzzy neighborhood operator $N^\mathscr{C}_1$, satisfies (O6), the new operator $N^\mathbf{C}_1$ can be seen a kind of $\mathbb{N}_{1}^{\mathscr{C}}$.
\begin{proposition}\label{p:3.1}
	Let $\mathscr{C}$ be a finite fuzzy covering on $\mathbb{U}$ and $I_O$ be an $R_O$-implicator. Then $\forall x,y \in \mathbb{U}$, it holds that$$\underset{K\in \mathscr{C}}{\inf} I_O(K(x),K(y))=\underset{K\in \mathbb{C}(\mathscr{C},x)}{\inf} I_O(K(x),K(y))=\underset{K\in md(\mathscr{C},x)}{\inf} I_O(K(x),K(y)).$$
\end{proposition}
\begin{proof}
First note that if $K(x)=0$, because of the boundary condition of $R_O$-implicator, we have $I_O(K(x),K(y))=1$, hence $$\underset{K\in \mathscr{C}}{\inf}I_O(K(x),K(y))=\underset{K\in \mathbb{C}(\mathscr{C},x)}{\inf}I_O(K(x),K(y)).$$ Since $md(\mathscr{C},x)\subseteq \mathbb{C}(\mathscr{C},x)$, we have that $$\underset{K\in \mathbb{C}(\mathscr{C},x)}{\inf}I_O(K(x),K(y))=\min\left(\underset{K\in md (\mathscr{C},x)}{\inf}I_O(K((x),K(y)), \underset{K\in \mathbb{C}(\mathscr{C},x)\setminus md(\mathscr{C},x)}{\inf}I_O(K(x),K(y)) \right).$$ If $K\in \mathbb{C}(\mathscr{C},x)\setminus md(\mathscr{C},x)$, then there exists a $K^{'}\in md(\mathscr{C},x)$ such that $K^{'}\subseteq K$ and $K^{'}(x)=K(x)$. Therefore, $\forall y \in U$, $$I_O(K(x),K(y))=I_O(K^{'}(x),K(y))\geq I_O(K^{'}(x),K^{'}(y)). $$ Hence, we can conclude that $$\underset{K\in md(\mathscr{C},x)}{\inf}I_O(K(x),K(y))\leq \underset{K\in \mathbb{C}(\mathscr{C},x)\setminus md(\mathscr{C},x)}{\inf}I_O(K(x),K(y))$$and thus, $$ \underset{K\in \mathscr{C}}{\inf}I_O(K(x),K(y))=\underset{K\in md(\mathscr{C},x)}{\inf}I_O(K(x),K(y)).$$

\end{proof}
In order to apply Lemma \ref{2.1}, the fuzzy covering $\mathscr{C}$ is assumed to be finite.

The definition of fuzzy neighborhood operator $\mathbb{N}_{2}^{\mathscr{C}}$ is given by \cite{d2017fuzzy}, and the fuzzy neighborhood $\mathbb{N}_2^{\mathscr{C}}(x)$ is defined by $$\mathbb{N}_2^{\mathscr{C}}(x):\mathbb{U}\rightarrow [0,1]:y\mapsto \underset{K\in md(\mathscr{C},x)}{\sup} \mathscr{T}(K(x),K(y)).$$ Then we can consider a natural extension of this definition by replacing the t-norm $\mathscr{T}$ with an overlap function $O$.
\begin{definition}
	Let $\mathscr{C}$ be a fuzzy covering on $\mathbb{U}$ and $O$ an overlap function. Then $N_2^{\mathscr{C}}:\mathbb{U}\rightarrow \mathscr{F}(\mathbb{U}): x\mapsto N_2^{\mathscr{C}}$ is an overlap function based fuzzy neighborhood operator, for which the overlap function-based fuzzy neighborhood $N_2^{\mathscr{C}}(x)$ is defined by $$N_2^{\mathscr{C}}(x):\mathbb{U}\rightarrow [0,1]:y\mapsto \underset{K\in md(\mathscr{C},x)}{\sup}O(K(x),K(y)).$$
\end{definition}
Note that if $O$, which is used to construct the overlap function-based fuzzy neighborhood operator $N_2^{\mathscr{C}}$, satisfies the exchange principle (O6). The new operator can also be seen a sort of $\mathbb{N}_2^{\mathscr{C}}$.

Considering the fuzzy neighborhood operator $\mathbb{N}_3^{\mathscr{C}}$ (see \cite{d2017fuzzy}), which is defined by an implicator $\mathscr{I}$, and the fuzzy neighborhood $\mathbb{N}_3^{\mathscr{C}}(x)$ is defined as follow:$$\mathbb{N}_3^{\mathscr{C}}(x):\mathbb{U}\rightarrow [0,1]:y\mapsto \underset{K\in MD(\mathscr{C},x)}{\inf}\mathscr{I}(K(x),K(y)).$$ A natural extension of this definition by replacing the implicator $\mathscr{I}$ with $I_O$ based on an overlap function.
\begin{definition}
		Let $\mathscr{C}$ be a fuzzy covering on $\mathbb{U}$ and $I_{O}$ an $R_O$-implicator. Then $N_{3}^{\mathscr{C}}: \mathbb{U}\rightarrow \mathscr{F}(\mathbb{U}):x \mapsto N_{3}^{\mathscr{C}}$ is an overlap function- based fuzzy neighborhood operator, for which the overlap function based fuzzy neighborhood $N_{3}^{\mathscr{C}}(x)$ is defined by
	$$N_3^{\mathscr{C}}(x):\mathbb{U}\rightarrow [0,1]:y\mapsto \underset{K\in MD(\mathscr{C},x)}{\inf}I_{O}(K(x),K(y)).$$
\end{definition}
Given an $R_O$-implicator $I_O$, which is derived by an associative overlap function $O$. It can be use to define an overlap function-based neighborhood operator $N_3^{\mathscr{C}}$, which is also a kind of $\mathbb{N}_3^{\mathscr{C}}$.

Finally, we consider how to extend the fuzzy neighborhood operator $\mathbb{N}_4^\mathscr{C}$ (see \cite{d2017fuzzy}) to the overlap function based fuzzy neighborhood operator. The fuzzy neighborhood $\mathbb{N}_4^{\mathscr{C}}(x)$ is defined by $$\mathbb{N}_4^{\mathscr{C}}(x):\mathbb{U}\rightarrow [0,1]: y\mapsto \underset{K\in \mathscr{C}}{\sup}\mathscr{T}(K(x),K(y)).$$  A natural extension of this definition by replacing the t-norm $\mathscr{T}$ with the overlap function $O$.
\begin{definition}
	Let $\mathscr{C}$ be a fuzzy covering on $\mathbb{U}$ and $O$ an overlap function. Then $N_4^{\mathscr{C}}:\mathbb{U}\rightarrow \mathscr{F}(\mathbb{U}): x\mapsto N_4^{\mathscr{C}}$ is an overlap function based fuzzy neighborhood operator, for which the overlap function-based fuzzy neighborhood $N_4^{\mathscr{C}}(x)$ is defined by $$N_4^{\mathscr{C}}(x):\mathbb{U}\rightarrow [0,1]:y\mapsto \underset{K\in \mathscr{C}}{\sup}O(K(x),K(y)).$$
\end{definition}
 It is easy to see when the overlap function $O$ satisfies (O6), and the overlap function-based fuzzy neighborhood operator $N_4^\mathscr{C}$ is a kind of $\mathbb{N}_4^\mathscr{C}$.
 \begin{proposition}\label{p:3.2}
 	Let $\mathscr{C}$ be a finite fuzzy covering on $\mathbb{U}$ and $O$ be an overlap function. Then $\forall x,y\in \mathbb{U}$, it holds that $$\underset{K\in \mathscr{C}}{\sup}O(K(x),K(y))=\underset{K\in \mathbb{C}(\mathscr{C},x)}{\sup}O(K(x),K(y))=\underset{K\in MD(\mathscr{C},x)}{\sup}O(K(x),K(y)).$$
 \end{proposition}
\begin{proof}
	Analogously as the proof of Proposition \ref{p:3.1}, we can prove that $$\underset{K\in \mathscr{C}}{\sup}O(K(x),K(y))=\underset{K\in \mathbb{C}(\mathscr{C},x)}{\sup}O(K(x),K(y))$$ and $$\underset{K\in \mathscr{C}}{\sup}O(K(x),K(y))=\underset{K\in MD(\mathscr{C},x)}{\sup}O(K(x),K(y))\geq \underset{K\in \mathbb{C}(\mathscr{C},x)\setminus MD(\mathscr{C},x)}{\sup}O(K(x),K(y)).$$
\end{proof}

In \cite{d2017fuzzy}, when $\mathscr{C}$ is a fuzzy covering on $\mathbb{U}$, $\mathscr{T}$ is a $t$-norm and $\mathscr{I}$ is the residual implication (i.e. $R$-implicator) of $\mathscr{T}$, the reflexivity properties of the fuzzy operators $\mathbb{N}^{\mathscr{C}}_1$, $\mathbb{N}^{\mathscr{C}}_2$, $\mathbb{N}^{\mathscr{C}}_4$ and $\mathbb{N}^{\mathscr{C}}_4$ are all hold in different cases. Besides these, $\mathbb{N}^{\mathscr{C}}_1$ and $\mathbb{N}^{\mathscr{C}}_3$ have transitivity property and $\mathbb{N}^{\mathscr{C}}_4$ has transitivity property. In this section, these three kinds of properties of overlap function-based neighborhood operators are discussed as well.

First of all, we consider the reflexivity property of the new fuzzy operators, and we give a proposition as below.
 \begin{proposition}
 	Let $\mathscr{C}$ be a fuzzy covering, $O$ an overlap function and $I_O$ an $R_O$-implicator based on an overlap function which satisfies (O7), the following properties are hold.
\begin{enumerate}
	\item[(1)]
	$N_1^{\mathscr{C}}$ and $N_3^{\mathscr{C}}$ are reflexive;
	\item [(2)]
	$N_4^{\mathscr{C}}$ is reflexive;
	\item[(3)]
	If $\mathscr{C}$ is finite, $N_2^{\mathscr{C}}$ is reflexive.
\end{enumerate}
 	
 \end{proposition}
\begin{proof}
	 If $x\leq y$, since (O7), we have $O(x,1)\leq x \leq y$, then $I_O(x,y)=max\{z\in [0,1]|O(x,z)\leq y\}=1$. For all $a\in [0,1]$, $I(a,a)=1$, then we can get $$N_1^{\mathscr{C}}(a)(a)=\underset{K\in \mathscr{C}}{\inf}I_O(K(a),K(a))=1,$$ and also get $N_3^{\mathscr{C}}(a)(a)=1$. $N_1^{\mathscr{C}}$ and $N_3^{\mathscr{C}}$ are reflexive. Moreover, let $x\in U$. Then $\exists K\in \mathscr{C}$ such that $K(x)=1$. Hence, $$N_4^{\mathscr{C}}(x)(x)\geq O(K(x),K(x))=1.$$
	
	 If $\mathscr{C}$ is finite, because of Lemma \ref{2.1}, there exists $K_1\in md(\mathscr{C},x)$ with $K_1(x)=K(x)=1$ and $K_1\subseteq K$. So that, $$N_2^{\mathscr{C}}(x)(x)\geq O(K_1(x),K_1(x))=1,$$ hence, $N_2^{\mathscr{C}}$ is reflexive.
\end{proof}
Now, we give an example to interpret that the property (O7) is necessary for the dual overlap function of $I_O$.
\begin{example}
	Considering the overlap function $O_{m\frac{1}{2}}$, and the function satisfies the property (O8), and the implicator $I_{O_{m\frac{1}{2}}}$ is used to define operators $N_1^{\mathscr{C}}$ and $N_3^{\mathscr{C}}$. Let $\mathbb{U}$ be a set with $\mathbb{U}=\{x,y\}$ and $\mathscr{C}$ a fuzzy covering with $\mathscr{C}=\{K_1,K_2\} $. And for covering $\mathscr{C}$ we have $K_1=\frac{1}{x}+\frac{0.81}{y}$ and $K_2=\frac{0.64}{x}+\frac{1}{y}$. Then we get that$$N_1^{\mathscr{C}}(x)(x)=N_3^{\mathscr{C}}(x)(x)=0.4096\neq 1.$$Hence, the operators $N_1^{\mathscr{C}}$ and $N_3^{\mathscr{C}}$ are not reflexive in this condition.
\end{example}
Secondly, we consider about transitivity property. If $\mathscr{T}$ is a left-continuous t-norm and $\mathscr{I}$ is its R-implicator, for each $x,y,z\in \mathbb{U}$ it holds that $$\mathscr{T}(\mathbb{N}_1^{\mathscr{C}}(y)(x),\mathbb{N}_1^{\mathscr{C}}(x)(z))\leq \mathbb{N}_1^{\mathscr{C}}(y)(z),$$ $$\mathscr{T}(\mathbb{N}_3^{\mathscr{C}}(y)(x),\mathbb{N}_3^{\mathscr{C}}(x)(z))\leq \mathbb{N}_3^{\mathscr{C}}(y)(z),$$ i.e., $\mathbb{N}_1^{\mathscr{C}}$ and $\mathbb{N}_3^{\mathscr{C}}$ are $\mathscr{T}$-transitive fuzzy neighborhood operators (see \cite{d2017fuzzy}). Analogously, we can give the definition of $O$-transitive.
\begin{definition}
	Let $N$ be a fuzzy neighborhood operator, $x,y,z \in \mathbb{U}$, and $O$ an overlap function.
	Then $N$ is $O$-transitive if and only if the operator satisfies $O(N(x)(y),N(y)(z))\leq N(x)(z)$.
\end{definition}
 For the overlap function-based fuzzy neighborhood operators, if the overlap function is associative which can be seen as a t-norm, then we have the operators $N_1^{\mathscr{C}}$ and $N_3^{\mathscr{C}}$ which are defined by its $R_O$-implicator are $O$-transitive. If not, these two operators do not have this property.

  Considering the overlap function $O_{m\frac{1}{2}}$ and its $R_O$-implicatorr $I_{O_{m\frac{1}{2}}}$, we can give the following example:
\begin{example}\label{ex3.2}
	Let $\mathbb{U}$ be a set with $\mathbb{U}=\{x,y,z\}$ and $\mathscr{C}$ a fuzzy covering with $\mathscr{C}=\{K_1, K_2\}$. And for covering $\mathscr{C}$ we have $K_1=\frac{1}{x}+\frac{1}{y}+\frac{0.8}{z}$ and $K_2=\frac{0.9}{x}+\frac{0.9}{y}+\frac{1}{z}$. The $R_O$-implicator $I_{O_{m\frac{1}{2}}}$ is used to define the operator $N_1^{\mathscr{C}}$ and $N_3^{\mathscr{C}}$, and then we have that$$N_1^{\mathscr{C}}(y)(x)=N_3^{\mathscr{C}}(y)(x)=0.81,$$ $$N_1^{\mathscr{C}}(x)(z)=N_3^{\mathscr{C}}(x)(z)=0.81,$$ $$N_1^{\mathscr{C}}(y)(z)=N_3^{\mathscr{C}}(y)(z)=0.64.$$
	Therefore, we have that $O(0.81,0.81)=0.9>0.64$, hence the overlap function based fuzzy operators $N_1^{\mathscr{C}}$ and $N_3^{\mathscr{C}}$ are not $O_{m\frac{1}{2}}$-transitive.
\end{example}
Finally, we will discuss about symmetric property. In \cite{d2017fuzzy}, D'eer et al. proved the fuzzy neighborhood operator $\mathbb{N}_4^{\mathscr{C}}$ is symmetric. Analogously, we can also prove overlap function-based fuzzy neighborhood operator $N_4^{\mathscr{C}}$ satisfies symmetric property.
\begin{proposition}
	Let $\mathscr{C}$ be a fuzzy covering on $\mathbb{U}$ and $O$ an overlap function. Then for all $x,y\in \mathbb{U}$ it holds that $N_{4}^{\mathscr{C}}(x)(y)=N_4^{\mathscr{C}}(y)(x).$
	
\end{proposition}
\begin{proof}
	This follows immediately from the fact that an overlap function satisfies (O1), i.e. overlap function is commutative.
	\end{proof}
Note that, although the overlap function-based fuzzy neighborhood operator $N_2^{\mathscr{C}}$ is also constructed with an overlap function, the operator does not satisfies symmetric property. Because the minimal description of ``$\forall x,y\in U$" are not necessarily equal.
\subsection{Fuzzy covering derived from a fuzzy covering}
Given a fuzzy covering $\mathscr{C}$, we can give the definitions of the derived coverings $\mathscr{C}_1$, $\mathscr{C}_2$, $\mathscr{C}_3$, $\mathscr{C}_4$, $\mathscr{C}_{\cap}$ and $\mathscr{C}_{\cup}$.
\begin{definition}\cite{d2017fuzzy}
	Let $\mathscr{C} $ be a fuzzy covering on $\mathbb{U}$, $O$ an overlap function and $I_O$ an $R_O$-implicator. Then define the following collections of fuzzy set:
		\begin{enumerate}
		\item[(1)]
		$\mathscr{C}_{1}=\bigcup\{md(\mathscr{C},x):x\in \mathbb{U}\},$
		\item[(2)]
		$\mathscr{C}_{2}=\bigcup\{MD(\mathscr{C},x):x\in \mathbb{U}\},$
		
		\item[(3)]
		$\mathscr{C}_{\cap}=\mathscr{C}\setminus \{K\in \mathscr{C}: (\exists \mathscr{C}^{'}\subseteq \mathscr{C} \setminus \{K\})(K=\bigcap \mathscr{C}^{'})\},$
		\item[(4)]
		$\mathscr{C}_{\cup}=\mathscr{C}\setminus \{K\in \mathscr{C}: (\exists \mathscr{C}^{'}\subseteq \mathscr{C} \setminus \{K\})(K=\bigcup \mathscr{C}^{'})\}.$
	\end{enumerate}
\end{definition}
In \cite{d2017fuzzy}, the fuzzy covering $\mathscr{C}_{\cap}$ is a fuzzy subcovering of the fuzzy covering $\mathscr{C}$, if $\mathscr{C}$ is finite, the fuzzy covering $\mathscr{C}_1$, $\mathscr{C}_2$ and $\mathscr{C}_{\cup}$ are all subcoverings of it. Especially, when $\mathscr{C}$ is a finite fuzzy covering, we have $\mathscr{C}_{\cup}=\mathscr{C}_1$ and $\mathscr{C}_2$ is a fuzzy subcovering of $\mathscr{C}_{\cap}$.
\begin{proposition}\cite{d2017fuzzy}
	Let $\mathscr{C}$ be a finite fuzzy covering on $\mathbb{U}$. Then $\mathscr{C}_{1}$, $\mathscr{C}_{2}$ and $\mathscr{C}_{\cup}$ are all finite fuzzy subcoverings of $\mathscr{C}$.
\end{proposition}
\begin{proposition}\cite{d2017fuzzy}
	Let $\mathscr{C}$ be a fuzzy covering on $\mathbb{U}$. Then $\mathscr{C}_{\cap}$ is a fuzzy subcovering of $\mathscr{C}$.
\end{proposition}
\begin{proposition}\label{p:3.7}\cite{d2017fuzzy}
	Let $\mathscr{C}$ be a finite fuzzy covering on $\mathbb{U}$. Then $\mathscr{C}_{2}$ is a fuzzy subcovering of $\mathscr{C}_{\cap}$.
\end{proposition}
\begin{proposition}\cite{d2017fuzzy}
	Let $\mathscr{C}$ be a finite fuzzy covering on $\mathbb{U}$. Then $\mathscr{C}_{\cup}=\mathscr{C}_{1}$.
\end{proposition}
\begin{proposition}\label{p:3.9}\cite{d2017fuzzy}
	Let $\mathscr{C}$ be a finite fuzzy covering on $\mathbb{U}$. Then for all $x\in U$ it holds that
	\begin{enumerate}
		\item[(1)]
		$md(\mathscr{C}_1,x)=md(\mathscr{C},x)$,
		\item[(2)]
		$MD(\mathscr{C}_2,x)=MD(\mathscr{C},x)$,
		\item[(3)]
		$MD(\mathscr{C}_{\cap},x)=MD(\mathscr{C},x)$.
	\end{enumerate}
	\end{proposition}

 Now, we consider about the other fuzzy coverings $\mathscr{C}_3$ and $\mathscr{C}_4$.
\begin{definition}
	Let $\mathscr{C} $ be a fuzzy covering on $\mathbb{U}$, $N_{4}^{\mathscr{C}}$ an overlap function-based fuzzy neighborhood operator defined by $O$ and $N_{1}^{\mathscr{C}}$ an overlap function based fuzzy neighborhood operator defined by $I_O$. Then define the following collections of fuzzy set:
		\begin{enumerate}
	
		\item[(1)]
		$\mathscr{C}_{3}=\{N_{1}^{\mathscr{C}}(x):x\in \mathbb{U}\},$
		\item[(2)]
		$\mathscr{C}_{4}=\{N_{4}^{\mathscr{C}}(x):x\in \mathbb{U}\}.$
		
	\end{enumerate}
\end{definition}

\begin{proposition}
	Let $\mathscr{C}$ be a fuzzy covering on $\mathbb{U}$, $\mathscr{C}_4$ a fuzzy set defined by an overlap function $O$ and $\mathscr{C}_3$ a fuzzy set defined by the $R_O$-implicator of $O$. If $O$ satisfies (O7), the fuzzy sets $\mathscr{C}_3$ and $\mathscr{C}_4$ are two fuzzy coverings.
\end{proposition}
\begin{proof}
	we can prove immediately from the fact that operators $N_1^{\mathscr{C}}$ and $N_4^{\mathscr{C}}$ are all reflexive.
\end{proof}
In this part, five generalized fuzzy coverings $\mathscr{C}_1(\mathscr{C}_1=\mathscr{C}_{\cup}), \mathscr{C}_2, \mathscr{C}_3, \mathscr{C}_4$ and $\mathscr{C}_{\cap}$ are derived by the original fuzzy covering $\mathscr{C}$.
\section{The groupings of overlap function-based fuzzy neighborhood operators based on a finite fuzzy covering and their lattice relationships}\label{Section 4}
Combining four original operators with five derived coverings, twenty new operators can be obtained. Based on the previous sections, the relations among the total twenty-four operators are discussed in the two following parts. Furthermore, by the end of this section, the similarities and differences between new kinds of operators and others discussed in \cite{d2017fuzzy} will be analyzed.
\subsection{Equalities among overlap function-based fuzzy neighborhood operators based on a finite fuzzy covering}
In this subsection, the equalities among twenty four overlap function-based fuzzy neighborhood operators are discussed. When an overlap function satisfies the exchange principle, it can be seen as a $t$-norm and the conclusions of coverings and operators which are defined by $t$-norm have proven in \cite{d2017fuzzy}. Therefore, we need to prove whether the groupings stated in \cite{d2017fuzzy} are still maintained when the overlap function which is used to establish fuzzy coverings and fuzzy neighborhood operators does not satisfy the property (O6). And in the end of this part, we will divide the operators in seventeen different groups. Note that, in this part, we assume the original covering $\mathscr{C}$ is finite, therefore, the derived fuzzy covering $\mathscr{C}_{\cup}$ is equal to the covering $\mathscr{C}_1$. So that, we just need to discuss the relations among twenty one operators.

First, we consider the operators which are $N_1^{\mathscr{C}}$, $N_1^{\mathscr{C}_1}$, $N_1^{\mathscr{C}_3}$ and $N_1^{\mathscr{C}_{\cap}}$.
\begin{proposition}
	Let $\mathscr{C}$ be a fuzzy covering on $\mathbb{U}$ and $I_O$ an $R_O$-implicator. $N_1^{\mathscr{C}}$, $N_1^{\mathscr{C}_1}$ and $N_1^{\mathscr{C}_{\cap}}$ are three fuzzy neighborhood operators defined by $I_O$. Then
		\begin{enumerate}
			\item[(1)]
		$N_1^{\mathscr{C}}=N_1^{\mathscr{C}_1},$
		\item[(2)]
		$N_1^{\mathscr{C}}=N_1^{\mathscr{C}_{\cap}}.$
		\end{enumerate}
\end{proposition}
\begin{proof}
	\quad
	\begin{enumerate}
		\item[(1)]
		This follows immediately from Proposition \ref{p:3.1} and (1) of Proposition \ref{p:3.9}.
		\item[(2)]
		Since $\mathscr{C}_{\cap}\subseteq \mathscr{C}$, we can get $N_1^{\mathscr{C}}(x)\subseteq N_1^{\mathscr{C}_{\cap}}(x)$ for all $x\in \mathbb{U}$. Now we need to prove $N_1^{\mathscr{C}_{\cap}}(x)\subseteq N_1^{\mathscr{C}}(x)$, take $y\in \mathbb{U}$, since $\mathscr{C}$ is a finite fuzzy covering, let $K\in \mathscr{C}$ such that $N_1^{\mathscr{C}}(x)(y)=I_O(K(x),K(y)).$ If $K\in \mathscr{C}$, we can get $N_1^{\mathscr{C}_{\cap}}(x)(y)\leq I_O(K(x),K(y))=N_1^{\mathscr{C}}(x)(y)$. In other condition, if $K\neq \mathscr{C}_{\cap}$, a collection $\mathscr{C}^{'}\subseteq \mathscr{C}_{\cap}$ can be found such that $K=\bigcap \mathscr{C}^{'}$. Since $\mathscr{C}$ is finite, there exists $K'\in \mathscr{C}^{'}$ with $K'(y)=K(y)$ and $K\subseteq K'.$ Hence, $I_O(K(x),K(y))\geq I_O((K'(x),K'(y))$. Since $ \mathscr{C}^{'}\subseteq \mathscr{C}$ then we have $K'\in \mathscr{C}$ and the infimum of $N_1^{\mathscr{C}}(x)(y)$ is reached in $K$. So that, $N_1^{\mathscr{C}_{\cap}}(x)\leq I_O((K'(x),K'(y))=I_O(K(x),K(y))=N_1^{\mathscr{C}}(x)(y)$, then we get the conclusion that $N_1^{\mathscr{C}}(x)\subseteq N_1^{\mathscr{C}_{\cap}}(x)$. In both cases, we conclude that $N_1^{\mathscr{C}}=N_1^{\mathscr{C}_{\cap}}$.
	\end{enumerate}
\end{proof}
Especially, when the overlap function which is used to define $R_O$-implicator $I_O$ satisfies (O6), we can also conclude that $N_1^{\mathscr{C}}=N_1^{\mathscr{C}_{3}}$. If not, we get the opposite conclusion.
\begin{example}\label{e:4.1}
	Let $\mathbb{U}$ be a set with $\mathbb{U}=\{x,y\}$ and $\mathscr{C}$ a fuzzy covering with $\mathscr{C}=\{K_1, K_2\}$. And for covering $\mathscr{C}$ we have  $K_1=\frac{1}{x}+\frac{0.905}{y}$ and $K_2=\frac{0.745}{x}+\frac{1}{y}$, considering the overlap function $O_2^{V}$ and the $R_O$-implicator $I_{O_2^V}$, we can get the fuzzy covering $\mathscr{C}_3=\{N_1^{\mathscr{C}}(x),N_1^{\mathscr{C}}(y)\}$ with $N_1^{\mathscr{C}}(x)=\frac{1}{x}+\frac{0.95}{y}$ and $N_1^{\mathscr{C}}(y)=\frac{0.85}{x}+\frac{1}{y}$. We have that $N_1^{\mathscr{C}}(x)(y)=0.95$ and $N_1^{\mathscr{C}_3}(x)(y)\geq 0.974$, i.e., $N_1^{\mathscr{C}}\neq N_1^{\mathscr{C}_3}$.
\end{example}
Next, we will discuss the relation between the operators $N_2^{\mathscr{C}}$ and $N_2^{\mathscr{C}_1}$.
\begin{proposition}
	Let $\mathscr{C}$ be a fuzzy covering and $O$ an overlap function. $N_2^{\mathscr{C}}$ and $N_2^{\mathscr{C}_1}$ are two fuzzy neighborhood operators constructed by $O$. Then we have $N_2^{\mathscr{C}}=N_2^{\mathscr{C}_1}$.
\end{proposition}
\begin{proof}
	This follows immediately from (1) of Proposition \ref{p:3.9}.
\end{proof}
Thirdly, we consider whether the operators $N_3^{\mathscr{C}}$, $N_3^{\mathscr{C}_2}$ and $N_3^{\mathscr{C}_{\cap}}$ can be divided into the same grouping.
\begin{proposition}
	Let $\mathscr{C}$ be a fuzzy covering on $\mathbb{U}$ and $I_O$ an $R_O$-implicator. $N_3^{\mathscr{C}}$, $N_3^{\mathscr{C}_2}$ and $N_3^{\mathscr{C}_{\cap}}$ are three fuzzy operators defined by $I_O$. Then
		\begin{enumerate}
		\item[(1)]
		$N_3^{\mathscr{C}}=N_3^{\mathscr{C}_2},$
		\item[(2)]
		$N_3^{\mathscr{C}}=N_3^{\mathscr{C}_{\cap}}.$
	\end{enumerate}
\end{proposition}
\begin{proof}
	This follows immediately from (2) and (3) of Proposition \ref{p:3.9}.
\end{proof}
Finally, we consider about three operators $N_4^{\mathscr{C}}$, $N_4^{\mathscr{C}_2}$ and $N_4^{\mathscr{C}_{\cap}}$, and they are also equal to each other.
\begin{proposition}
	Let $\mathscr{C}$ be a fuzzy covering on $\mathbb{U}$ and $O$ an overlap function. $N_4^{\mathscr{C}}$, $N_4^{\mathscr{C}_2}$ and $N_4^{\mathscr{C}_{\cap}}$ are three fuzzy neighborhood operators based on $O$, then
	\begin{enumerate}
		\item[(1)]
		$N_4^{\mathscr{C}}=N_4^{\mathscr{C}_2},$
		\item[(2)]
		$N_4^{\mathscr{C}}=N_4^{\mathscr{C}_{\cap}}.$
	\end{enumerate}
\end{proposition}
\begin{proof}
	This follows immediately from Proposition \ref{p:3.2} and (2) and (3) of Proposition \ref{p:3.9}.
\end{proof}
Now, we have seventeen different groups of overlap function based neighborhood operators which are listed in Table ~\ref{tab:a}.
\begin{table}[H]\caption{Overlap function-based fuzzy neighborhood operators based on fuzzy coverings.}\label{tab:a}
	\begin{center}
		\begin{tabular}{l|l|l|l}
			\hline
			Group & Operators  & Group & Operators \\
			\hline
			\rule{0pt}{12pt}
			$A1$.  & $N_{1}^{\mathscr{C}}$, $N_{1}^{\mathscr{C}_{1}}$,  $N_{1}^{\mathscr{C}_{\cap}}$ & $G$. & $N_{1}^{\mathscr{C}_{4}}$\\
			\rule{0pt}{16pt}
			$A2$.& $N_{2}^{\mathscr{C}_{3}}$ & $H1$. & $N_{4}^{\mathscr{C}}$, $N_{4}^{\mathscr{C}_{2}}$, $N_{4}^{\mathscr{C}_{\cap}}$\\
			\rule{0pt}{16pt}
			$A3$. & $N_{1}^{\mathscr{C}_{3}}$ & $H2$. & $N_{2}^{\mathscr{C}_{2}}$\\
			\rule{0pt}{16pt}
			$B$. & $N_{3}^{\mathscr{C}_{1}}$ & $I$. & $N_{2}^{\mathscr{C}_{4}}$\\
			\rule{0pt}{16pt}
			$C$. & $N_{3}^{\mathscr{C}_{3}}$ & $J$. & $N_{3}^{\mathscr{C}_{4}}$\\
			\rule{0pt}{16pt}
			$D$. & $N_{4}^{\mathscr{C}_{3}}$ & $K$. & $N_{4}^{\mathscr{C}_{4}}$ \\
			\rule{0pt}{16pt}
			$E$. & $N_{2}^{\mathscr{C}}$, $N_{2}^{\mathscr{C}_{1}}$ & $L$. & $N_{4}^{\mathscr{C}_{1}}$ \\
			\rule{0pt}{16pt}
			$F1$. & $N_{3}^{\mathscr{C}}$, $N_{3}^{\mathscr{C}_{2}}$, $N_{3}^{\mathscr{C}_{\cap}}$ & $M$. & $N_{2}^{\mathscr{C}_{\cap}}$\\
			\rule{0pt}{16pt}
			$F2$. & $N_{1}^{\mathscr{C}_2}$ & \quad & \quad \\
			\hline
		\end{tabular}
	\end{center}

\end{table}
\subsection{A lattice of overlap function-based fuzzy neighborhood operators based on a finite fuzzy covering}\label{Section 5}
The partial order relations among the different groups of overlap function- based fuzzy neighborhood operators of Table~\ref{tab:a} which are discussed in this subsection. Now we give some assumptions in this part to guarantee all equalities of Table~\ref{tab:a}. We assume $\mathscr{C}$ to be a finite fuzzy covering, $O$ an overlap function, it satisfies (O7), which is used to construct the fuzzy covering $\mathscr{C}_4$ and the operators $N_2^{\mathscr{C}_j}$ and $N_4^{\mathscr{C}_j}$ and $I_O$ its $R_O$-implicator that is used to define covering $\mathscr{C}_3$, and the operators $N_1^{\mathscr{C}_j}$ and $N_3^{\mathscr{C}_j}$ which are also constructed with $I_O$.

We give a definition of partial order relation $\leq$ between overlap unction-based fuzzy neighborhood operators: $N\leq N'$ if only if $\forall x,y\in \mathbb{U}: N(x)(y)\leq N'(x)(y).$ If two operators $N$ and $N'$ are incomparable, we have that either $N\leq N'$ or $N\geq N'$ does not hold.

Firstly, we consider four fuzzy operators $N_1^{\mathscr{C}}$, $N_2^{\mathscr{C}}$, $N_3^{\mathscr{C}}$ and $N_4^{\mathscr{C}}$. If overlap function $O$, which defines $I_O$ and operators, satisfies (O6), can be consider as a left-continuous t-norm. Hence, we have $N_1^{\mathscr{C}}\leq N_2^{\mathscr{C}}$, $N_1^{\mathscr{C}}\leq N_3^{\mathscr{C}}$, $N_2^{\mathscr{C}}\leq N_4^{\mathscr{C}}$ and $N_3^{\mathscr{C}}\leq N_4^{\mathscr{C}}$. If overlap function does not satisfies (O6), we get a different conclusion.
\begin{proposition}\label{p:5.1}
	Let $\mathscr{C}$ be a fuzzy covering on $\mathbb{U}$, $O$ an overlap function and $I_O$ an $R_O$-implicator. $N_2^{\mathscr{C}}$ and $N_4^{\mathscr{C}}$ are two fuzzy operators defined by $I_O$, $N_1^{\mathscr{C}}$ and $N_3^{\mathscr{C}}$ are two fuzzy operators defined by $O$. Then
		\begin{enumerate}
		\item [(1)]
		$N_1^{\mathscr{C}}\leq N_3^{\mathscr{C}},$
		\item [(2)]
		$N_2^{\mathscr{C}}\leq N_4^{\mathscr{C}}.$
	\end{enumerate}
\end{proposition}
\begin{proof}

		\begin{enumerate}
		\item [(1)]
	This follows immediately from the fact that $MD(\mathscr{C},x)\subseteq \mathscr{C}.$
		\item [(2)]
	This follows immediately from the fact that $md(\mathscr{C},x)\subseteq \mathscr{C}.$
	\end{enumerate}
\end{proof}
The following example show that operators $N_1^{\mathscr{C}}$ and $N_2^{\mathscr{C}}$ are incomparable and $N_3^{\mathscr{C}}\leq N_4^{\mathscr{C}}$ also does not hold.
\begin{example}
	Let $\mathbb{U}$ be a set with $\mathbb{U}=\{x,y\}$ and $\mathscr{C}$ a fuzzy covering with $\mathscr{C}=\{K_1, K_2\}$. For $\mathscr{C}$ we have $K_1=\frac{1}{x}+\frac{0.82}{y}$ and $K_2=\frac{0.82}{x}+\frac{1}{y}$. Then, we get $MD(\mathscr{C},x)=md({\mathscr{C},x})=\mathscr{C}=\{K_1, K_2\}$. Now considering overlap function $O_2^V$ and its $R_O$-implicator $I_{O_2^V}$, we get
	\begin{align*}&N_1^{\mathscr{C}}(x)(y)=N_3^{\mathscr{C}}(x)(y)=0.9,\\&N_2^{\mathscr{C}}(x)(y)=N_4^{\mathscr{C}}(x)(y)=0.7048.\end{align*}
	So that, in this condition, we have $N_1^{\mathscr{C}}(x)(y)\geq N_2^{\mathscr{C}}(x)(y)$ and $N_3^{\mathscr{C}}(x)(y)\geq N_4^{\mathscr{C}}(x)(y)$.
\end{example}
Considering other derived coverings, we can get the same conclusions which are like Proposition~\ref{p:5.1}.
\begin{corollary}
	Let $\mathscr{C}$ be a fuzzy covering on $\mathbb{U}$, $O$ an overlap function and $I_{O}$ an $R_O$-implicator. Then
		\begin{enumerate}
		\item [(1)]
		for $\mathscr{C}$ we obtain that $A1 \leq F1$ and $E \leq H1,$
		\item [(2)]
		for $\mathscr{C}_1$ we obtain that $A1 \leq B$ and $E \leq L,$
			\item [(3)]
		for $\mathscr{C}_2$ we obtain that $F2 \leq F1$ and $H2 \leq H1,$
			\item [(4)]
		for $\mathscr{C}_3$ we obtain that $A3 \leq C$ and $A2 \leq D,$
			\item [(5)]
		for $\mathscr{C}_4$ we obtain that $I \leq K$ and $G \leq J,$
			\item [(6)]
		for $\mathscr{C}_{\cap}$ we obtain that $M \leq H1$ and $A1 \leq F1.$
	\end{enumerate}
\end{corollary}
In \cite{d2017fuzzy}, they conclude $A1\leq A2 \leq D$ and $A1\leq C\leq D$ by their groupings. However, in this paper, because of the new group $A3$, we do not compare the group $A1$ with other groups by means of operators $N^{\mathscr{C}_3}_j$.

We can give an example to illustrate $A1$ is incomparable with $A2$ and $D$.
\begin{example}
	Let $\mathbb{U}$ be a set with $\mathbb{U}=\{x,y\}$ and $\mathscr{C}$ a fuzzy covering with $\mathscr{C}=\{K_1, K_2\}$. For covering $\mathscr{C}$ we have $K_1=\frac{1}{x}+\frac{0.905}{y}$ and $K_2=\frac{0.82}{x}+\frac{1}{y}$. Then we have $\mathscr{C}_3=\{N_1^{\mathscr{C}}(x), N_1^{\mathscr{C}}(y)\}$ with $N_1^{\mathscr{C}}(x)=\frac{1}{x}+\frac{0.95}{y}$ and $N_1^{\mathscr{C}}(y)=\frac{0.9}{x}+\frac{1}{y}$ Considering overlap function $O_2^V$ and its $R_O$-implicator $I_{O_2^V}$, we get \begin{align*}&N_1^{\mathscr{C}}(x)(y)=0.95,\\&N_2^{\mathscr{C}_3}(x)(y)=N_4^{\mathscr{C}_3}(x)(y)=0.905.\end{align*}
	Hence, in this case, the group $A1$ is incomparable with $A2$ and $D$.
\end{example}
 Now, we consider the relation between  groups $A1$ and $A3$. By means of the Example \ref{e:4.1}, we can get $N_1^{\mathscr{C}_3}\geq N_1^{\mathscr{C}}$ in that case, actually, they are incomparable.
\begin{proposition}\label{p:5.2}
	Let $\mathscr{C}$ be a fuzzy covering on $\mathbb{U}$, $O$ an overlap function and $I_O$ an $R_O$-implicator. $N_1^{\mathscr{C}}$ and $N_1^{\mathscr{C}_3}$ are two fuzzy operators defined by $I_O$. Then we have $N_1^{\mathscr{C}}$ and $N_1^{\mathscr{C}_3}$ are incomparable.
\end{proposition}
\begin{proof}
	In Example \ref{e:4.1}, we have the inequality $N_1^{\mathscr{C}_3}\geq N_1^{\mathscr{C}}$, now we assume $N_1^{\mathscr{C}}$ and $N_1^{\mathscr{C}_3}$ satisfy $N_1^{\mathscr{C}_3}\geq N_1^{\mathscr{C}}$ for any finite fuzzy coverings and overlap functions which satisfy (O7). So that, we can get that $$N_1^{\mathscr{C}_3}(x)(y)=\underset{K\in \mathscr{C}_3}{\inf}I_O(K(x),K(y))=\underset{z\in \mathbb{U}}{\inf}I_O(N_1^{\mathscr{C}}(z)(x),N_1^{\mathscr{C}}(z)(y))\geq N_1^{\mathscr{C}}(x)(y).$$ Hence, for any $x,y,z\in \mathbb{U}$, we have $$I_O(N_1^{\mathscr{C}}(z)(x),N_1^{\mathscr{C}}(z)(y))\geq N_1^{\mathscr{C}}(x)(y)\Rightarrow O(N_1^{\mathscr{C}}(z)(x),N_1^{\mathscr{C}}(x)(y))\leq N_1^{\mathscr{C}}(z)(y).$$ Therefore, the overlap function based fuzzy neighborhood operator $N_1^{\mathscr{C}}$ is $O$-transitive.
	
	However, by means of Example \ref{ex3.2}, when overlap function is not associative, the operator $N_1^{\mathscr{C}}$ is not $O$-transitive. The above conclusion is contradiction. So that, the assumption is not established, and the operators $N_1^{\mathscr{C}}$ and $N_1^{\mathscr{C}_3}$ are incomparable, i.e., the groups $A1$ and $A3$ are incomparable.
\end{proof}
\begin{remark}
	Notice that, in \cite{d2017fuzzy}, they get the conclusion $A1\leq C$ by their grouping. Considering the overlap function-based fuzzy neighborhood operator $N_1^{\mathscr{C}}$ in group $A1$ and the operator $N_3^{\mathscr{C}_3}$ in group $C$, when the $O$ satisfies (O6), which defines the fuzzy coverings and operators, we can obtain the same conclusion $N_1^{\mathscr{C}}\leq N_3^{\mathscr{C}_3}$. However, when $O$ is not associative, we are no sure of relation between these two operators. Considering the conclusion between $N_1^{\mathscr{C}}$ and $N_1^{\mathscr{C}_3}$, we prove that they are incomparable by means of Proposition \ref{p:5.2}. However, it is still difficult to find an example to verify $N_1^{\mathscr{C}}\geq N_1^{\mathscr{C}_3}$ in some cases. Analogously, the relation between $N_3^{\mathscr{C}_3}$ and $ N_1^{\mathscr{C}}$ is also difficult to be determined. We can find some instances to explain $N_1^{\mathscr{C}_3}=N_3^{\mathscr{C}_3}$ or $N_1^{\mathscr{C}}\leq N_3^{\mathscr{C}_3}$. But it is not easy to prove or give an example that whether the inequality $N_1^{\mathscr{C}}\leq N_3^{\mathscr{C}_3}$ is true in all cases or these two operators are incomparable. With this in mind, we cannot compare the group $A1$ and the group $C$ by default in this paper.
\end{remark}

In our paper, the inequalities $A1\leq F2$ and $L\leq H1$ are also established.
\begin{proposition}
	Let $\mathscr{C}$ be a fuzzy covering on $\mathbb{U}$, $O$ an overlap function and $I_O$ an $R_O$-implicator. $N_1^{\mathscr{C}}$, $N_1^{\mathscr{C}_2}$, $N_4^{\mathscr{C}_1}$ and $N_4^{\mathscr{C}}$ are four fuzzy neighborhood operators. Then $N_1^{\mathscr{C}}\leq N_1^{\mathscr{C}_2}$ and $N_4^{\mathscr{C}_1}\leq N_4^{\mathscr{C}}$.
\end{proposition}
\begin{proof}
	This follows immediately from $\mathscr{C}_2 \subseteq \mathscr{C}$ and $\mathscr{C}_1 \subseteq \mathscr{C}$.
\end{proof}
Furthermore, since, when overlap function satisfies (O6), it can be seen as a t-norm, we can obtain $F1\leq G\leq H1 \leq I$. But when this condition does not hold, we will get some different conclusions.
\begin{example}
		Let $\mathbb{U}$ be a set with $\mathbb{U}=\{x,y\}$ and $\mathscr{C}$ a fuzzy covering with $\mathscr{C}=\{K_1, K_2\}$. For covering $\mathscr{C}$ we have $K_1=\frac{1}{x}+\frac{0.82}{y}$ and $K_2=\frac{0.82}{x}+\frac{1}{y}$. Then, we have $$MD(\mathscr{C},x)=md({\mathscr{C},x})=\mathscr{C}=\{K_1, K_2\}$$ and $$\mathscr{C}_4= MD(\mathscr{C}_4,x)=md({\mathscr{C}_4,x})=\{N_4^{\mathscr{C}}(x), N_4^{\mathscr{C}}(y)\}$$ with $N_4^{\mathscr{C}}(x)=\frac{1}{x}+\frac{0.7048}{y}$ and $N_4^{\mathscr{C}}(y)=\frac{0.7048}{x}+\frac{1}{y}$. Considering overlap function $O_2^V$ and its $R_O$-implicator $I_{O_2^V}$, we have
		\begin{align*}
			&N_3^{\mathscr{C}}(x)(y)=0.9, \notag \\
			&N_1^{\mathscr{C}_4}(x)(y)=0.82, \notag \\
			&N_4^{\mathscr{C}}(x)(y)=0.7048, \notag \\
			&N_2^{\mathscr{C}_4}(x)(y)=0.58388608. \notag
		\end{align*}
	Hence, in this case, the groups $F1$, $G$, $H1$ and $I$ are incomparable.
\end{example}
In \cite{d2017fuzzy}, the relation $D \leq L$ is established, so that when overlap function satisfies (O6), we can also get this conclusion. And now, we consider the other case.
\begin{example}
	Let $\mathbb{U}$ be a set with $\mathbb{U}=\{x,y\}$ and $\mathscr{C}$ a fuzzy covering with $\mathscr{C}=\{K_1, K_2\}$. For $\mathscr{C}$ we have $K_1=\frac{1}{x}+\frac{0.905}{y}$ and $K_2=\frac{0.82}{x}+\frac{1}{y}$. Considering overlap function $O_2^V$ and its $R_O$-implicator $I_{O_2^V}$, and we have $\mathscr{C}_1=\mathscr{C}=\{K_1, K_2\}$ and $\mathscr{C}_3= \{N_4^{\mathscr{C}}(x), N_4^{\mathscr{C}}(y)\}$ with $N_1^{\mathscr{C}}(x)=\frac{1}{x}+\frac{0.95}{y}$ and $N_1^{\mathscr{C}}(y)=\frac{0.9}{x}+\frac{1}{y}$. Then, we have $$N_4^{\mathscr{C}_3}(x)(y)=0.905>N_4^{\mathscr{C}_1}(x)(y)=0.82805.$$
	Hence, in this case, the groups $D$ and $L$ are incomparable.
\end{example}
To the end, we discuss that the relation, $M \leq H2$, holds in all conditions.
\begin{proposition}
	Let $\mathscr{C}$ be a fuzzy covering on $\mathbb{U}$ and $O$ an overlap function. $N_2^{\mathscr{C}_{\cap}}$ and $N_2^{\mathscr{C}_2}$ are two fuzzy neighborhood operators. Then $N_2^{\mathscr{C}_{\cap}} \leq N_2^{\mathscr{C}_2}$.
\end{proposition}
\begin{proof}
	Let $x\in \mathbb{U}$. Then we first consider how to prove $md(\mathscr{C}_{\cap},x)\bigcap
	\mathscr{C}_2 \subseteq md(\mathscr{C}_2,x)$. Let $K$ be a fuzzy set in $md(\mathscr{C}_{\cap},x) \bigcap \mathscr{C}_2$, and we take $K'\in \mathscr{C}_2$ with $K'(x)=K(x)$ and $K'\subseteq K$. Since proposition \ref{p:3.7}, we also have $K'\in \mathscr{C}_{\cap}$. As this relation and $K\in md(\mathscr{C}_{\cap},x)$, we get $K=K'$ and thus, $K\in md(\mathscr{C}_2,x).$
	
	Let $y\in \mathbb{U}$ and we assume that $N_2^{\mathscr{C}_{\cap}}=O(K_1(x),K_1(y))$ for $K_1\in md(\mathscr{C}_\cap,x)$. If $K_1\notin \mathscr{C}_2$, then for each $z\in \mathbb{U}$, we can find a $K_2\in \mathscr{C}_2$ with $K_1(z)=K_2(z)$, and $K_1\subsetneq K_2.$ Considering the definition of $\mathscr{C}_{\cap}$, above assumption means that $K_1=\underset{z\in \mathbb{U}}{\bigcap}K_2$. Since $K_1\in \mathscr{C}_\cap$, this is a contradiction. Hence, $K_1$ will be a set in $\mathscr{C}_2$ , so that, we can obtain $K_1\in md(\mathscr{C}_2,x)$, and thus, we get that $N_2^{\mathscr{C}_{\cap}}(x)(y) \leq N_2^{\mathscr{C}_2}(x)(y)$.
\end{proof}
	
In \cite{d2017fuzzy}, the operator $\mathbb{N}_1^{\mathscr{C}_3}$ is less than others. In our case, there are no other comparable operators, i.e., with respect to the partial order relation $\leq $, the overlap function based operator $N_1^{\mathscr{C}_3}$ is incomparable with other operators which are in the groups $B$, $D$, $E$, $ F1$, $F2$, $G$, $H1$, $H2$, $I$, $J$, $K$, $L$ and $M$. We will give some examples to illustrate they are incomparable as follow.
	
	 \begin{example}
	 Let $\mathbb{U}$ be a set with $\mathbb{U}=\{x,y\}$ and $\mathscr{C}$ a fuzzy covering with $\mathscr{C}=\{K_1, K_2\}$. For covering $\mathscr{C}$ we have $K_1=\frac{1}{x}+\frac{0.905}{y}$ and $K_2=\frac{0.905}{x}+\frac{1}{y}$. Considering overlap function $O_2^V$ and its $R_O$-implicator $I_{O_2^V}$, and we have $\mathscr{C}=\mathscr{C}_1=\mathscr{C}_2=\mathscr{C}_{\cap}$, $\mathscr{C}_3=\{N_1^{\mathscr{C}}(x), N_1^{\mathscr{C}}(y)\}$ with $N_1^{\mathscr{C}}(x)=\frac{1}{x}+\frac{0.95}{y}$ and $N_1^{\mathscr{C}}(y)=\frac{0.95}{x}+\frac{1}{y}$ and $\mathscr{C}_4=\{N_4^{\mathscr{C}}(x), N_4^{\mathscr{C}}(y)\}$ with $N_4^{\mathscr{C}}(x)=\frac{1}{x}+\frac{0.82805}{y}$ and $N_4^{\mathscr{C}}(y)=\frac{0.82805}{x}+\frac{1}{y}$. Then we have \begin{align*}&N_1^{\mathscr{C}_3}(x)(y)\geq 0.974, ~(A3) \\&N_3^{\mathscr{C}_1}(x)(y)=N_3^{\mathscr{C}}(x)(y)=N_1^{\mathscr{C}_2}(x)(y)=0.95, ~(B, F1, F2) \\&N_4^{\mathscr{C}_3}(x)(y)=N_1^{\mathscr{C}_4}(x)(y)=N_3^{\mathscr{C}_4}(x)(y)=0.905, ~(D, G, J) \\&N_4^{\mathscr{C}_1}(x)(y)=N_2^{\mathscr{C}}(x)(y)=N_4^{\mathscr{C}}(x)(y)=N_2^{\mathscr{C}_2}(x)(y)=N_2^{\mathscr{C}_{\cap}}(x)(y)=0.82805, ~(E, H1, H2, L, M)
	\\& N_4^{\mathscr{C}_4}(x)(y)=N_2^{\mathscr{C}_4}(x)(y)=0.715233605. ~(I, K)\end{align*}
	 i.e., the group $A3$ is incomparable with them.
	 \end{example}
 	At end of this subsection, we illustrate the partial order relation $\leq $ among groups which are in Table \ref{tab:a} by Figure ~\ref{fig:a}
 	\begin{figure}[htbp]
 		\centering
 		\includegraphics[width=14cm]{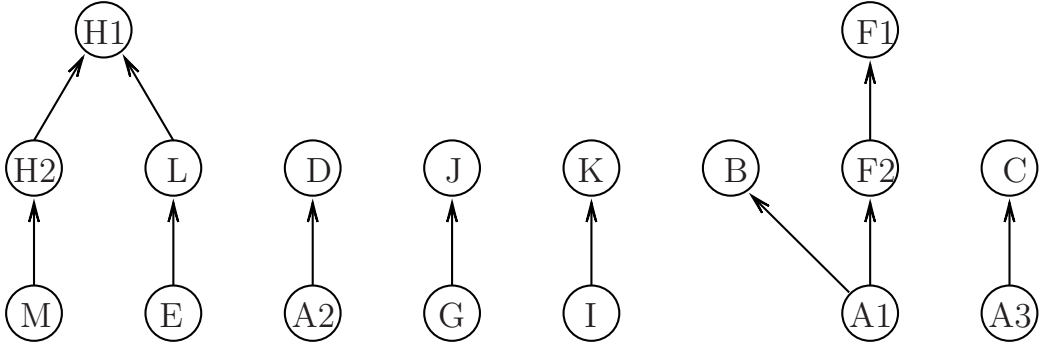}\\
 		\caption{Lattice of the overlap function based fuzzy neighborhood operators from Table ~\ref{tab:a}.}\label{fig:a}
 	\end{figure}
 \subsection{Comparison and analysis}\label{Section 7}
  Considering the partial order relation $\leq $ among fuzzy neighborhood operators, which can be defined as follows: $N\leq N^{'}$ if and only if $\forall x,y \in \mathbb{U}:N(x)(y)\leq N^{'}(x)(y)$. D'eer et al. have compared each fuzzy neighborhood operator with others in \cite{d2017fuzzy}, and divided them into 16 different groups. Meanwhile they have drawn a Hasse diagram with respect to the lattice of groups as shown in Figure~\ref{fig:before}.

 \begin{figure}[htpb]
 \caption{Lattice of the fuzzy neighborhood operators\cite{d2017fuzzy}}\label{fig:before}
 	\centering
 	\includegraphics[width=8cm]{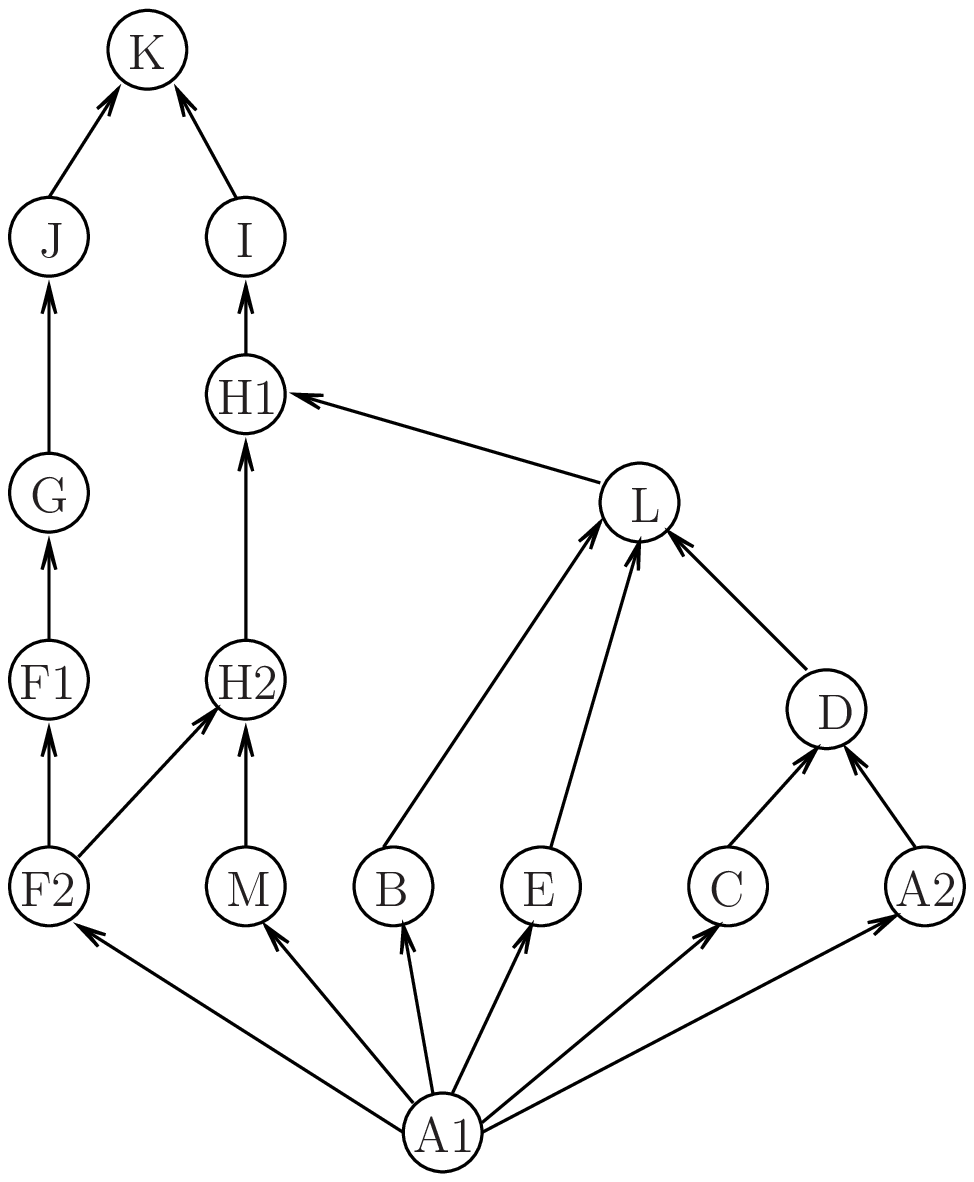}\\
 	
 \begin{table}[H]
 \caption{t-norm based fuzzy neighborhood operators based on fuzzy coverings\cite{d2017fuzzy}}\label{tab:b}
	\begin{center}
		\begin{tabular}{l|l|l|l}
			\hline
			Group & Operators  & Group & Operators \\
			\hline
			\rule{0pt}{12pt}
			$A1$.  & $\mathbb{N}_{1}^{\mathscr{C}}$, $\mathbb{N}_{1}^{\mathscr{C}_{1}}$, $\mathbb{N}_{1}^{\mathbf{C}_{3}}$, $\mathbb{N}_{1}^{\mathscr{C}_{\cap}}$ & $G$. & $\mathbb{N}_{1}^{\mathbf{C}_{4}}$\\
			\rule{0pt}{16pt}
			$A2$.& $\mathbb{N}_{2}^{\mathbf{C}_{3}}$ & $H1$. & $\mathbb{N}_{4}^{\mathscr{C}}$, $\mathbb{N}_{4}^{\mathscr{C}_{2}}$, $\mathbb{N}_{4}^{\mathscr{C}_{\cap}}$\\
			\rule{0pt}{16pt}
			$B$. & $\mathbb{N}_{3}^{\mathscr{C}_{1}}$ &$ H2$. & $\mathbb{N}_{2}^{\mathscr{C}_{2}}$\\
			\rule{0pt}{16pt}
			$C$. & $\mathbb{N}_{3}^{\mathbf{C}_{3}}$ & $I$. & $\mathbb{N}_{2}^{\mathbf{C}_{4}}$\\
			\rule{0pt}{16pt}
			$D$. & $\mathbb{N}_{4}^{\mathbf{C}_{3}}$ & $J$. & $\mathbb{N}_{3}^{\mathbf{C}_{4}}$\\
			\rule{0pt}{16pt}
			$E$. & $\mathbb{N}_{2}^{\mathscr{C}}$, $\mathbb{N}_{2}^{\mathscr{C}_{1}}$ & $K$. & $\mathbb{N}_{4}^{\mathbf{C}_{4}}$ \\
			\rule{0pt}{16pt}
			$F1$. & $\mathbb{N}_{3}^{\mathscr{C}}$, $\mathbb{N}_{3}^{\mathscr{C}_{2}}$, $\mathbb{N}_{3}^{\mathscr{C}_{\cap}}$ & $L$. & $\mathbb{N}_{4}^{\mathscr{C}_{1}}$ \\
			\rule{0pt}{16pt}
			$F2$. & $\mathbb{N}_{1}^{\mathscr{C}_2}$ &$ M.$ & $\mathbb{N}_{2}^{\mathscr{C}_{\cap}}$\\
			\hline
		\end{tabular}
	\end{center}
	
\end{table}
 \end{figure}

Firstly, we compare the groupings between Table \ref{tab:a} and Table \ref{tab:b}. The Table \ref{tab:a} is one group more than the Table \ref{tab:b}, which is $A3$. Besides, the operator $N_{1}^{\mathscr{C}_3}$ in $A3$ corresponds to $\mathbb{N}_{1}^{\mathbf{C}_3}$ in Table \ref{tab:b} which is grouped into $A1$. And then, when we compare Figure \ref{fig:a} with Figure \ref{fig:before}, we find that there are more incomparable relationships among the groups in Figure \ref{fig:a} than the groups in Figure \ref{fig:before}.

The main reason causes these two differences is that general overlap functions do not satisfy the exchange principle. Many proofs, such as the relation between $N_1^{\mathscr{C}}$ and $N_2^{\mathscr{C}_3}$, will obtain different conclusions if the constructing functions satisfy the exchange principle. Moreover, because this principle does not establish, the operators $N_1^{\mathscr{C}_{i}}$ and $N_3^{\mathscr{C}_{i}}$ are not $O$-transitive, which is also a significant property of the partial order relation proofs. Hence the relationships such as $N_1^{\mathscr{C}}$ and $N_1^{\mathscr{C}_3}$, can be explained they are incomparable. Furthermore, the unique boundary condition of overlap functions is an important factor that causes the differences between the $t$-norm-based fuzzy neighborhood operators and overlap function-based ones as well. Although the condition (O6) does not satisfy, some relations among operators are also the same as the conclusions of their corresponding operators in \cite{d2017fuzzy} when $O(x,1)=x$, such as the relation between the operators $N_{4}^{\mathscr{C}_1}$ and $N_4^{\mathscr{C}_3}$. However, when the constructing function satisfies $O(x,1) < x$, these two operators are incomparable. In other words, the conclusions under the condition $O(x,1)=1$ are not general.

By means of the above analyses, we find that both the exchange principle and the boundary condition of overlap functions cause the differences from the conclusions of \cite{d2017fuzzy}. These two conditions bring about overlap function-based fuzzy neighborhood operators that are more complicated, but more actual in applications. For example, in the real world, the exchange principle probably is not satisfied during the image processing, and the operators which are based on overlap functions are more suitable in this case.
\section{Some neighborhood-related fuzzy covering-based rough set models}\label{5c}
In this section,  we assume $\mathscr{C}$ to be a finite fuzzy covering, $O$ an overlap function, it satisfies (O7), which is used to construct the fuzzy covering $\mathscr{C}_4$ and the operators $N_2^{\mathscr{C}_j}$ and $N_4^{\mathscr{C}_j}$, and $I_O$ its $R_O$-implicator that is used to define covering $\mathscr{C}_3$, and the operators $N_1^{\mathscr{C}_j}$ and $N_3^{\mathscr{C}_j}$, which are overlap function-based neighborhood-related fuzzy covering-based rough set (ONRFRS) models and $t$-norm-based neighborhood-related fuzzy covering-based rough set (TNRFRS) models, respectively. By introducing the above concepts, two kinds of fuzzy covering-based rough set models with different bases will be defined. In addition, divided into two groups according to the construction basis of fuzzy neighborhood operators and $R$-implicators, we sorted out the relationship among different pairs of approximation operators.

Considering some properties of fuzzy covering-based rough set models, which are studied by Ma in \cite{ma2016two}, the approximation spaces of two diverse groups of neighborhood operators are established are follow. For convenience, the member of any four original overlap function-based neighborhood operators and their derived operators can be written uniformly as $N_i^{\mathscr{C}_{j}}$, where $i=1,2,3,4$ and $j=0,1,2,3,4,5$, noteworthy, $N_i^{\mathscr{C}}$ and $N_i^{\mathscr{C}_{\cap}}$ are equivalently denoted as $N_i^{\mathscr{C}_{0}}$ and $N_i^{\mathscr{C}_{5}}$, respectively. Analogously, the member of $t$-norm-based neighborhood operators is also written uniformly are $\mathbb{N}_i^{\mathbf{C}_j}$, where $i=1,2,3,4$ and $j=0,1,2,3,4,5$. In summary, we define the following two groups of fuzzy rough set models.

\begin{definition}\label{D5.1}
Let $N_i^{\mathscr{C}_{j}}$ be one of overlap function-based fuzzy neighborhood operators, $\mathbb{N}_i^{\mathbf{C}_j}$ one of $t$-norm based fuzzy neighborhood operators. For each $X\in \mathscr{F}(\mathbb{U})$, the overlap function based lower approximation and upper approximation of $X$, $(\underline{\mathfrak{C}}_{i,j}(X)$, $\overline{\mathfrak{C}}_{i,j}(X))$, and the $t$-norm based lower approximation and upper approximation of $X$, $(\underline{\mathcal{C}}_{i,j}(X)$, $\overline{\mathcal{C}}_{i,j}(X))$, are respectively defined as follow:
\begin{align*}
&\underline{\mathfrak{C}}_{i,j}(X)(x)=\wedge_{y\in \mathbb{U}}[(1-N_i^{\mathscr{C}_{j}}(x)(y))\vee X(y)], x\in \mathbb{U},
\\&\overline{\mathfrak{C}}_{i,j}(X)(x)=\vee_{y\in \mathbb{U}}[N_i^{\mathscr{C}_{j}}(x)(y)\wedge X(y)], x\in \mathbb{U},
\\&\underline{\mathcal{C}}_{i,j}(X)(x)=\wedge_{y\in \mathbb{U}}[(1-\mathbb{N}_i^{\mathbf{C}_{j}}(x)(y))\vee X(y)], x\in \mathbb{U},
\\&\overline{\mathcal{C}}_{i,j}(X)(x)=\vee_{y\in \mathbb{U}}[\mathbb{N}_i^{\mathbf{C}_{j}}(x)(y)\wedge X(y)], x\in \mathbb{U}.
\end{align*}

Where $i=1,2,3,4$ and $j=0,1,2,3,4,5$.
\begin{example}
Let $\mathbb{U}=\{x_1,x_2,x_3,x_4,x_5,x_6\}$ and $\mathscr{C}$ a fuzzy covering with $\mathscr{C}=\{K_1,K_2,K_3,K_4\}$, where \begin{align*}&K_1=\frac{1}{x_1}+\frac{0.4}{x_2}+\frac{0.7}{x_3}+\frac{0.3}{x_4}+\frac{0.6}{x_5}+\frac{1}{x_6},\\&K_2=\frac{0.4}{x_1}+\frac{1}{x_2}+\frac{0.1}{x_3}+\frac{1}{x_4}+\frac{1}{x_5}+\frac{0.5}{x_6},
\\&K_3=\frac{0.8}{x_1}+\frac{0.2}{x_2}+\frac{1}{x_3}+\frac{0.5}{x_4}+\frac{1}{x_5}+\frac{1}{x_6},\\&K_4=\frac{0.1}{x_1}+\frac{0.7}{x_2}+\frac{0.4}{x_3}+\frac{1}{x_4}+\frac{1}{x_5}+\frac{1}{x_6}.\end{align*}
Considering the overlap function $O_2^V$ and its implicator $I_{O_2^V}$ which are used to establish the derived coverings and operators $N_i^{\mathscr{C}_j}$. Let $i=4$ and $j=0$. Then the neighborhood operators, $N_4^{\mathscr{C}_0}(x_t)$, of $x_t$ $(t=1,2,3,4,5,6)$ can be listed as below:
\begin{align*}&N_4^{\mathscr{C}_0}(x_1)=\frac{1}{x_1}+\frac{0.4}{x_2}+\frac{0.68}{x_3}+\frac{0.5}{x_4}+\frac{0.68}{x_5}+\frac{1}{x_6},\\&N_4^{\mathscr{C}_0}(x_2)=\frac{0.4}{x_1}+\frac{1}{x_2}+\frac{0.4}{x_3}+\frac{1}{x_4}+\frac{1}{x_5}+\frac{0.58}{x_6},
\\&N_4^{\mathscr{C}_0}(x_3)=\frac{0.68}{x_1}+\frac{0.4}{x_2}+\frac{1}{x_3}+\frac{0.5}{x_4}+\frac{1}{x_5}+\frac{1}{x_6},
\\&N_4^{\mathscr{C}_0}(x_4)=\frac{0.5}{x_1}+\frac{1}{x_2}+\frac{0.5}{x_3}+\frac{1}{x_4}+\frac{1}{x_5}+\frac{1}{x_6},
\\&N_4^{\mathscr{C}_0}(x_5)=\frac{0.68}{x_1}+\frac{1}{x_2}+\frac{1}{x_3}+\frac{1}{x_4}+\frac{1}{x_5}+\frac{1}{x_6},
\\&N_4^{\mathscr{C}_0}(x_6)=\frac{1}{x_1}+\frac{0.58}{x_2}+\frac{1}{x_3}+\frac{1}{x_4}+\frac{1}{x_5}+\frac{1}{x_6}.\end{align*}
For $X(x_t)=\frac{0.3}{x_1}+\frac{0.4}{x_2}+\frac{0.7}{x_3}+\frac{0.5}{x_4}+\frac{0.6}{x_5}+\frac{0.1}{x_6},$ we have that \begin{align*}&\underline{\mathfrak{C}}_{4,0}(X)=\frac{0.1}{x_1}+\frac{0.4}{x_2}+\frac{0.1}{x_3}+\frac{0.1}{x_4}+\frac{0.1}{x_5}+\frac{0.1}{x_6},\\&\overline{\mathfrak{C}}_{4,0}(X)=\frac{0.68}{x_1}+\frac{0.6}{x_2}+\frac{0.7}{x_3}+\frac{0.6}{x_4}+\frac{0.7}{x_5}+\frac{0.7}{x_6}.\end{align*}
\end{example}

By means of the definitions of approximation operators and the groupings of neighborhood operators from the above, we can further group the newly defined forty-eight fuzzy covering-based approximation operators. Obviously, we can divide them into seventeen pairs of overlap function-based approximation operators and sixteen pairs of $t$-norm-based ones, and the classifications are illustrated in Table~\ref{t3} and Table~\ref{t4}.
\begin{table}[H]
	\begin{center}
\caption{{Overlap function-based neighborhood-related fuzzy covering-based rough set models.}}\label{t3}

\resizebox{.95\columnwidth}{!}{
		\begin{tabular}{l|l|l|l}
			\hline
			Classification & Operators  & Classification & Operators \\
			\hline
			\rule{0pt}{18pt}
			$(\underline{\mathfrak{C}}_{A1}(X)$, $\overline{\mathfrak{C}}_{A1}(X))$  & \makecell[l]{$(\underline{\mathfrak{C}}_{1,0}(X)$, $\overline{\mathfrak{C}}_{1,0}(X))$, $(\underline{\mathfrak{C}}_{1,1}(X)$, $\overline{\mathfrak{C}}_{1,1}(X))$,\\  $(\underline{\mathfrak{C}}_{1,5}(X)$, $\overline{\mathfrak{C}}_{1,5}(X))$ } & $(\underline{\mathfrak{C}}_{G}(X)$, $\overline{\mathfrak{C}}_{G}(X))$ & $(\underline{\mathfrak{C}}_{1,4}(X)$, $\overline{\mathfrak{C}}_{1,4}(X))$\\
			\rule{0pt}{18pt}
			$(\underline{\mathfrak{C}}_{A2}(X)$, $\overline{\mathfrak{C}}_{A2}(X))$& $(\underline{\mathfrak{C}}_{2,3}(X)$, $\overline{\mathfrak{C}}_{2,3}(X))$ & $(\underline{\mathfrak{C}}_{H1}(X)$, $\overline{\mathfrak{C}}_{H1}(X))$ & \makecell[l]{$(\underline{\mathfrak{C}}_{4,0}(X)$, $\overline{\mathfrak{C}}_{4,0}(X))$, $(\underline{\mathfrak{C}}_{4,2}(X)$, $\overline{\mathfrak{C}}_{4,2}(X))$, \\ $(\underline{\mathfrak{C}}_{4,5}(X)$, $\overline{\mathfrak{C}}_{4,5}(X))$} \\
			\rule{0pt}{18pt}
			$(\underline{\mathfrak{C}}_{A3}(X)$, $\overline{\mathfrak{C}}_{A3}(X))$ & $(\underline{\mathfrak{C}}_{1,3}(X)$, $\overline{\mathfrak{C}}_{1,3}(X))$ & $(\underline{\mathfrak{C}}_{H2}(X)$, $\overline{\mathfrak{C}}_{H2}(X))$ & $(\underline{\mathfrak{C}}_{2,2}(X)$, $\overline{\mathfrak{C}}_{2,2}(X))$\\
			\rule{0pt}{18pt}
			$(\underline{\mathfrak{C}}_{B}(X)$, $\overline{\mathfrak{C}}_{B}(X))$ & $(\underline{\mathfrak{C}}_{3,1}(X)$, $\overline{\mathfrak{C}}_{3,1}(X))$ & $(\underline{\mathfrak{C}}_{I}(X)$, $\overline{\mathfrak{C}}_{I}(X))$ & $(\underline{\mathfrak{C}}_{2,4}(X)$, $\overline{\mathfrak{C}}_{2,4}(X))$\\
			\rule{0pt}{18pt}
			$(\underline{\mathfrak{C}}_{C}(X)$, $\overline{\mathfrak{C}}_{C}(X))$ & $(\underline{\mathfrak{C}}_{3,3}(X)$, $\overline{\mathfrak{C}}_{3,3}(X))$ & $(\underline{\mathfrak{C}}_{J}(X)$, $\overline{\mathfrak{C}}_{J}(X))$ & $(\underline{\mathfrak{C}}_{3,4}(X)$, $\overline{\mathfrak{C}}_{3,4}(X))$\\
			\rule{0pt}{18pt}
			$(\underline{\mathfrak{C}}_{D}(X)$, $\overline{\mathfrak{C}}_{D}(X))$ & $(\underline{\mathfrak{C}}_{4,3}(X)$, $\overline{\mathfrak{C}}_{4,3}(X))$ & $(\underline{\mathfrak{C}}_{K}(X)$, $\overline{\mathfrak{C}}_{K}(X))$ & $(\underline{\mathfrak{C}}_{4,4}(X)$, $\overline{\mathfrak{C}}_{4,4}(X))$ \\
			\rule{0pt}{18pt}
			$(\underline{\mathfrak{C}}_{E}(X)$, $\overline{\mathfrak{C}}_{E}(X))$ & $(\underline{\mathfrak{C}}_{2,0}(X)$, $\overline{\mathfrak{C}}_{2,0}(X))$, $(\underline{\mathfrak{C}}_{2,1}(X)$, $\overline{\mathfrak{C}}_{2,1}(X))$ & $(\underline{\mathfrak{C}}_{L}(X)$, $\overline{\mathfrak{C}}_{L}(X))$ & $(\underline{\mathfrak{C}}_{4,1}(X)$, $\overline{\mathfrak{C}}_{4,1}(X))$ \\
			\rule{0pt}{18pt}
			$(\underline{\mathfrak{C}}_{F1}(X)$, $\overline{\mathfrak{C}}_{F1}(X))$ & \makecell[l]{$(\underline{\mathfrak{C}}_{3,0}(X)$, $\overline{\mathfrak{C}}_{3,0}(X))$, $(\underline{\mathfrak{C}}_{3,2}(X)$, $\overline{\mathfrak{C}}_{3,2}(X))$, \\ $(\underline{\mathfrak{C}}_{3,5}(X)$, $\overline{\mathfrak{C}}_{3,5}(X))$}  & $(\underline{\mathfrak{C}}_{M}(X)$, $\overline{\mathfrak{C}}_{M}(X))$ & $(\underline{\mathfrak{C}}_{2,5}(X)$, $\overline{\mathfrak{C}}_{2,5}(X))$\\
			\rule{0pt}{18pt}
			$(\underline{\mathfrak{C}}_{F2}(X)$, $\overline{\mathfrak{C}}_{F2}(X))$ & $(\underline{\mathfrak{C}}_{1,2}(X)$, $\overline{\mathfrak{C}}_{1,2}(X))$ & \quad & \quad \\
			\hline
		\end{tabular}

}
	\end{center}	
\end{table}

\begin{table}[H]
\caption{$t$ norm-based neighborhood-related fuzzy covering-based rough set models.}\label{t4}

	\begin{center}
\resizebox{.95\columnwidth}{!}{
		\begin{tabular}{l|l|l|l}
			\hline
			Classification & Operators  & Classification & Operators \\
			\hline
			\rule{0pt}{18pt}
			$(\underline{\mathcal{C}}_{A1}(X)$, $\overline{\mathcal{C}}_{A1}(X))$  & \makecell[l]{$(\underline{\mathcal{C}}_{1,0}(X)$, $\overline{\mathcal{C}}_{1,0}(X))$, $(\underline{\mathcal{C}}_{1,1}(X)$, $\overline{\mathcal{C}}_{1,1}(X))$,\\  $(\underline{\mathcal{C}}_{1,3}(X)$, $\overline{\mathcal{C}}_{1,3}(X))$, $(\underline{\mathcal{C}}_{1,5}(X)$, $\overline{\mathcal{C}}_{1,5}(X))$ } & $(\underline{\mathcal{C}}_{G}(X)$, $\overline{\mathcal{C}}_{G}(X))$ & $(\underline{\mathcal{C}}_{1,4}(X)$, $\overline{\mathcal{C}}_{1,4}(X))$\\
			\rule{0pt}{18pt}
			$(\underline{\mathcal{C}}_{A2}(X)$, $\overline{\mathcal{C}}_{A2}(X))$& $(\underline{\mathcal{C}}_{2,3}(X)$, $\overline{\mathcal{C}}_{2,3}(X))$ & $(\underline{\mathcal{C}}_{H1}(X)$, $\overline{\mathcal{C}}_{H1}(X))$ & \makecell[l]{$(\underline{\mathcal{C}}_{4,0}(X)$, $\overline{\mathcal{C}}_{4,0}(X))$, $(\underline{\mathcal{C}}_{4,2}(X)$, $\overline{\mathcal{C}}_{4,2}(X))$, \\ $(\underline{\mathcal{C}}_{4,5}(X)$, $\overline{\mathcal{C}}_{4,5}(X))$} \\
			\rule{0pt}{18pt}
			$(\underline{\mathcal{C}}_{B}(X)$, $\overline{\mathcal{C}}_{B}(X))$ & $(\underline{\mathcal{C}}_{3,1}(X)$, $\overline{\mathcal{C}}_{3,1}(X))$ & $(\underline{\mathcal{C}}_{H2}(X)$, $\overline{\mathcal{C}}_{H2}(X))$ & $(\underline{\mathcal{C}}_{2,2}(X)$, $\overline{\mathcal{C}}_{2,2}(X))$\\
			\rule{0pt}{18pt}
			 $(\underline{\mathcal{C}}_{C}(X)$, $\overline{\mathcal{C}}_{C}(X))$ & $(\underline{\mathcal{C}}_{3,3}(X)$, $\overline{\mathcal{C}}_{3,3}(X))$& $(\underline{\mathcal{C}}_{I}(X)$, $\overline{\mathcal{C}}_{I}(X))$ & $(\underline{\mathcal{C}}_{2,4}(X)$, $\overline{\mathcal{C}}_{2,4}(X))$\\
			\rule{0pt}{18pt}
			$(\underline{\mathcal{C}}_{D}(X)$, $\overline{\mathcal{C}}_{D}(X))$ & $(\underline{\mathcal{C}}_{4,3}(X)$, $\overline{\mathcal{C}}_{4,3}(X))$ & $(\underline{\mathcal{C}}_{J}(X)$, $\overline{\mathcal{C}}_{J}(X))$ & $(\underline{\mathcal{C}}_{3,4}(X)$, $\overline{\mathcal{C}}_{3,4}(X))$\\
			\rule{0pt}{18pt}
			$(\underline{\mathcal{C}}_{E}(X)$, $\overline{\mathcal{C}}_{E}(X))$ & $(\underline{\mathcal{C}}_{2,0}(X)$, $\overline{\mathcal{C}}_{2,0}(X))$, $(\underline{\mathcal{C}}_{2,1}(X)$, $\overline{\mathfrak{C}}_{2,1}(X))$ & $(\underline{\mathcal{C}}_{K}(X)$, $\overline{\mathcal{C}}_{K}(X))$ & $(\underline{\mathcal{C}}_{4,4}(X)$, $\overline{\mathcal{C}}_{4,4}(X))$ \\
			\rule{0pt}{18pt}
			$(\underline{\mathcal{C}}_{F1}(X)$, $\overline{\mathcal{C}}_{F1}(X))$ & \makecell[l]{$(\underline{\mathcal{C}}_{3,0}(X)$, $\overline{\mathcal{C}}_{3,0}(X))$, $(\underline{\mathcal{C}}_{3,2}(X)$, $\overline{\mathcal{C}}_{3,2}(X))$, \\ $(\underline{\mathcal{C}}_{3,5}(X)$, $\overline{\mathcal{C}}_{3,5}(X))$} & $(\underline{\mathcal{C}}_{L}(X)$, $\overline{\mathcal{C}}_{L}(X))$ & $(\underline{\mathcal{C}}_{4,1}(X)$, $\overline{\mathcal{C}}_{4,1}(X))$ \\
			\rule{0pt}{18pt}
			 $(\underline{\mathcal{C}}_{F2}(X)$, $\overline{\mathcal{C}}_{F2}(X))$ & $(\underline{\mathcal{C}}_{1,2}(X)$, $\overline{\mathcal{C}}_{1,2}(X))$ & $(\underline{\mathcal{C}}_{M}(X)$, $\overline{\mathcal{C}}_{M}(X))$ & $(\underline{\mathcal{C}}_{2,5}(X)$, $\overline{\mathcal{C}}_{2,5}(X))$\\\hline
		\end{tabular}}
	\end{center}	
\end{table}
\end{definition}

\begin{remark}\label{r5.1}
According to definitions of $(\underline{\mathfrak{C}}_{i,j}(X),\overline{\mathfrak{C}}_{i,j}(X))$ and $(\underline{\mathcal{C}}_{i,j}(X),\overline{\mathcal{C}}_{i,j}(X))$, if the neighborhood operators $N_a^{\mathscr{C}_b}$ and $N_c^{\mathscr{C}_d}$, which are used to construct approximation spaces, satisfy the partial order relation $``\leq"$ discussed in subsection \ref{Section 5}, that is, $N_a^{\mathscr{C}_b}\leq N_c^{\mathscr{C}_d}$, we have $\underline{\mathfrak{C}}_{a,b}(X)\supseteq\underline{\mathfrak{C}}_{c,d}(X)$ and $\overline{\mathfrak{C}}_{a,b}(X)\subseteq\overline{\mathfrak{C}}_{c,d}(X)$. Analogously, binary pair, $(\underline{\mathcal{C}}_{i,j}(X),\overline{\mathcal{C}}_{i,j}(X))$, has the same property.
\end{remark}

\begin{definition}
Let $(\underline{\mathfrak{C}}_{i_1,j_1}(X),\overline{\mathfrak{C}}_{i_1,j_1}(X))$ and $(\underline{\mathfrak{C}}_{i_2,j_2}(X),\overline{\mathfrak{C}}_{i_2,j_2}(X))$ be two of ONRFRS models. We call $$(\underline{\mathfrak{C}}_{i_1,j_1}(X),\overline{\mathfrak{C}}_{i_1,j_1}(X))\prec(\underline{\mathfrak{C}}_{i_2,j_2}(X),\overline{\mathfrak{C}}_{i_2,j_2}(X)),$$ If $(\underline{\mathfrak{C}}_{i_1,j_1}(X)\supseteq(\underline{\mathfrak{C}}_{i_2,j_2}(X)$, and $\overline{\mathfrak{C}}_{i_1,j_1}(X))\subseteq\overline{\mathfrak{C}}_{i_2,j_2}(X))$.
\end{definition}

\begin{remark}\label{rere}
Obviously, $\big((\underline{\mathfrak{C}}_{i,j}(X),\overline{\mathfrak{C}}_{i,j}(X)),\prec)\big)$ is a partially order set and we can also similarly define the partially order set $\big((\underline{\mathcal{C}}_{i,j}(X),\overline{\mathcal{C}}_{i,j}(X)),\prec^{*})\big)$. Combining Remark \ref{r5.1} with the relationship among the neighborhood operators we discussed in the above sections, we can conclude that the relationships among these binary pairs are consistent with their constituent elements. In other words, the Hasse diagrams of two types of approximation operators are same as their corresponding fuzzy neighborhood operators.
\end{remark}
 Some properties of these two types of neighborhood-related fuzzy covering-based rough set models are illustrated as follows.
\begin{proposition}\label{P1}Let $N_i^{\mathscr{C}_{j}}$ be one of overlap function-based fuzzy neighborhood operators and $(\underline{\mathfrak{C}}_{i,j}(X),\overline{\mathfrak{C}}_{i,j}(X))$ a pair of approximation operators. For each $X,Y\in \mathscr{F}(\mathbb{U})$, we obtain some conclusions as below.
 \begin{flushleft}
		(1) $\underline{\mathfrak{C}}_{i,j}(X^c)=(\overline{\mathfrak{C}}_{i,j}(X))^c$, $\overline{\mathfrak{C}}_{i,j}(X^c)=(\underline{\mathfrak{C}}_{i,j}(X))^c$, where $X^c(x)=1-X(x), x\in \mathbb{U}$ ;\\
		(2) $\underline{\mathfrak{C}}_{i,j}(\mathbb{U})=\mathbb{U}$, $\overline{\mathfrak{C}}_{i,j}(\emptyset)=\emptyset$;\\
		(3) $\underline{\mathfrak{C}}_{i,j}(X\bigcap Y)=\underline{\mathfrak{C}}_{i,j}(X)\bigcap\underline{\mathfrak{C}}_{i,j}( Y)$, $\overline{\mathfrak{C}}_{i,j}(X\bigcup Y)=\overline{\mathfrak{C}}_{i,j}(X)\bigcup\overline{\mathfrak{C}}_{i,j}( Y)$;\\
		(4) If $X\subseteq Y$, then $\underline{\mathfrak{C}}_{i,j}(X)\subseteq\underline{\mathfrak{C}}_{i,j}( Y)$, $\overline{\mathfrak{C}}_{i,j}(X)\subseteq\overline{\mathfrak{C}}_{i,j}( Y)$;\\
		(5) $\underline{\mathfrak{C}}_{i,j}(X\bigcup Y)\supseteq\underline{\mathfrak{C}}_{i,j}(X)\bigcup\underline{\mathfrak{C}}_{i,j}( Y)$, $\overline{\mathfrak{C}}_{i,j}(X\bigcap Y)\subseteq\overline{\mathfrak{C}}_{i,j}(X)\bigcap\overline{\mathfrak{C}}_{i,j}( Y)$;\\
(6) $X\subseteq Y$ or $X\supseteq Y\Longrightarrow \underline{\mathfrak{C}}_{i,j}(X\bigcup Y)=\underline{\mathfrak{C}}_{i,j}(X)\bigcup\underline{\mathfrak{C}}_{i,j}( Y)$ and $\overline{\mathfrak{C}}_{i,j}(X\bigcap Y)=\overline{\mathfrak{C}}_{i,j}(X)\bigcap\overline{\mathfrak{C}}_{i,j}( Y)$ ;\\
(7) If $1-N_i^{\mathscr{C}_{j}}(x)(x)\leq X(x)\leq N_i^{\mathscr{C}_{j}}(x)(x)$ for all $x\in \mathbb{U}$, then $\underline{\mathfrak{C}}_{i,j}\big(\underline{\mathfrak{C}}_{i,j}(X)\big)\subseteq\underline{\mathfrak{C}}_{i,j}(X)\subseteq X\subseteq\overline{\mathfrak{C}}_{i,j}(X)\subseteq\overline{\mathfrak{C}}_{i,j}\big(\overline{\mathfrak{C}}_{i,j}(X)\big)$.

	\end{flushleft}
Where $i=1,2,3,4$ and $j=0,1,2,3,4,5$
\begin{proof}
(1) Considering that $$\overline{\mathfrak{C}}_{i,j}(X^c)(x)=\vee_{y\in \mathbb{U}}[N_i^{\mathscr{C}_{j}}(x)(y)\wedge X^c(y)]=1-\wedge_{y\in \mathbb{U}}[(1-N_i^{\mathscr{C}_{j}}(x)(y))\vee X(y)]=1-\underline{\mathfrak{C}}_{i,j}(X)(x)=(\underline{\mathfrak{C}}_{i,j}(X))^c(x).$$ Hence, we have $\overline{\mathfrak{C}}_{i,j}(X^c)=(\underline{\mathfrak{C}}_{i,j}(X))^c$. Correspondingly, replacing $X^c$ with $X$, the equation $\overline{\mathfrak{C}}_{i,j}(X^c)=(\underline{\mathfrak{C}}_{i,j}(X))^c$ can be obtained.

(2) For all $x\in \mathbb{U}$, we have $\mathbb{U}(x)=1$ and $\emptyset(x)=0$, thus, \begin{align*}&\underline{\mathfrak{C}}_{i,j}(\mathbb{U})(x)=\wedge_{y\in \mathbb{U}}[(1-N_i^{\mathscr{C}_{j}}(x)(y))\vee \mathbb{U}(y)]=1=\mathbb{U}(x),\\&\overline{\mathfrak{C}}_{i,j}(\emptyset)(x)=\vee_{y\in \mathbb{U}}[N_i^{\mathscr{C}_{j}}(x)(y)\wedge \emptyset(y)]=0=\emptyset(x).\end{align*}That is, we have $\underline{\mathfrak{C}}_{i,j}(\mathbb{U})=\mathbb{U}$ and $\overline{\mathfrak{C}}_{i,j}(\emptyset)=\emptyset$.

(3) We can verify the conclusion, $\underline{\mathfrak{C}}_{i,j}(X\bigcap Y)=\underline{\mathfrak{C}}_{i,j}(X)\bigcap\underline{\mathfrak{C}}_{i,j}( Y)$, by following operations,
$$\begin{aligned}\big(\underline{\mathfrak{C}}_{i,j}(X\bigcap Y)\big)(x)&=\wedge_{y\in \mathbb{U}}[(1-N_i^{\mathscr{C}_{j}}(x)(y))\vee (X\bigcap Y)(y)]\\&=\wedge_{y\in \mathbb{U}}[\big((1-N_i^{\mathscr{C}_{j}}(x)(y))\vee X(y)\big)\wedge(1-N_i^{\mathscr{C}_{j}}(x)(y))\vee Y(y)\big)]\\&=\big(\wedge_{y\in \mathbb{U}}[(1-N_i^{\mathscr{C}_{j}}(x)(y))\vee X(y)]\big)\wedge\big(\wedge_{y\in \mathbb{U}}[(1-N_i^{\mathscr{C}_{j}}(x)(y))\vee Y(y)]\big)\\&=\big(\underline{\mathfrak{C}}_{i,j}(X)\bigcap\underline{\mathfrak{C}}_{i,j}( Y)\big)(x),\end{aligned}$$
Similarly, $\overline{\mathfrak{C}}_{i,j}(X\bigcup Y)=\overline{\mathfrak{C}}_{i,j}(X)\bigcup\overline{\mathfrak{C}}_{i,j}( Y)$ can also be obtained.

(4) According to the monotonicity of $``\vee"$ and $``\wedge"$, and the definition of $(\underline{\mathfrak{C}}_{i,j}(X),\overline{\mathfrak{C}}_{i,j}(X))$, the above conclusions are obvious.

(5) Since $X\subseteq X\bigcup Y$, $Y\subseteq X\bigcup Y$, $X\bigcap Y\subseteq X$ and $X\bigcap Y\subseteq Y$, combining with conclusion (4), we have the following relations, $$\underline{\mathfrak{C}}_{i,j}(X)\subseteq\underline{\mathfrak{C}}_{i,j}(X\bigcup Y),~\underline{\mathfrak{C}}_{i,j}(Y)\subseteq\underline{\mathfrak{C}}_{i,j}(X\bigcup Y),~\overline{\mathfrak{C}}_{i,j}(X\bigcap Y)\subseteq\overline{\mathfrak{C}}_{i,j}(X),~\overline{\mathfrak{C}}_{i,j}(X\bigcap Y)\subseteq\overline{\mathfrak{C}}_{i,j}( Y).$$
Hence, we have $\underline{\mathfrak{C}}_{i,j}(X\bigcup Y)\supseteq\underline{\mathfrak{C}}_{i,j}(X)\bigcup\underline{\mathfrak{C}}_{i,j}( Y)$, $\overline{\mathfrak{C}}_{i,j}(X\bigcap Y)\subseteq\overline{\mathfrak{C}}_{i,j}(X)\bigcap\overline{\mathfrak{C}}_{i,j}( Y)$.

(6) The condition $``X\subseteq Y"$ is similar to the condition $``X\supseteq Y"$, we just need to prove one of them, and we now verify the conclusion under the condition $``X\subseteq Y"$.
$$\begin{aligned}\big(\underline{\mathfrak{C}}_{i,j}(X\bigcup Y)\big)(x)&=\wedge_{y\in \mathbb{U}}[(1-N_i^{\mathscr{C}_{j}}(x)(y))\vee (X\bigcup Y)(y)]\\&=\wedge_{y\in \mathbb{U}}[\big((1-N_i^{\mathscr{C}_{j}}(x)(y))\vee X(y)\vee Y(y)\big)]\\&=\wedge_{y\in \mathbb{U}}[\big((1-N_i^{\mathscr{C}_{j}}(x)(y))\vee Y(y)\big)]\\&=\underline{\mathfrak{C}}_{i,j}( Y)(x).\end{aligned}$$
According to the condition $``X\subseteq Y"$ and the statement (4), we have the following operations,
$$\begin{aligned}\big(\underline{\mathfrak{C}}_{i,j}(X)\bigcup\underline{\mathfrak{C}}_{i,j}( Y)\big)(x)&=\big(\wedge_{y\in \mathbb{U}}[(1-N_i^{\mathscr{C}_{j}}(x)(y))\vee X(y)]\big)\vee\big(\wedge_{y\in \mathbb{U}}[(1-N_i^{\mathscr{C}_{j}}(x)(y))\vee Y(y)]\big)\\&=\big(\wedge_{y\in \mathbb{U}}[(1-N_i^{\mathscr{C}_{j}}(x)(y))\vee Y(y)]\big)\\&=\underline{\mathfrak{C}}_{i,j}( Y)(x).\end{aligned}$$ Thus, we have $\underline{\mathfrak{C}}_{i,j}(X\bigcup Y)=\underline{\mathfrak{C}}_{i,j}(X)\bigcup\underline{\mathfrak{C}}_{i,j}( Y)$. Analogously, we can prove $\overline{\mathfrak{C}}_{i,j}(X\bigcap Y)=\overline{\mathfrak{C}}_{i,j}(X)\bigcap\overline{\mathfrak{C}}_{i,j}( Y)$ as well.

(7) When we have $1-N_i^{\mathscr{C}_{j}}(x)(x)\leq X(x)\leq N_i^{\mathscr{C}_{j}}(x)(x)$ for each $x\in \mathbb{U}$, via the definition of $(\underline{\mathfrak{C}}_{i,j}(X),\overline{\mathfrak{C}}_{i,j}(X))$, we can conclude that \begin{align*}&X(x)=N_i^{\mathscr{C}_{j}}(x)(y)\wedge X(y)\leq \vee_{y\in \mathbb{U}}[N_i^{\mathscr{C}_{j}}(x)(y)\wedge X(y)]=\overline{\mathfrak{C}}_{i,j}(X)(x),\\&\underline{\mathfrak{C}}_{i,j}(X)(x)=\wedge_{y\in \mathbb{U}}[(1-N_i^{\mathscr{C}_{j}}(x)(y))\vee X(y)]\leq (1-N_i^{\mathscr{C}_{j}}(x)(y))\vee X(y)=X(x).\end{align*}
Thus, $\underline{\mathfrak{C}}_{i,j}(X)\subseteq X\subseteq\overline{\mathfrak{C}}_{i,j}(X)$. Considering the statement (4), we have  $\underline{\mathfrak{C}}_{i,j}\big(\underline{\mathfrak{C}}_{i,j}(X)\big)\subseteq\underline{\mathfrak{C}}_{i,j}(X)$ and $\overline{\mathfrak{C}}_{i,j}\big(\overline{\mathfrak{C}}_{i,j}(X)\big)\supseteq\overline{\mathfrak{C}}_{i,j}(X)$, that is,  $\underline{\mathfrak{C}}_{i,j}\big(\underline{\mathfrak{C}}_{i,j}(X)\big)\subseteq\underline{\mathfrak{C}}_{i,j}(X)\subseteq X\subseteq\overline{\mathfrak{C}}_{i,j}(X)\subseteq\overline{\mathfrak{C}}_{i,j}\big(\overline{\mathfrak{C}}_{i,j}(X)\big)$.
\end{proof}

There is an example to illustrate that ``$ \underline{\mathfrak{C}}_{i,j}(X\bigcup Y)=\underline{\mathfrak{C}}_{i,j}(X)\bigcup\underline{\mathfrak{C}}_{i,j}( Y)$ and $\overline{\mathfrak{C}}_{i,j}(X\bigcap Y)=\overline{\mathfrak{C}}_{i,j}(X)\bigcap\overline{\mathfrak{C}}_{i,j}( Y)\nRightarrow X\subseteq Y$ or $X\supseteq Y$" and ``$\underline{\mathfrak{C}}_{i,j}(X\bigcup Y)\neq\underline{\mathfrak{C}}_{i,j}(X)\bigcup\underline{\mathfrak{C}}_{i,j}( Y)$, $\overline{\mathfrak{C}}_{i,j}(X\bigcap Y)\neq\overline{\mathfrak{C}}_{i,j}(X)\bigcap\overline{\mathfrak{C}}_{i,j}( Y)$".
\begin{remark}
Let $\mathbb{U}=\{x_1,x_2\}$, $\mathscr{C}$ a fuzzy covering with $\mathscr{C}_a=\{K_1,K_2\}$, where $K_1=\frac{1}{x_1}+\frac{0.1}{x_2},$ $K_2=\frac{0.1}{x_1}+\frac{1}{x_2}.$
Considering the overlap function $O_2^V$ and its implicator $I_{O_2^V}$ which are used to establish the derived coverings and operators $N_i^{\mathscr{C}_j}$. Let $i=4$, $j=0$,  and $i=1$, $j=0$, then the neighborhood operators, $N_4^{{\mathscr{C}}_0}(x_t)$ and $N_1^{{\mathscr{C}}_0}(x_t)$, of $x_t$ $(t=1,2)$ can be listed as below:
$$N_4^{{\mathscr{C}}_0}(x_1)=\frac{1}{x_1}+\frac{0.1}{x_2}, N_4^{{\mathscr{C}}_0}(x_2)=\frac{0.1}{x_1}+\frac{1}{x_2},$$ $$N_1^{{\mathscr{C}}_0}(x_1)=\frac{1}{x_1}+\frac{0.1}{x_2}, N_1^{{\mathscr{C}}_0}(x_2)=\frac{0.1}{x_1}+\frac{1}{x_2}.$$
For $X_1(x_t)=\frac{0.33}{x_1}+\frac{0.35}{x_2}$ and $Y_1(x_t)=\frac{0.35}{x_1}+\frac{0.33}{x_2}$, which are no containment relationship between each other, we have that, $$ \underline{{\mathfrak{C}}}_{4,0}(X_1\bigcup Y_1)=\frac{0.35}{x_1}+\frac{0.35}{x_2}=\underline{{\mathfrak{C}}}_{4,0}(X_1)\bigcup\underline{{\mathfrak{C}}}_{4,0}( Y_1),$$$$ \overline{{\mathfrak{C}}}_{4,0}(X_1\bigcap Y_1)=\frac{0.33}{x_1}+\frac{0.33}{x_2}=\overline{{\mathfrak{C}}}_{4,0}(X_1)\bigcap\overline{{\mathfrak{C}}}_{4,0}( Y_1),$$$$ \underline{{\mathfrak{C}}}_{1,0}(X_1\bigcup Y_1)=\frac{0.35}{x_1}+\frac{0.35}{x_2}=\underline{{\mathfrak{C}}}_{1,0}(X_1)\bigcup\underline{{\mathfrak{C}}}_{1,0}( Y_1),$$$$ \overline{{\mathfrak{C}}}_{1,0}(X_1\bigcap Y_1)=\frac{0.33}{x_1}+\frac{0.33}{x_2}=\overline{{\mathfrak{C}}}_{1,0}(X_1)\bigcap\overline{{\mathfrak{C}}}_{1,0}( Y_1).$$
Therefore, the statement $`` \underline{\mathfrak{C}}_{i,j}(X\bigcup Y)=\underline{\mathfrak{C}}_{i,j}(X)\bigcup\underline{\mathfrak{C}}_{i,j}( Y)$ and $\overline{\mathfrak{C}}_{i,j}(X\bigcap Y)=\overline{\mathfrak{C}}_{i,j}(X)\bigcap\overline{\mathfrak{C}}_{i,j}( Y)\nRightarrow X\subseteq Y$ or $X\supseteq Y"$ is true.

For $X_2(x_t)=\frac{0.91}{x_1}+\frac{0.92}{x_2}$ and $Y_2(x_t)=\frac{0.92}{x_1}+\frac{0.91}{x_2}$, we have that,
$$ \underline{{\mathfrak{C}}}_{4,0}(X_2\bigcup Y_2)=\frac{0.92}{x_1}+\frac{0.92}{x_2}\neq \underline{{\mathfrak{C}}}_{4,0}(X_2)\bigcup\underline{{\mathfrak{C}}}_{4,0}( Y_2)=\frac{0.91}{x_1}+\frac{0.91}{x_2},$$$$ \underline{{\mathfrak{C}}}_{1,0}(X_2\bigcup Y_2)=\frac{0.92}{x_1}+\frac{0.92}{x_2}\neq \underline{{\mathfrak{C}}}_{1,0}(X_2)\bigcup\underline{{\mathfrak{C}}}_{1,0}( Y_2)=\frac{0.91}{x_1}+\frac{0.91}{x_2}.$$For $X_3(x_t)=\frac{0.09}{x_1}+\frac{0.08}{x_2}$ and $Y_3(x_t)=\frac{0.08}{x_1}+\frac{0.09}{x_2}$, we have that,
$$ \overline{{\mathfrak{C}}}_{4,0}(X_3\bigcap Y_3)=\frac{0.08}{x_1}+\frac{0.08}{x_3}\neq \overline{{\mathfrak{C}}}_{4,0}(X_3)\bigcap\overline{{\mathfrak{C}}}_{4,0}( Y_3)=\frac{0.09}{x_1}+\frac{0.09}{x_3},$$$$ \overline{{\mathfrak{C}}}_{1,0}(X_3\bigcap Y_3)=\frac{0.08}{x_1}+\frac{0.08}{x_3}\neq \overline{{\mathfrak{C}}}_{1,0}(X_3)\bigcap\overline{{\mathfrak{C}}}_{1,0}( Y_3)=\frac{0.09}{x_1}+\frac{0.09}{x_3}.$$
The above examples mean the statement ``$\underline{\mathfrak{C}}_{i,j}(X\bigcup Y)\neq\underline{\mathfrak{C}}_{i,j}(X)\bigcup\underline{\mathfrak{C}}_{i,j}( Y)$, $\overline{\mathfrak{C}}_{i,j}(X\bigcap Y)\neq\overline{\mathfrak{C}}_{i,j}(X)\bigcap\overline{\mathfrak{C}}_{i,j}( Y)$" is true.
\end{remark}

\begin{remark}
Analogously, the $t$-norm based approximation operators, $(\underline{\mathcal{C}}_{i,j}(X),\overline{\mathcal{C}}_{i,j}(X))$, also satisfy the above properties.
\end{remark}
\end{proposition}

Considering the overlap functions satisfy (O7) and also their $R_O$-implicators which are used to establish coverings and neighborhood operators, for each $x\in \mathbb{U}$, we have $N_i^{\mathscr{C}_j}(x)(x)=1.$ Thus, we obtain the following properties.
\begin{proposition}
Let $N_i^{\mathscr{C}_{j}}$ be one of overlap function-based fuzzy neighborhood operators and $(\underline{\mathfrak{C}}_{i,j}(X),\overline{\mathfrak{C}}_{i,j}(X))$ a pair of approximation operators. For each $X\in \mathscr{F}(\mathbb{U})$ and $\Lambda\in [0,1]$, we obtain some conclusions as below.
\begin{flushleft}
		
		(1) $\underline{\mathfrak{C}}_{i,j}(X\bigcup \Lambda_{\mathbb{U}})=\underline{\mathfrak{C}}_{i,j}(X)\bigcup\Lambda_{\mathbb{U}}$, $\overline{\mathfrak{C}}_{i,j}(X\bigcap \Lambda_{\mathbb{U}})=\overline{\mathfrak{C}}_{i,j}(X)\bigcap\Lambda_{\mathbb{U}}$;\\
(2) $\underline{\mathfrak{C}}_{i,j}(\Lambda_{\mathbb{U}})=\Lambda_{\mathbb{U}}$, $\overline{\mathfrak{C}}_{i,j}(\Lambda_{\mathbb{U}})=\Lambda_{\mathbb{U}}$;\\
		(3) $\underline{\mathfrak{C}}_{i,j}(X\bigcap \Lambda_{\mathbb{U}})=\underline{\mathfrak{C}}_{i,j}(X)\bigcap\Lambda_{\mathbb{U}}$, $\overline{\mathfrak{C}}_{i,j}(X\bigcup \Lambda_{\mathbb{U}})=\overline{\mathfrak{C}}_{i,j}(X)\bigcup\Lambda_{\mathbb{U}}$;\\
(4) $\underline{\mathfrak{C}}_{i,j}(\emptyset)=\emptyset,$ $\overline{\mathfrak{C}}_{i,j}(\mathbb{U})=\mathbb{U}.$
	\end{flushleft}
Where $i=1,2,3,4$, $j=0,1,2,3,4,5$ and $\Lambda_{\mathbb{U}}$ is the constant fuzzy set: $\Lambda_\mathbb{U}(x)=\Lambda$, for each $x\in \mathbb{U}$.
\begin{proof}
(1) For each $x\in \mathbb{U}$, we have that $$\begin{aligned}\big(\underline{\mathfrak{C}}_{i,j}(X\bigcup \Lambda_\mathbb{U})\big)(x)&=\wedge_{y\in \mathbb{U}}[(1-N_i^{\mathscr{C}_{j}}(x)(y))\vee (X\bigcup \Lambda_\mathbb{U})(y)]\\&=\wedge_{y\in \mathbb{U}}[\big((1-N_i^{\mathscr{C}_{j}}(x)(y))\vee X(y)\big)\vee \Lambda\big)]\\&=\big(\wedge_{y\in \mathbb{U}}[(1-N_i^{\mathscr{C}_{j}}(x)(y))\vee X(y)]\big)\vee\Lambda\\&=\underline{\mathfrak{C}}_{i,j}(X)(x)\vee \Lambda.\end{aligned}$$
Thus, $\underline{\mathfrak{C}}_{i,j}(X\bigcup \Lambda_{\mathbb{U}})=\underline{\mathfrak{C}}_{i,j}(X)\bigcup\Lambda_{\mathbb{U}}$. Similarly, $\overline{\mathfrak{C}}_{i,j}(X\bigcap \Lambda_\mathbb{U})=\overline{\mathfrak{C}}_{i,j}(X)\bigcap\Lambda_\mathbb{U}$ can also be proven.

(2) According to Definition \ref{D1}, for each $x\in\mathbb{U}$, there exists at least one set $K$ from the covering, such that $K(x)=1$, then at least one neighborhood of $x$, $N_i^{\mathscr{C}_j}(x)(x)=1$. Thus $$\underline{\mathfrak{C}}_{i,j}(\Lambda_{\mathbb{U}})(x)=\wedge_{y\in \mathbb{U}}[(1-N_i^{\mathscr{C}_{j}}(x)(y))\vee \Lambda]=\Lambda,$$ and
$$\overline{\mathfrak{C}}_{i,j}(\Lambda_{\mathbb{U}})(x)=\vee_{y\in \mathbb{U}}[N_i^{\mathscr{C}_{j}}(x)(y)\wedge \Lambda]=\Lambda.$$ That is $\underline{\mathfrak{C}}_{i,j}(\Lambda_{\mathbb{U}})=\Lambda_{\mathbb{U}}$ and $\overline{\mathfrak{C}}_{i,j}(\Lambda_{\mathbb{U}})=\Lambda_{\mathbb{U}}$.

(3) According to Proposition \ref{P1} (3), we obtain that $\underline{\mathfrak{C}}_{i,j}(X\bigcap \Lambda_\mathbb{U})=\underline{\mathfrak{C}}_{i,j}(X)\bigcap\underline{\mathfrak{C}}_{i,j}( \Lambda_\mathbb{U}).$ Depending on statement (2), we have $\underline{\mathfrak{C}}_{i,j}(\Lambda_{\mathbb{U}})=\Lambda_{\mathbb{U}}$, that is $\underline{\mathfrak{C}}_{i,j}(X\bigcap \Lambda_\mathbb{U})=\underline{\mathfrak{C}}_{i,j}(X)\bigcap \Lambda_\mathbb{U}.$
Similarly, $\overline{\mathfrak{C}}_{i,j}(X\bigcup \Lambda_\mathbb{U})=\overline{\mathfrak{C}}_{i,j}(X)\bigcup\Lambda_\mathbb{U}$ can also be proven.

(4) For each $x\in \mathbb{U}$, we have $N_i^{\mathscr{C}_j}(x)(x)=1.$ Thus, $$\vee_{y\in \mathbb{U}}N_i^{\mathscr{C}_j}(x)(y)=1\Longleftrightarrow\overline{\mathfrak{C}}_{i,j}(\mathbb{U})(x)=1\Longleftrightarrow\underline{\mathfrak{C}}_{i,j}(\emptyset)(x)=0.$$
\end{proof}
\begin{remark}
Analogously, the $t$-norm based approximation operators, $(\underline{\mathcal{C}}_{i,j}(X),\overline{\mathcal{C}}_{i,j}(X))$, also satisfy the above properties.
\end{remark}
\end{proposition}

\begin{proposition}
Let $N_i^{\mathscr{C}_{j}}$ be one of overlap function-based fuzzy neighborhood operators and $(\underline{\mathfrak{C}}_{i,j}(X),\overline{\mathfrak{C}}_{i,j}(X))$ a pair of approximation operators. For each $X\in \mathscr{F}(\mathbb{U})$, we obtain some conclusions as below.
\begin{flushleft}
		
		(1) $\overline{\mathfrak{C}}_{i,j}(1_y)(x)=N_i^{\mathscr{C}_j}(x)(y)$;\\
(2) $\underline{\mathfrak{C}}_{i,j}(1_{U-\{y\}})(x)=1-N_i^{\mathscr{C}_j}(x)(y)$;\\
		(3) $\overline{\mathfrak{C}}_{i,j}(1_X)(x)=\vee_{y\in X}N_i^{\mathscr{C}_j}(x)(y)$;\\
(4) $\underline{\mathfrak{C}}_{i,j}(1_X)(x)=\wedge_{y\notin X}\big(1-N_i^{\mathscr{C}_j}(x)(y)\big).$
	\end{flushleft}
Where $1_y$ denotes the fuzzy singleton with value $1$ at $y$ and $0$ at the others; $1_X$ denotes the characteristic function of $X$.
\begin{proof}
(1) For each $x,y\in \mathbb{U},$ we obtain that $1_y(x)=0$ if only if $x\neq y$ by the definition of $1_y$. Thus, the following equations hold.
$$\overline{\mathfrak{C}}_{i,j}(1_y)(x)=\vee_{z\in \mathbb{U}}[N_i^{\mathscr{C}_j}(x)(z)\wedge1_y(z)]=N_i^{\mathscr{C}_j}(x)(y).$$

(2) This conclusion is obvious from (1) and the duality.

(3) For each $x\in\mathbb{U}$ and $X\in \mathscr{F}(\mathbb{U})$, we obtain that $1_X(x)=0$ if only if $x\notin X$ by the definition of $1_X$. Thus, $$\begin{aligned}\overline{\mathfrak{C}}_{i,j}(1_X)(x)&=\vee_{y\in \mathbb{U}}[N_i^{\mathscr{C}_j}(x)(y)\wedge1_X(y)]\\&=(\vee_{y\in X}[N_i^{\mathscr{C}_j}(x)(y)\wedge1_X(y)])\vee(\vee_{y\notin X}[N_i^{\mathscr{C}_j}(x)(y)\wedge1_X(y)])\\&=\vee_{y\in X}N_i^{\mathscr{C}_j}(x)(y).\end{aligned}$$

(4)  This conclusion is obvious from (3) and the duality.
\end{proof}
\begin{remark}
Analogously, the $t$-norm based approximation operators, $(\underline{\mathcal{C}}_{i,j}(X),\overline{\mathcal{C}}_{i,j}(X))$, also satisfy the above properties.
\end{remark}
\end{proposition}

\section{Decision-making}\label{dm}
In this section, two series of new TOPSIS methods which are established by different foundations are put forward, in detail, the fuzzy covering-based rough set models defined in the previous section can be used to calculate the objective weights of attributes. The data processing procedures of our new ones for MADM problems are respectively discussed in Subsection \ref{ss61}. In order to demonstrate the practicality and effectiveness of our methods, we solve a medical issue with the new and the existing methodologies in the Subsection \ref{ss62} and compare them with each other in the last subsection.

\subsection{Processes and algorithms of new methods}\label{ss61}
In this subsection, we aim at combining MADM problems with the ONRFRS models and TNRFRS models. To achieve this objective, the background we need is described in the first part. Besides, the details and the algorithms for decision-making are depicted in the second part and last part, respectively.
\subsubsection{Background description}
Assume that $\mathbb{U}=\{x_t:t=1, \dots, n\}$ is a finite discrete set of decision objects and $\mathscr{C}=\{K_s:s=1,\dots, m\}$ indicates a non-empty finite set of attributes. Then, let $K_s(x_t)$ be the assessment of experts for the alternative, $x_t$, on the attribute, $K_s$. Particularly, for each object $x_t\in \mathbb{U}$, there is at least one attribute $K_s\in\mathscr{C}$ that satisfies $K_s(x_t)=1$. In addition, considering the principle of TOPSIS methodologies, we need to calculate the positive ideal solution (PIS), $\mathfrak{I}_{\uparrow}$, and the negative ideal solution (NIS), $\mathfrak{I}_{\downarrow}$, for each attribute.
\subsubsection{Decision-making processes}
Considering the indispensable role of fuzzy rough set theory in the decision-making process for uncertain information, we present a new objective attribute weight determination method, via uniting conventional TOPSIS methods with the models $(\underline{\mathfrak{C}}_{i,j}(X),\overline{\mathfrak{C}}_{i,j}(X))$ and $(\underline{\mathcal{C}}_{i,j}(X),\overline{\mathcal{C}}_{i,j}(X))$. Next, the thorough steps are described below.

For the first step, we need to construct the MADM matrix  by means of fuzzy information as shown in the Table \ref{ddddd}. It should be noted that each alternative $x_t$ has a scoring attribute of $1$.
\begin{table}[H]
\caption{{Muti-attribute decision-making matrix}}\label{ddddd}
	\begin{center}
		\begin{tabular}{l|llll}
			\hline
			$\mathbb{U}/ \mathscr{C}$ & $K_1$  & $K_2$ & $\cdots$ & $K_m$\\
			\hline
			\rule{0pt}{12pt}
			$x_1$  & $K_1(x_1)$ & $K_2(x_1)$ & $\cdots$ & $K_m(x_1)$\\
			\rule{0pt}{16pt}
			$x_2$& $K_1(x_2)$ & $K_2(x_2)$ & $\cdots$ & $K_m(x_2)$\\
			\rule{0pt}{16pt}
			$\vdots$ & $\vdots$ &$ \vdots$ & $\cdots$ & $\vdots$\\
			\rule{0pt}{16pt}
			$x_n$ & $K_1(x_n)$ & $K_2(x_n)$ & $\cdots$ & $K_m(x_n)$\\
			
			\hline
		\end{tabular}
	\end{center}
	
\end{table}

And then, the second step is to determine the PIS ${\mathfrak{I}_s}_{\uparrow}$ and the NIS ${\mathfrak{I}_s}_{\downarrow}$ under the attribute $K_s$. The following formulas reveal their calculation ways.
\begin{equation}
{\mathfrak{I}_s}_{\uparrow}=\left\{\begin{matrix}
		\vee_{1\leq t\leq n}K_s(x_t), & K_s\in \mathfrak{M}, s\in\{1,2,\dots, m\};\\
	\wedge_{1\leq t\leq n}K_s(x_t), & K_s\in \mathfrak{N}, s\in\{1,2,\dots, m\}.
		\end{matrix}\right.\label{ee1}
\end{equation}
\begin{equation}
{\mathfrak{I}_s}_{\downarrow}=\left\{\begin{matrix}
		\wedge_{1\leq t\leq n}K_s(x_t), & K_s\in \mathfrak{M}, s\in\{1,2,\dots, m\};\\
	\vee_{1\leq t\leq n}K_s(x_t), & K_s\in \mathfrak{N}, s\in\{1,2,\dots, m\}.
		\end{matrix}\right.\label{ee2}
\end{equation}
Where $\mathfrak{M}$ and $\mathfrak{N}$ respectively mean the set of benefit attributes and cost attributes.

For the third step, by means of the one-dimensional Euclidean distance formula, we can obtain $\mathfrak{D}_{st}^{\uparrow}$ which is the distance between the expert score $K_s(x_t)$ of the option $x_t$ under the attribute $K_s$ and the PIS ${\mathfrak{I}_s}_{\uparrow}$. Analogously, the distance, $\mathfrak{D}_{st}^{\downarrow}$, from $K_s(x_t)$ to NIS ${\mathfrak{I}_s}_{\downarrow}$ can also be computed. The following functions elaborate on the calculation process.
\begin{equation}
\mathfrak{D}_{st}^{\uparrow}=\|K_s(x_t)-{\mathfrak{I}_s}_{\uparrow}\|={\mathfrak{I}_s}_{\uparrow}-K_s(x_t).
\end{equation}
\begin{equation}
\mathfrak{D}_{st}^{\downarrow}=\|K_s(x_t)-{\mathfrak{I}_s}_{\downarrow}\|=K_s(x_t)-{\mathfrak{I}_s}_{\downarrow}.
\end{equation}

For the penultimate step, the calculation method for attribute weights is given. According to the definition of fuzzy set and the fuzzy information is shown in the Table \ref{ddddd}, $m$ fuzzy sets can be obtained via taking the score of each alternative $x_t$ under every attribute $K_s$ as the membership value of $x_t$ under the fuzzy set $K_s$. And the attribute fuzzy sets are expressed as below.
\begin{equation}
K_s=\sum_{t=1}^{n}\frac{K_s(x_t)}{x_t}, s\in\{1,2,\dots,m\}.
\end{equation}
Go further, for the purpose of determining the weight of each attribute fuzzy set $K_s$, we resort to the concept of approximate precision and there are two types of models available in its construction process. For the ONRFRS models displayed in Table \ref{t3}, the overlap function-based approximate precision of $K_s$ is defined as below.
\begin{equation}
\mathfrak{A}_{(\underline{\mathfrak{C}}_{*},\overline{\mathfrak{C}}_{*})}(K_s)=\frac{|\underline{\mathfrak{C}}_{*}(K_s)|}{|\overline{\mathfrak{C}}_{*}(K_s)|}, s\in \{1,2,\dots,m\}.
\end{equation}
Where $*$ refers to the index of ONRFRS models, and for each $K\in \mathscr{F}(\mathbb{U})$, $|K|$ represents the cardinality of $K$. Analogously, $t$-norm-based approximate precision of $K_s$ is given as follow,
\begin{equation}
\mathcal{A}_{(\underline{\mathcal{C}}_{*},\overline{\mathcal{C}}_{*})}(K_s)=\frac{|\underline{\mathcal{C}}_{*}(K_s)|}{|\overline{\mathcal{C}}_{*}(K_s)|}, s\in \{1,2,\dots,m\}.
\end{equation}
Thus, the attribute weight vector $\mathfrak{W}=\{\mathfrak{W}_1,\mathfrak{W}_2,\dots, \mathfrak{W}_m\}$ with respect to the $\mathfrak{A}_{(\underline{\mathfrak{C}}_{*},\overline{\mathfrak{C}}_{*})}$ and the attribute weight vector $\mathcal{W}=\{\mathcal{W}_1,\mathcal{W}_2,\dots, \mathcal{W}_m\}$ with respect to the $\mathcal{A}_{(\underline{\mathcal{C}}_{*},\overline{\mathcal{C}}_{*})}$ are defined as below, respectively.
\begin{equation}\label{wf1}
\mathfrak{W}_s=\frac{\mathfrak{A}_{(\underline{\mathfrak{C}}_{*},\overline{\mathfrak{C}}_{*})}(K_s)}{\sum_{s=1}^{m}\mathfrak{A}_{(\underline{\mathfrak{C}}_{*},\overline{\mathfrak{C}}_{*})}(K_s)},
s\in\{1,2,\dots,m\},
\end{equation}
and
\begin{equation}\label{wf2}
\mathcal{W}_s=\frac{\mathcal{A}_{(\underline{\mathcal{C}}_{*},\overline{\mathcal{C}}_{*})}(K_s)}{\sum_{s=1}^{m}\mathcal{A}_{(\underline{\mathcal{C}}_{*},\overline{\mathcal{C}}_{*})}(K_s)},
s\in\{1,2,\dots,m\}.
\end{equation}

For the final step, let
\begin{equation}\label{func14}
\mathfrak{H}^{\uparrow}=\min_{1\leq t\leq n}(\sum_{s=1}^{m}\mathfrak{W}_s\mathfrak{D}_{st}^{\uparrow}),
\mathfrak{H}^{\downarrow}=\max_{1\leq t\leq n}(\sum_{s=1}^{m}\mathfrak{W}_s\mathfrak{D}_{st}^{\downarrow}),
\end{equation}
or
\begin{equation}\label{func15}
\mathcal{H}^{\uparrow}=\min_{1\leq t\leq n}(\sum_{s=1}^{m}\mathcal{W}_s\mathfrak{D}_{st}^{\uparrow}),
\mathcal{H}^{\downarrow}=\max_{1\leq t\leq n}(\sum_{s=1}^{m}\mathcal{W}_s\mathfrak{D}_{st}^{\downarrow}).
\end{equation}
Depending on the type of weights chosen, the closeness coefficient of alternative $x_t$ can be defined in two separate forms as follows,
\begin{equation}
\mathfrak{H}_t=\frac{\sum_{s=1}^{m}\mathfrak{W}_s\mathfrak{D}_{st}^{\downarrow}}{\mathfrak{H}^{\downarrow}}-\frac{\sum_{s=1}^{m}\mathfrak{W}_s\mathfrak{D}_{st}^{\uparrow}}{\mathfrak{H}^{\uparrow}}, t\in\{1,2,\dots,n\},
\end{equation}
\begin{equation}
\mathcal{H}_t=\frac{\sum_{s=1}^{m}\mathcal{W}_s\mathfrak{D}_{st}^{\downarrow}}{\mathcal{H}^{\downarrow}}-\frac{\sum_{s=1}^{m}\mathcal{W}_s\mathfrak{D}_{st}^{\uparrow}}{\mathcal{H}^{\uparrow}}, t\in\{1,2,\dots,n\}.
\end{equation}
Where $\mathfrak{H}_t\leq0$ ($\mathcal{H}_t\leq0$), $t\in\{1,2,\dots,n\}$. According to the value of $\mathfrak{H}_t$ ($\mathcal{H}_t$), we can get the ranking of alternative targets, where the larger the value of $\mathfrak{H}_t$ ($\mathcal{H}_t$), the higher the priority of its corresponding alternative $x_t$.
\subsubsection{The algorithms}
In this part, we sort out four algorithms, which are respectively about the fuzzy minimal description, the fuzzy maximal description, the independent element and the above decision-making process.

The Algorithm \ref{amin} which describes how to calculate the fuzzy minimal description of $x_t$ from the fuzzy neighborhood system $\mathbb{C}(\mathscr{C}, x)$. Similarly, the algorithm about the fuzzy maximal description of $x_t$ are shown in Algorithm \ref{amax}. In order to establish the model $(\mathfrak{C}_M, \mathfrak{C}_M)$ and the model $(\mathcal{C}_M, \mathcal{C}_M)$, we need to find a way to calculate the independent element $\mathscr{C}_{\cap}$ in the fuzzy covering $\mathscr{C}$. Hence we give Algorithm \ref{ind} to our requirements. With the above guarantees, the decision-making process we discussed in the previous part is illustrated in the Algorithm \ref{dcm}.
\begin{algorithm}[p]
\caption{ The algorithm for calculating fuzzy minimal description}\label{amin} 
\hspace*{0.02in} {\bf Input:} 
A finite universe $\mathbb{U}=\{x_t:t=1,\dots, n\}$ and a finite fuzzy covering $\mathscr{C}=\{K_s:s=1,\dots,m\}$.\\
\hspace*{0.02in} {\bf Output:} 
The fuzzy minimal description $md(\mathscr{C},x_t)$ for each $x_t$.
\begin{algorithmic}[1]
\For{$t=1:n$} 
\For{$s=1:m$}
\State $q=0$; $c=0;$
\For{$p=1:m$}
\If{$K_s(x_t)=K_p(x_t)$}
\State $q++;$
\For{$a=1:n$}
\State  $b=0$;
\If{$K_s(x_a)\leq K_p(x_a)$}
\State $b++$;
\EndIf
\EndFor
\If{$b=1$}
\State $c++$;
\EndIf
\EndIf　

\EndFor
\If{$q=1 $ $| $ $|$ $ c=0$}
\State $K_s\in md(\mathscr{C},x_t)$;
\EndIf
\EndFor
\EndFor

\Return $md(\mathscr{C},x)$.
\end{algorithmic}
\end{algorithm}

\begin{algorithm}[htpb]
\caption{ {The algorithm for calculating fuzzy maximal description} }\label{amax} 
\hspace*{0.02in} {\bf Input:} 
A finite universe $\mathbb{U}=\{x_t:t=1,\dots, n\}$ and a finite fuzzy covering $\mathscr{C}=\{K_s:s=1,\dots,m\}$.\\
\hspace*{0.02in} {\bf Output:} 
The fuzzy maximal description $MD(\mathscr{C},x_t)$ for each $x_t$.
\begin{algorithmic}[1]
\For{$t=1:n$} 
\For{$s=1:m$}
\State $q=0$; $c=0;$
\For{$p=1:m$}
\If{$K_s(x_t)=K_p(x_t)$}
\State $q++;$
\For{$a=1:n$}
\State  $b=0$;
\If{$K_s(x_a)\geq K_p(x_a)$}
\State $b++$;
\EndIf
\EndFor
\If{$b=1$}
\State $c++$;
\EndIf
\EndIf　

\EndFor
\If{$q=1 $ $| $ $|$ $ c=0$}
\State $K_s\in MD(\mathscr{C},x_t)$;
\EndIf
\EndFor
\EndFor

\Return $MD(\mathscr{C},x)$.
\end{algorithmic}
\end{algorithm}

\begin{algorithm}[htpb]
\caption{{The algorithm for calculating independent element} }\label{ind} 
\hspace*{0.02in} {\bf Input:} 
A finite universe $\mathbb{U}=\{x_t:t=1,\dots, n\}$ and a finite fuzzy covering $\mathscr{C}=\{K_s:s=1,\dots,m\}$.\\
\hspace*{0.02in} {\bf Output:} 
The independent element $\mathscr{C}_{\cap}$ for $\mathscr{C}$.
\begin{algorithmic}[1]
\State $\mathscr{C}_{\cap}=\mathscr{C}$;
\For{$s=1:m$} 
\State $\mathscr{C}^{'}=\mathscr{C}\setminus K_s$; $\mathscr{C}^{'}\ni K^{'}_{s^{'}}$; $i=0$;
\For{$s^{'}=1:(m-1)$}
\State $j=0;$ $K=K^{'}_{s^{'}}$;
\While{$s^{'}+j\neq m-1$}
\For{$a=(s^{'}+j):(m-1)$}
\State $K=K\cap K^{'}_a$;
\If {$K=K_s$}
\State $i++;$
\EndIf
\EndFor
\State $j++$; $K=K^{'}_{s^{'}}$;
\EndWhile
\EndFor
\If{$i\neq 0$}
\State $\mathscr{C}_{\cap}=\mathscr{C}_{\cap}\setminus K_s$;
\EndIf
\EndFor
\Return $\mathscr{C}_{\cap}$.
\end{algorithmic}
\end{algorithm}
\begin{algorithm}[htpb]\scriptsize
\caption{{ Decision making} }\label{dcm} 
\hspace*{0.02in} {\bf Input:} 
Muti-attribute decision-making matrix, ONRFRS models and TNRFRS models.\\
\hspace*{0.02in} {\bf Output:} 
The ranking of alternatives.
\begin{algorithmic}[1]
\State ${\mathfrak{I}_s}_{\uparrow}=\emptyset$; ${\mathfrak{I}_s}_{\downarrow}=\emptyset$;
\For{$s=1:m$} 
\For{$t=1:n$}
\State {\bf compute:} the PIS and the NIS;
\If{$K_s\in \mathfrak{M}$}
\State ${\mathfrak{I}_s}_{\uparrow}\ni\vee_{1\leq t\leq n}K_s(x_t)$;
\State ${\mathfrak{I}_s}_{\downarrow}\ni\wedge_{1\leq t\leq n}K_s(x_t)$;
\EndIf
\If{$K_s\in \mathfrak{N}$}
\State ${\mathfrak{I}_s}_{\uparrow}\ni\wedge_{1\leq t\leq n}K_s(x_t)$;
\State ${\mathfrak{I}_s}_{\downarrow}\ni\vee_{1\leq t\leq n}K_s(x_t)$;
\EndIf
\EndFor
\EndFor
\State $\mathfrak{D}_{st}^{\uparrow}=\emptyset$; $\mathfrak{D}_{st}^{\downarrow}=\emptyset$;
\For{$s=1:m$} 
\State {\bf compute:} the distances;
\For{$t=1:n$}

\State $\mathfrak{D}_{st}^{\uparrow}\ni\|K_s(x_t)-{\mathfrak{I}_s}_{\uparrow}\|$;
\State $\mathfrak{D}_{st}^{\downarrow}\ni\|K_s(x_t)-{\mathfrak{I}_s}_{\downarrow}\|$;
\EndFor
\EndFor
\If{we use the ONRFRS models}
\State $\mathfrak{A}_{(\underline{\mathfrak{C}}_{*},\overline{\mathfrak{C}}_{*})}(K_s)=\emptyset;$
\State {\bf compute:} the overlap function-based approximate precision of $K_s$;
\For{$s=1:m$}
\State $\mathfrak{A}_{(\underline{\mathfrak{C}}_{*},\overline{\mathfrak{C}}_{*})}(K_s)\ni\frac{|\underline{\mathfrak{C}}_{*}(K_s)|}{|\overline{\mathfrak{C}}_{*}(K_s)|}$;
\EndFor
\State $\mathfrak{W}_s=\emptyset;$
\State {\bf compute:} the attribute weight vector;
\For{$s=1:m$}
\State $\mathfrak{W}_s\ni\frac{\mathfrak{A}_{(\underline{\mathfrak{C}}_{*},\overline{\mathfrak{C}}_{*})}(K_s)}{\sum_{s=1}^{m}\mathfrak{A}_{(\underline{\mathfrak{C}}_{*},\overline{\mathfrak{C}}_{*})}(K_s)}
$;
\EndFor
\State $\mathfrak{H}^{\uparrow}=\emptyset;$ $\mathfrak{H}^{\downarrow}=\emptyset;$
\For{$t=1:n$}
\State $\mathfrak{H}^{\uparrow}\ni\min_{1\leq t\leq n}(\sum_{s=1}^{m}\mathfrak{W}_s\mathfrak{D}_{st}^{\uparrow});$
\State $\mathfrak{H}^{\downarrow}\ni\max_{1\leq t\leq n}(\sum_{s=1}^{m}\mathfrak{W}_s\mathfrak{D}_{st}^{\downarrow});$
\EndFor
\State $\mathfrak{H}_t=\emptyset;$
\State {\bf compute:} the closeness coefficient;
\For{$t=1:n$}
\State $\mathfrak{H}_t\ni\frac{\sum_{s=1}^{m}\mathfrak{W}_s\mathfrak{D}_{st}^{\downarrow}}{\mathfrak{H}^{\downarrow}}-\frac{\sum_{s=1}^{m}\mathfrak{W}_s\mathfrak{D}_{st}^{\uparrow}}{\mathfrak{H}^{\uparrow}}$;
\EndFor
\EndIf
\If{we use the TNRFRS models}
\State $\mathcal{A}_{(\underline{\mathcal{C}}_{*},\overline{\mathcal{C}}_{*})}(K_s)=\emptyset;$
\State {\bf compute:} the $t$-norm-based approximate precision of $K_s$;
\For{$s=1:m$}
\State $\mathcal{A}_{(\underline{\mathcal{C}}_{*},\overline{\mathcal{C}}_{*})}(K_s)\ni\frac{|\underline{\mathcal{C}}_{*}(K_s)|}{|\overline{\mathcal{C}}_{*}(K_s)|}$;
\EndFor
\State $\mathcal{W}_s=\emptyset;$
\State {\bf compute:} the attribute weight vector;
\For{$s=1:m$}
\State $\mathcal{W}_s\ni\frac{\mathcal{A}_{(\underline{\mathcal{C}}_{*},\overline{\mathcal{C}}_{*})}(K_s)}{\sum_{s=1}^{m}\mathcal{A}_{(\underline{\mathcal{C}}_{*},\overline{\mathcal{C}}_{*})}(K_s)}
$;
\EndFor
\State $\mathcal{H}^{\uparrow}=\emptyset;$ $\mathcal{H}^{\downarrow}=\emptyset;$
\For{$t=1:n$}
\State $\mathcal{H}^{\uparrow}\ni\min_{1\leq t\leq n}(\sum_{s=1}^{m}\mathcal{W}_s\mathfrak{D}_{st}^{\uparrow});$
\State $\mathcal{H}^{\downarrow}\ni\max_{1\leq t\leq n}(\sum_{s=1}^{m}\mathcal{W}_s\mathfrak{D}_{st}^{\downarrow});$
\EndFor
\State $\mathcal{H}_t=\emptyset;$
\State {\bf compute:} the closeness coefficient;
\For{$t=1:n$}
\State $\mathcal{H}_t\ni\frac{\sum_{s=1}^{m}\mathcal{W}_s\mathfrak{D}_{st}^{\downarrow}}{\mathcal{H}^{\downarrow}}-\frac{\sum_{s=1}^{m}\mathcal{W}_s\mathfrak{D}_{st}^{\uparrow}}{\mathcal{H}^{\uparrow}}$;
\EndFor
\EndIf

\Return $\mathfrak{H}_t$ or $\mathcal{H}_t$.
\end{algorithmic}
\end{algorithm}
\subsection{Solutions for an example}\label{ss62}
In order to reduce the rejection of synthetic materials in the patient's body and to avoid environmental pollution from medical waste, biosynthetic nanomaterials technology based on bio-materials and nano-technology receives wide attention. In view of the shortcomings of traditional medical materials, clinicians are demanding more performance from new materials, properties such as non-chemical activity, non-carcinogenicity, non-allergic reaction, processability, sterility, anti-infectiousness, and good reactivity to in vivo tissues have become the new benchmark for measuring medical materials.

Bone transplantation is the primary treatment for repairing bone defects resulting from the occurrence of fractures and the clinician's choice of the desired biosynthetic nanomaterials can be translated into a MADM problem. Suppose now facing a bone transplantation surgery and the attending physician needs to select the most suitable one for the patient from the $5520$ bone transplant replacement materials provided by Johnson $\&$ Johnson China Ltd. and other companies. Let $\mathbf{U}=\{X_t: t=1,\dots, 5520\}$ be $5520$ alternatives which are the bone transplant replacement materials, and $\mathscr{C}=\{K_s: s=1,\dots, 15\}$ be $15$ benefit attributes, such as non-chemical activity, non-carcinogenicity, non-allergic reaction, processability, sterility, anti-infectiousness, good reactivity to in vivo tissues and so on. By means of the above indication, the bone transplant replacement material selection problem is described as a multi-attribute decision problem.

The clinician's evaluation value of the alternative with respect to the attributes are shown in \cite{zhang2019topsis}, and Zhang et al. transformed them into fuzzy numbers in their paper. For the sake of simplifying the calculation, we selected $15$ bone transplant replacement materials without loss of generality from the original set to form the sample data set $\mathbb{U}=\{x_t: t=1,\dots, 15\}$, and the MADM matrix are illustrated in Table \ref{aaaa}.
\begin{sidewaystable}[thp]
\caption{{Muti-attribute decision-making matrix with fuzzy information}}\label{aaaa}
	\begin{center}
\resizebox{.95\columnwidth}{!}{
\begin{tabular}{l|lllllllllllllll}
			\hline
			$\mathbb{U}/ \mathscr{C}$ & $K_1$  & $K_2$ & $K_3$ & $K_4$ & $K_5$ & $K_6$ & $K_7$ & $K_8$ & $K_9$ & $K_{10}$ & $K_{11}$ & $K_{12}$ & $K_{13}$ & $K_{14}$ & $K_{15}$ \\
			\hline
			\rule{0pt}{12pt}
			$x_1$ & $0.46$  & $0.68$ & $0.11$ & $0.99$ & $0.16$ & $0.63$ & $1$ & $0.78$ & $0.71$ & $0.5$ & $0.29$ & $0.41$ & $0.12$ & $0.5$ & $0.02$\\
			\rule{0pt}{12pt}
			$x_2$ & $1$  & $0.13$ & $0.93$ & $0.17$ & $0.67$ & $0.36$ & $0.07$ & $0.69$ & $ 0.62$ & $ 0.49$ & $0.54$ & $ 0.67$ & $ 0.81$ & $0.61$ & $0.44$\\
			\rule{0pt}{12pt}
			$x_3$ & $1$  & $0.72$ & $0.19$ & $0.26$ & $0.89$ & $1$ & $0.94$ & $0.01$ & $0.34$ & $0.88$ & $0.98$ & $0.93$ & $0.32$ & $0.82$ & $0.83$\\
		\rule{0pt}{12pt}
			$x_4$ & $0.33$ &$0.11$&$ 0.27$&$ 0.4$&$ 0.52$&$ 0.22$&$0.02$&$ 0.84$&$ 0.94$&$ 0.35$&$ 0.72$&$ 1$&$ 0.25$&$ 0.53$&$ 0.62$\\
			\rule{0pt}{12pt}
			$x_5$ & $0.3$&$ 0.12$&$ 0.8$&$ 0.07$&$ 0.7$&$ 0.65$&$ 0.68$&$ 0.92$&$ 1$&$ 0.45$&$ 0.84$&$ 0.48$&$ 0.34$&$ 0.2$&$ 0.52$\\
\rule{0pt}{12pt}
			$x_6$ & $0.06$&$ 0.64$&$ 0.49$&$ 0.68$&$ 1$&$ 0.6$&$ 0.78$&$ 0.77$&$ 0.73$&$ 0.96$&$ 0.43$&$ 0.76$&$ 0.38$&$ 0.45$&$ 0.86$\\
\rule{0pt}{12pt}
			$x_7$ & $ 0.3$&$ 0.33$&$ 0.77$&$ 0.4$&$ 0.95$&$ 0.39$&$ 0.53$&$ 0.04$&$ 0.65$&$ 1$&$ 0.47$&$ 0.42$&$ 0.55$&$ 0.43$&$ 0.1$\\
\rule{0pt}{12pt}
			$x_8$ & $0.63$&$ 0.74$&$ 0.67$&$ 0.15$&$ 0.81$&$ 0.18$&$ 0.14$&$ 1$&$ 0.84$&$ 0.59$&$ 0.5$&$ 0.39$&$ 0.52$&$ 0.72$&$ 0.14$\\
			\rule{0pt}{12pt}
			$x_9$ & $0.09$&$ 0.23$&$ 1$&$ 0.38$&$ 0.75$&$ 0.73$&$ 0.22$&$ 0.27$&$ 0.32$&$ 0.68$&$ 0.65$&$ 0.45$&$ 0.66$&$ 0.35$&$ 0.56$\\
\rule{0pt}{12pt}
			$x_{10}$ & $1$&$ 0.87$&$ 0.06$&$ 0.87$&$ 0.33$&$ 0.73$&$ 0.11$&$ 0.62$&$ 0.55$&$ 0.81$&$ 0.48$&$ 0.88$&$ 0.4$&$ 1$&$ 0.85$\\
\rule{0pt}{12pt}
			$x_{11}$ & $0.67$&$ 1$&$ 0.66$&$ 0.45$&$ 0.93$&$ 0.64$&$ 0.77$&$ 0.49$&$ 0.33$&$ 0.26$&$ 0.94$&$ 0.76$&$ 0.54$&$ 0.78$&$ 0.08$\\
\rule{0pt}{12pt}
			$x_{12}$ & $0.74$&$ 0.09$&$ 0.25$&$ 1$&$ 0.42$&$ 0.62$&$ 0.08$&$ 0.1$&$ 0.35$&$ 0.69$&$ 0.52$&$ 0.69$&$ 0.89$&$ 0.02$&$ 0.37$\\
\rule{0pt}{12pt}
			$x_{13}$ & $0.6$&$ 0.77$&$ 0.72$&$ 0.42$&$ 0.16$&$ 0.18$&$ 0.34$&$ 0.19$&$ 0.84$&$ 0.33$&$ 1 $&$0.88$&$ 0.1$&$ 0.73$&$ 0.66$\\
\rule{0pt}{12pt}
			$x_{14}$ & $0.69$&$ 0.48$&$ 0.1$&$ 0.12$&$ 0.75$&$ 0.63$&$ 0.41$&$ 1$&$ 0.28$&$ 0.93$&$ 0.94$&$ 0.28$&$ 1$&$ 0.11$&$ 0.94$\\
\rule{0pt}{12pt}
			$x_{15}$ & $0.01$&$ 0.37$&$ 0.01$&$ 0.34$&$ 0.92$&$ 0.37$&$ 0.68$&$ 0.18$&$ 0.69$&$ 0.12$&$ 0.56$&$ 0.21$&$ 0.3$&$ 0.45$&$ 1$\\
			\hline
		
		\end{tabular}}
	\end{center}
	
\end{sidewaystable}

No loss of generality, we use the model $(\underline{\mathfrak{C}}_{A1}(X)$, $\overline{\mathfrak{C}}_{A1}(X))$ and the model $(\underline{\mathcal{C}}_{A1}(X)$, $\overline{\mathcal{C}}_{A1}(X))$ as an example to show the whole decision process, and the decision results of other models are displayed in the Table \ref{ranking1}.

{\bf Step 1:} According to the scoring by medical experts, we can obtain the PIS ${\mathfrak{I}_s}_{\uparrow}$ and NIS ${\mathfrak{I}_s}_{\downarrow}$ through formula \ref{ee1} and formula \ref{ee2}, and the calculated conclusion are depicted in the Table \ref{a5a5}.
\begin{table}[H]\small
	\begin{center}
\caption{The positive ideal solution ${\mathfrak{I}_s}_{\uparrow}$ and the negative ideal solution ${\mathfrak{I}_s}_{\downarrow}$}\label{a5a5}
\resizebox{.95\columnwidth}{!}{
\begin{tabular}{l|lllllllllllllll}
			\hline
			$\quad$ & $K_1$  & $K_2$ & $K_3$ & $K_4$ & $K_5$ & $K_6$ & $K_7$ & $K_8$ & $K_9$ & $K_{10}$ & $K_{11}$ & $K_{12}$ & $K_{13}$ & $K_{14}$ & $K_{15}$ \\
			\hline
			\rule{0pt}{12pt}
			${\mathfrak{I}_s}_{\uparrow}$ & $1$  & $1$ & $1$ & $1$ & $1$ & $1$ & $1$ & $1$ & $1$ & $1$ & $1$ & $1$ & $1$ & $1$ & $1$\\
			\rule{0pt}{12pt}
			${\mathfrak{I}_s}_{\downarrow}$ & $0.01$  & $0.09$ & $0.01$ & $0.07$ & $0.16$ & $0.18$ & $0.02$ & $0.01$ & $ 0.28$ & $ 0.12$ & $0.29$ & $ 0.21$ & $ 0.1$ & $0.02$ & $0.02$\\
			\hline
		
		\end{tabular}}
	\end{center}
	
\end{table}

{\bf Step 2:} By means of Euclidean distance equation, the distances $\mathfrak{D}_{st}^{\uparrow}$ and $\mathfrak{D}_{st}^{\downarrow}$ are given in the Table \ref{top} and \ref{down}. Note that, each element in Table \ref{top} indicates a positive ideal distance. Specifically, the data value, $\mathfrak{D}_{11}^{\uparrow}=0.54$, means the alternative $x_1$ evaluated value under the attribute $K_1$ away from PIS of $K_1$ is $0.54$. Likewise, each data in Table \ref{down} has a similar meaning that represents a negative ideal distance.
\begin{table}[H]
	\begin{center}
\caption{The distance $\mathfrak{D}_{st}^{\uparrow}$ from the attribute value $K_s(x_t)$ to the PIS ${\mathfrak{I}_s}_{\uparrow}$}\label{top}
\resizebox{.95\columnwidth}{!}{
\begin{tabular}{l|lllllllllllllll}
			\hline
			$\mathfrak{D}_{st}^{\uparrow}$ & $K_1$  & $K_2$ & $K_3$ & $K_4$ & $K_5$ & $K_6$ & $K_7$ & $K_8$ & $K_9$ & $K_{10}$ & $K_{11}$ & $K_{12}$ & $K_{13}$ & $K_{14}$ & $K_{15}$ \\
			\hline
			\rule{0pt}{12pt}
			$x_1$ & $0.54$  & $0.32$ & $0.89$ & $0.01$ & $0.84$ & $0.37$ & $0$ & $0.22$ & $0.29$ & $0.5$ & $0.71$ & $0.59$ & $0.88$ & $0.5$ & $0.98$\\
			\rule{0pt}{12pt}
			$x_2$ & $0$  & $0.87$ & $0.07$ & $0.83$ & $0.33$ & $0.64$ & $0.93$ & $0.31$ & $ 0.38$ & $ 0.51$ & $0.46$ & $ 0.33$ & $ 0.19$ & $0.39$ & $0.56$\\
			\rule{0pt}{12pt}
			$x_3$ & $0$  & $0.28$ & $0.81$ & $0.74$ & $0.11$ & $0$ & $0.06$ & $0.99$ & $0.66$ & $0.12$ & $0.02$ & $0.07$ & $0.68$ & $0.18$ & $0.17$\\
		\rule{0pt}{12pt}
			$x_4$ & $0.67$ &$0.89$&$ 0.73$&$ 0.6$&$ 0.48$&$ 0.78$&$0.98$&$ 0.16$&$ 0.06$&$ 0.65$&$ 0.28$&$ 0$&$ 0.75$&$ 0.47$&$ 0.38$\\
			\rule{0pt}{12pt}
			$x_5$ & $0.7$&$ 0.88$&$ 0.2$&$ 0.93$&$ 0.3$&$ 0.35$&$ 0.32$&$ 0.08$&$ 0$&$ 0.55$&$ 0.16$&$ 0.52$&$ 0.66$&$ 0.8$&$ 0.48$\\
\rule{0pt}{12pt}
			$x_6$ & $0.94$&$ 0.36$&$ 0.51$&$ 0.32$&$ 0$&$ 0.4$&$ 0.22$&$ 0.23$&$ 0.27$&$ 0.04$&$ 0.57$&$ 0.24$&$ 0.62$&$ 0.55$&$ 0.14$\\
\rule{0pt}{12pt}
			$x_7$ & $ 0.7$&$ 0.67$&$ 0.23$&$ 0.6$&$ 0.05$&$ 0.61$&$ 0.47$&$ 0.96$&$ 0.35$&$ 0$&$ 0.53$&$ 0.58$&$ 0.45$&$ 0.57$&$ 0.9$\\
\rule{0pt}{12pt}
			$x_8$ & $0.37$&$ 0.26$&$ 0.33$&$ 0.85$&$ 0.19$&$ 0.82$&$ 0.86$&$ 0$&$ 0.16$&$ 0.41$&$ 0.5$&$ 0.61$&$ 0.48$&$ 0.28$&$ 0.86$\\
			\rule{0pt}{12pt}
			$x_9$ & $0.91$&$ 0.77$&$ 0$&$ 0.62$&$ 0.25$&$ 0.27$&$ 0.78$&$ 0.73$&$ 0.68$&$ 0.32$&$ 0.35$&$ 0.55$&$ 0.34$&$ 0.65$&$ 0.44$\\
\rule{0pt}{12pt}
			$x_{10}$ & $0$&$ 0.13$&$ 0.94$&$ 0.13$&$ 0.67$&$ 0.27$&$ 0.89$&$ 0.38$&$ 0.45$&$ 0.19$&$ 0.52$&$ 0.12$&$ 0.6$&$ 0$&$ 0.15$\\
\rule{0pt}{12pt}
			$x_{11}$ & $0.33$&$ 0$&$ 0.34$&$ 0.55$&$ 0.07$&$ 0.36$&$ 0.23$&$ 0.51$&$ 0.67$&$ 0.74$&$ 0.06$&$ 0.24$&$ 0.46$&$ 0.22$&$ 0.92$\\
\rule{0pt}{12pt}
			$x_{12}$ & $0.26$&$ 0.91$&$ 0.75$&$ 0$&$ 0.58$&$ 0.38$&$ 0.92$&$ 0.9$&$ 0.65$&$ 0.31$&$ 0.48$&$ 0.31$&$ 0.11$&$ 0.98$&$ 0.63$\\
\rule{0pt}{12pt}
			$x_{13}$ & $0.4$&$ 0.23$&$ 0.28$&$ 0.58$&$ 0.84$&$ 0.82$&$ 0.66$&$ 0.81$&$ 0.16$&$ 0.67$&$ 0 $&$0.12$&$ 0.9$&$ 0.27$&$ 0.34$\\
\rule{0pt}{12pt}
			$x_{14}$ & $0.31$&$ 0.52$&$ 0.9$&$ 0.88$&$ 0.25$&$ 0.37$&$ 0.59$&$ 0$&$ 0.72$&$ 0.07$&$ 0.06$&$ 0.72$&$ 0$&$ 0.89$&$ 0.06$\\
\rule{0pt}{12pt}
			$x_{15}$ & $0.99$&$ 0.63$&$ 0.99$&$ 0.66$&$ 0.08$&$ 0.63$&$ 0.32$&$ 0.82$&$ 0.31$&$ 0.88$&$ 0.44$&$ 0.79$&$ 0.7$&$ 0.55$&$ 0$\\
			\hline
		
		\end{tabular}}
	\end{center}
	
\end{table}
\begin{table}[H]
	\begin{center}\caption{The distance $\mathfrak{D}_{st}^{\downarrow}$ from the attribute value $K_s(x_t)$ to the NIS ${\mathfrak{I}_s}_{\downarrow}$}\label{down}
\resizebox{.95\columnwidth}{!}{
\begin{tabular}{l|lllllllllllllll}
			\hline
			$\mathfrak{D}_{st}^{\downarrow}$ & $K_1$  & $K_2$ & $K_3$ & $K_4$ & $K_5$ & $K_6$ & $K_7$ & $K_8$ & $K_9$ & $K_{10}$ & $K_{11}$ & $K_{12}$ & $K_{13}$ & $K_{14}$ & $K_{15}$ \\
			\hline
			\rule{0pt}{12pt}
			$x_1$ & $0.45$&$	0.59$&$	0.1$&$	0.92$&$	0$&$	0.45$&$	0.98$&$	0.77$&$	0.43$&$	0.38$&$	0$&$	0.2$&$	0.02$&$	0.48$&$	0$\\
			\rule{0pt}{12pt}
			$x_2$ & $0.99$&$	0.04$&$	0.92$&$	0.1$&$	0.51$&$	0.18$&$	0.05$&$	0.68$&$	0.34$&$	0.37$&$	0.25$&$	0.46$&$	0.71$&$	0.59$&$	0.42$\\
			\rule{0pt}{12pt}
			$x_3$ & $0.99$&$	0.63$&$	0.18$&$	0.19$&$	0.73$&$	0.82$&$	0.92$&$	0$&$	0.06$&$	0.76$&$	0.69$&$	0.72$&$	0.22$&$	0.8$&$	0.81$\\
		\rule{0pt}{12pt}
			$x_4$ & $0.32$&$	0.02$&$	0.26$&$	0.33$&$	0.36$&$	0.04$&$	0$&$	0.83$&$	0.66$&$	0.23$&$	0.43$&$	0.79$&$	0.15$&$	0.51$&$	0.6$\\
			\rule{0pt}{12pt}
			$x_5$ & $0.29	$&$0.03	$&$0.79	$&$0	$&$0.54	$&$0.47	$&$0.66	$&$0.91	$&$0.72	$&$0.33	$&$0.55	$&$0.27$&$	0.24$&$	0.18$&$	0.5$\\
\rule{0pt}{12pt}
			$x_6$ & $0.05$&$	0.55$&$	0.48$&$	0.61$&$	0.84$&$	0.42$&$	0.76$&$	0.76$&$	0.45$&$	0.84$&$	0.14$&$	0.55$&$	0.28$&$	0.43$&$	0.84$\\
\rule{0pt}{12pt}
			$x_7$ & $0.29$ & $	0.24$ & $	0.76$ & $	0.33$ & $	0.79$ & $	0.21$ & $	0.51$ & $	0.03$ & $	0.37$ & $	0.88$ & $	0.18$ & $	0.21$ & $	0.45$ & $	0.41$ & $	0.08$\\
\rule{0pt}{12pt}
			$x_8$ & $0.62$ & $	0.65$ & $	0.66$ & $	0.08$ & $	0.65$ & $	0$ & $	0.12$ & $	0.99$ & $	0.56$ & $	0.47$ & $	0.21$ & $	0.18$ & $	0.42$ & $	0.7$ & $	0.12$\\
			\rule{0pt}{12pt}
			$x_9$ & $0.08$ & $	0.14$ & $	0.99$ & $	0.31$ & $	0.59$ & $	0.55$ & $	0.2$ & $	0.26$ & $	0.04$ & $	0.56$ & $	0.36$ & $	0.24$ & $	0.56$ & $	0.33$ & $	0.54$\\
\rule{0pt}{12pt}
			$x_{10}$ & $0.99$ & $	0.78$ & $	0.05$ & $0.8$ & $	0.17$ & $	0.55$ & $	0.09$ & $	0.61$ & $	0.27$ & $	0.69$ & $	0.19$ & $	0.67$ & $	0.3$ & $	0.98$ & $	0.83$\\
\rule{0pt}{12pt}
			$x_{11}$ & $0.66$ & $	0.91$ & $	0.65$ & $	0.38$ & $	0.77$ & $	0.46$ & $	0.75$ & $	0.48$ & $	0.05$ & $	0.14$ & $	0.65$ & $	0.55$ & $	0.44$ & $	0.76$ & $	0.06$\\
\rule{0pt}{12pt}
			$x_{12}$ & $0.73$ & $	0$ & $	0.24$ & $0	0.93$ & $	0.26$ & $	0.44$ & $ 0.06$ & $ 0.09$ & $	0.07$ & $	0.57$ & $	0.23$ & $	0.48$ & $	0.79$ & $	0	$ & $0.35$\\
\rule{0pt}{12pt}
			$x_{13}$ & $0.59$ & $0.68$ & $	0.71$ & $	0.35$ & $	0$ & $	0$ & $	0.32$ & $	0.18$ & $	0.56$ & $	0.21$ & $	0.71$ & $	0.67$ & $	0$ & $	0.71$ & $	0.64$\\
\rule{0pt}{12pt}
			$x_{14}$ & $0.68$ & $	0.39$ & $	0.09$ & $	0.05$ & $	0.59$ & $	0.45$ & $	0.39$ & $	0.99$ & $0$ & $	0.81$ & $	0.65$ & $0.07$ & $	0.9$ & $	0.09$ & $	0.92$\\
\rule{0pt}{12pt}
			$x_{15}$ & $0$ & $	0.28$ & $	0$ & $	0.27$ & $0.76$ & $	0.19$ & $	0.66$ & $ 0.17$ & $	0.41$ & $	0$ & $	0.27$ & $	0$ & $	0.2$ & $	0.43$ & $	0.98$\\
			\hline
		
		\end{tabular}}
	\end{center}

\end{table}

{\bf Step 3:} In light of the model $(\underline{\mathfrak{C}}_{A1}(X)$, $\overline{\mathfrak{C}}_{A1}(X))$ and the model $(\underline{\mathcal{C}}_{A1}(X)$, $\overline{\mathcal{C}}_{A1}(X))$, the different based attribute weight vectors $\mathfrak{W}$ and $\mathcal{W}$ can be obtained. In which, the non-associative overlap function $O_D(x,y)=x^2y^2$ that satisfies the condition (O7) and its $R_O$-implication  $I_{O_{D}}=\left\{\begin{matrix}
		1, & \text{if}$\ $ x^2\leq y;\\
		\sqrt{\frac{y}{x^2}}, & \text{if}$\ $ x^2>y.
	\end{matrix}\right.$ are used to construct the ONRFRS models, and the left-continuous $t$-norm $\mathscr{T}_D=xy$ and its $R$-implication $\mathscr{I}_{\mathscr{T}_{D}}=\left\{\begin{matrix}
		1, & \text{if}$\ $ x\leq y;\\
		\frac{y}{x}, & \text{if}$\ $ x>y.
	\end{matrix}\right.$ are used to construct the TNRFRS models. First of all, we calculate diverse types of approximate space for each attribute set, and the low approximation sets are shown in Table \ref{ol} and Table \ref{tl}, respectively. Besides, the upper approximation sets are shown in Table \ref{ou} and Table \ref{tu} as well.
\begin{sidewaystable}[thp]
	\begin{center}\caption{The overlap function-based lower approximation of attribute fuzzy set $K_s(x_t)$}\label{ol}
\resizebox{.95\columnwidth}{!}{
\begin{tabular}{l|lllllllllllllll}
			\hline
			$\quad$ & $x_1$  & $x_2$ & $x_3$ & $x_4$ & $x_5$ & $x_6$ & $x_7$ & $x_8$ & $x_9$ & $x_{10}$ & $x_{11}$ & $x_{12}$ & $x_{13}$ & $x_{14}$ & $x_{15}$ \\
			\hline
			\rule{0pt}{10pt}
			$\underline{\mathfrak{C}}_{A_1}(K_1(x_t))$ & $0.4600$ & $	0.4412$ & $	0.5188$ & $	0.2577$ & $	0.2193$ & $	0.0600$ & $	0.1753$ & $	0.5061$ & $	0.0900$ & $	0.6000$ & $	0.5522$ & $	0.3675$ & $	0.5000$ & $	0.4169$ & $	0.0100$\\
			\rule{0pt}{16pt}
			$\underline{\mathfrak{C}}_{A_1}(K_2(x_t))$ & $0.5309$ & $	0.1300$ & $	0.5188$ & $	0.1100$ & $	0.1200$ & $	0.3986$ & $	0.2300$ & $	0.5127$ & $	0.1780$ & $	0.5855$ & $	0.5522$ & $	0.0900$ & $	0.5000$ & $	0.3546$ & $	0.2788$ \\
			\rule{0pt}{16pt}
			$\underline{\mathfrak{C}}_{A_1}(K_3(x_t))$ & $0.1100$ & $	0.4412$ & $	0.1900$ & $	0.2110$ & $	0.4900$ & $	0.4900$ & $	0.4339$ & $	0.5061$ & $	0.4803$ & $	0.0600$ & $	0.5698$ & $	0.2500$ & $	0.4736$ & $	0.1000$ & $	0.0100$\\
		\rule{0pt}{16pt}
			$\underline{\mathfrak{C}}_{A_1}(K_4(x_t))$ & $0.4411$ & $	0.1700$ & $	0.2600$ & $	0.1814$ & $	0.0700$ & $	0.3986$ & $	0.2940$ & $	0.1500$ & $	0.1780$ & $	0.4200$ & $	0.4500$ & $	0.4000$ & $	0.4200$ & $	0.1200$ & $	0.2629$\\
			\rule{0pt}{16pt}
			$\underline{\mathfrak{C}}_{A_1}(K_5(x_t))$ & $0.1600$ & $	0.5200$ & $	0.5757$ & $	0.3300$ & $	0.5262$ & $	0.5747$ & $	0.5789$ & $	0.5061$ & $	0.5200$ & $	0.3300$ & $	0.5698$ & $	0.3300$ & $	0.1600$ & $0.3675$ & $	0.5122$\\
\rule{0pt}{16pt}
			$\underline{\mathfrak{C}}_{A_1}(K_6(x_t))$ & $0.4411$ & $	0.2062$ & $	0.5757$ & $	0.2200$ & $	0.4497$ & $	0.5649$ & $	0.2940$ & $	0.1800$ & $	0.3600$ & $	0.4188$ & $	0.5140$ & $	0.3900$ & $	0.1800$ & $0.3600$ & $	0.3700$\\
\rule{0pt}{16pt}
			$\underline{\mathfrak{C}}_{A_1}(K_7(x_t))$ & $0.4411$ & $	0.0700$ & $	0.5757$ & $	0.0200$ & $0.4352$ & $	0.3986$ & $	0.2200$ & $	0.1400$ & $	0.1780$ & $	0.1100$ & $	0.5140$ & $	0.0800$ & $	0.3400$ & $	0.3546$ & $	0.3102$\\
\rule{0pt}{16pt}
			$\underline{\mathfrak{C}}_{A_1}(K_8(x_t))$ & $0.4411$ & $	0.4900$ & $	0.0100$ & $	0.3982$ & $	0.4352$ & $	0.3986$ & $	0.0400$ & $	0.4900$ & $	0.2700$ & $	0.4188$ & $	0.4900$ & $	0.1000$ & $	0.1900$ & $	0.5652$ & $	0.1800$\\
			\rule{0pt}{16pt}
			$\underline{\mathfrak{C}}_{A_1}(K_9(x_t))$ & $0.3300$ & $	0.3571$ & $	0.3400$ & $	0.3982$ & $	0.4352$ & $	0.3986$ & $	0.3200$ & $0.3300$ & $	0.3200$ & $	0.5500$ & $0.3300$ & $	0.3300$ & $	0.4736$ & $	0.2800$ & $	0.2800$\\
\rule{0pt}{16pt}
			$\underline{\mathfrak{C}}_{A_1}(K_{10}(x_t))$ & $0.3224$ & $0.3571$ & $	0.5188$ & $	0.3500$ & $	0.4500$ & $	0.5900$ & $	0.4500$ & $	0.3161$ & $	0.4500$ & $	0.4188$ & $	0.2600$ & $	0.3291$ & $	0.3300$ & $	0.4500$ & $	0.1200$\\
\rule{0pt}{16pt}
			$\underline{\mathfrak{C}}_{A_1}(K_{11}(x_t))$ & $0.2900$ & $	0.5000$ & $	0.6019$ & $ 0.4300$ & $	0.4300$ & $	0.4300$ & $	0.4300$ & $	0.5000$ & $	0.4300$ & $0	0.4800$ & $	0.5140$ & $	0.4700$ & $	0.5000$ & $	0.4800$ & $	0.4300$\\
\rule{0pt}{16pt}
			$\underline{\mathfrak{C}}_{A_1}(K_{12}(x_t))$ & $0.4100$ & $0.3900$ & $	0.5188$ & $	0.3965$ & $	0.4497$ & $	0.4500$ & $	0.3900$ & $	0.3900$ & $0.4188$ & $	0.5598$ & $	0.5140$ & $	0.4200$ & $	0.4330$ & $	0.2800$ & $	0.2100$\\
\rule{0pt}{16pt}
			$\underline{\mathfrak{C}}_{A_1}(K_{13}(x_t))$ & $0.1200$ & $	0.4412$ & $	0.3200$ & $	0.2500$ & $	0.3400$ & $	0.3800$ & $	0.3400$ & $	0.5061$ & $	0.3400$ & $	0.4000$ & $	0.5200$ & $	0.4000$ & $	0.1000$ & $	0.4000$ & $	0.3000$\\
\rule{0pt}{16pt}
			$\underline{\mathfrak{C}}_{A_1}(K_{14}(x_t))$ & $0.4675$ & $	0.4522$ & $	0.5188$ & $ 0.3385$ & $	0.2000$ & $	0.3986$ & $	0.3385$ & $	0.5061$ & $	0.3037$ & $	0.6100$ & $	0.5522$ & $	0.0200$ & $	0.5000$ & $	0.1100$ & $	0.2629$\\
\rule{0pt}{16pt}
			$\underline{\mathfrak{C}}_{A_1}(K_{15}(x_t))$ & $0.0200$ & $0.2062$ & $	0.5200$ & $	0.3965$ & $0.4497$ & $	0.5600$ & $	0.1000$ & $	0.1400$ & $	0.4188$ & $	0.5598$ & $	0.0800$ & $	0.3291$ & $	0.4330$ & $	0.4400$ & $	0.5200$\\
			\hline
		
		\end{tabular}}
	\end{center}
	
\end{sidewaystable}
\begin{sidewaystable}[thp]
	\begin{center}\caption{The $t$-norm-based lower approximation of attribute fuzzy set $K_s(x_t)$}\label{tl}
\resizebox{.95\columnwidth}{!}{
\begin{tabular}{l|lllllllllllllll}
			\hline
			$\quad$ & $x_1$  & $x_2$ & $x_3$ & $x_4$ & $x_5$ & $x_6$ & $x_7$ & $x_8$ & $x_9$ & $x_{10}$ & $x_{11}$ & $x_{12}$ & $x_{13}$ & $x_{14}$ & $x_{15}$ \\
			\hline
			\rule{0pt}{10pt}
			$\underline{\mathcal{C}}_{A_1}(K_1(x_t))$ & $0.4600$ & $	0.6818$ & $	0.8202$ & $	0.3300$ & $	0.3000$ & $	0.0600$ & $	0.3000$ & $	0.6300$ & $	0.0900$ & $	0.7534$ & $	0.6700$ & $	0.7297$ & $	0.6000$ & $0.6900$ & $0.0100$\\
			\rule{0pt}{16pt}
			$\underline{\mathcal{C}}_{A_1}(K_2(x_t))$ & $0.6800$ & $	0.1300$ & $	0.7200$ & $	0.1100$ & $	0.1200$ & $	0.6400$ & $	0.3300$ & $	0.7400$ & $	0.2300$ & $	0.7700$ & $	0.8181$ & $	0.0900$ & $	0.7700$ & $	0.4800$ & $	0.3700$ \\
			\rule{0pt}{16pt}
			$\underline{\mathcal{C}}_{A_1}(K_3(x_t))$ & $0.1100$ & $0.6818$ & $	0.1900$ & $	0.2700$ & $	0.7934$ & $	0.4900$ & $	0.7358$ & $	0.6600$ & $	0.5100$ & $	0.0600$ & $	0.6600$ & $	0.2500$ & $	0.7200$ & $	0.1000$ & $	0.0100$\\
		\rule{0pt}{16pt}
			$\underline{\mathcal{C}}_{A_1}(K_4(x_t))$ & $0.5454$ & $	0.1700$ & $	0.2600$ & $	0.4000$ & $	0.0700$ & $0.6800$ & $	0.4000$ & $	0.1500$ & $	0.3800$ & $	0.7534$ & $	0.4500$ & $	0.7297$ & $	0.4200$ & $	0.1200$ & $	0.3400$\\
			\rule{0pt}{16pt}
			$\underline{\mathcal{C}}_{A_1}(K_5(x_t))$ & $0.1600$ & $0.6700$ & $	0.8202$ & $	0.5200$ & $	0.7000$ & $	0.7500$ & $	0.7500$ & $	0.8100$ & $	0.6818$ & $	0.3300$ & $	0.8181$ & $	0.4200$ & $	0.1600$ & $	0.7317$ & $	0.7500$\\
\rule{0pt}{16pt}
			$\underline{\mathcal{C}}_{A_1}(K_6(x_t))$ & $0.6300$ & $	0.3600$ & $	0.8202$ & $	0.2200$ & $	0.6500$ & $	0.6000$ & $0.3900$ & $	0.1800$ & $	0.6000$ & $	0.7300$ & $	0.6400$ & $	0.6200$ & $	0.1800$ & $	0.6300$ & $	0.3700$\\
\rule{0pt}{16pt}
			$\underline{\mathcal{C}}_{A_1}(K_7(x_t))$ & $0.7564$ & $	0.0700$ & $	0.8202$ & $	0.0200$ & $0.6800$ & $0.7179$ & $	0.5300$ & $0.1400$ & $	0.2200$ & $	0.1100$ & $	0.7700$ & $	0.0800$ & $	0.3400$ & $	0.4100$ & $	0.6764$\\
\rule{0pt}{16pt}
			$\underline{\mathcal{C}}_{A_1}(K_8(x_t))$ & $0.5454$ & $	0.6900$ & $	0.0100$ & $	0.6900$ & $	0.7065$ & $	0.7179$ & $	0.0400$ & $	0.6071$ & $	0.2700$ & $	0.6200$ & $	0.4900$ & $	0.1000$ & $	0.1900$ & $	0.7317$ & $	0.1800$\\
\rule{0pt}{16pt}
			$\underline{\mathcal{C}}_{A_1}(K_9(x_t))$ & $0.5454$ & $	0.6200$ & $	0.3400$ & $	0.6200$ & $0.7065$ & $	0.7179$ & $	0.6500$ & $	0.6071$ & $	0.3200$ & $	0.5500$ & $0.3300$ & $	0.3500$ & $	0.8311$ & $	0.2800$ & $	0.6764$\\
\rule{0pt}{16pt}
			$\underline{\mathcal{C}}_{A_1}(K_{10}(x_t))$ & $0.5000$ & $	0.4900$ & $	0.8202$ & $	0.3500$ & $0.4500$ & $	0.7179$ & $	0.7000$ & $	0.5900$ & $	0.6800$ & $	0.7534$ & $	0.2600$ & $	0.6900$ & $	0.3300$ & $	0.7500$ & $	0.1200$\\
\rule{0pt}{16pt}
			$\underline{\mathcal{C}}_{A_1}(K_{11}(x_t))$ & $0.2900$ & $	0.5400$ & $	0.8400$ & $	0.5750$ & $	0.7065$ & $	0.4300$ & $	0.4700$ & $	0.5000$ & $	0.5100$ & $	0.4800$ & $	0.8181$ & $	0.5200$ & $	0.7878$ & $	0.7317$ & $	0.4300$\\
\rule{0pt}{16pt}
			$\underline{\mathcal{C}}_{A_1}(K_{12}(x_t))$ & $0.4100$ & $	0.6700$ & $	0.8333$ & $ 0.6700$ & $	0.4800$ & $	0.7179$ & $	0.4200$ & $	0.3900$ & $	0.4500$ & $	0.8352$ & $	0.7600$ & $	0.6900$ & $	0.7878$ & $	0.2800$ & $	0.2100$\\
\rule{0pt}{16pt}
			$\underline{\mathcal{C}}_{A_1}(K_{13}(x_t))$ & $0.1200$ & $	0.6818$ & $ 0.3200$ & $	0.2500$ & $	0.3400$ & $	0.3800$ & $	0.5500$ & $	0.5200$ & $	0.5100$ & $	0.4000$ & $	0.5400$ & $ 0.7297$ & $	0.1000$ & $	0.7317$ & $	0.3000$\\
\rule{0pt}{16pt}
			$\underline{\mathcal{C}}_{A_1}(K_{14}(x_t))$ & $0.5000$ & $	0.6100$ & $	0.8200$ & $	0.5300$ & $	0.2000$ & $	0.4500$ & $	0.4300$ & $	0.7200$ & $	0.3500$ & $	0.7534$ & $	0.7800$ & $	0.0200$ & $	0.7300$ & $	0.1100$ & $0.4500$\\
\rule{0pt}{16pt}
			$\underline{\mathcal{C}}_{A_1}(K_{15}(x_t))$ & $0.0200$ & $0.4400$ & $	0.8202$ & $	0.5750$ & $	0.5200$ & $	0.7179$ & $	0.1000$ & $	0.1400$ & $	0.5600$ & $	0.7534$ & $	0.0800$ & $	0.3700$ & $	0.6600$ & $	0.7500$ & $	0.6764$\\
			\hline
		
		\end{tabular}}
	\end{center}
	
\end{sidewaystable}
\begin{sidewaystable}[thp]
	\begin{center}\caption{The overlap function-based upper approximation of attribute fuzzy set $K_s(x_t)$}\label{ou}
\resizebox{.95\columnwidth}{!}{
\begin{tabular}{l|lllllllllllllll}
			\hline
			$\quad$ & $x_1$  & $x_2$ & $x_3$ & $x_4$ & $x_5$ & $x_6$ & $x_7$ & $x_8$ & $x_9$ & $x_{10}$ & $x_{11}$ & $x_{12}$ & $x_{13}$ & $x_{14}$ & $x_{15}$ \\
			\hline
			\rule{0pt}{10pt}
			$\overline{\mathfrak{C}}_{A_1}(K_1(x_t))$ & $0.6700$ & $	1$ & $	1$ & $	0.8185$ & $	0.5502$ & $	0.5094$ & $	0.6300$ & $	0.6700$ & $	0.8219$ & $	1$ & $	0.6700$ & $	0.7400$ & $	0.6000$ & $	0.6900$ & $	0.6900$\\
\rule{0pt}{16pt}
			$\overline{\mathfrak{C}}_{A_1}(K_2(x_t))$ & $0.6800$ & $	0.7400$ & $	0.7200$ & $	0.7889$ & $	0.6400$ & $	0.6400$ & $	0.7059$ & $	0.7400$ & $	0.6400$ & $	0.8700$ & $	1$ & $	0.7106$ & $	0.7700$ & $ 0.6324$ & $	0.6400$  \\
\rule{0pt}{16pt}
			$\overline{\mathfrak{C}}_{A_1}(K_3(x_t))$ & $0.6600$ & $	0.9300$ & $	0.4811$ & $	0.8185$ & $	0.8000$ & $	0.6013$ & $	0.8246$ & $	0.6700$ & $	1	$ & $0.5811$ & $	0.6600$ & $	0.6600$ & $	0.7200$ & $	0.6453$ & $	0.7211$\\
\rule{0pt}{16pt}
			$\overline{\mathfrak{C}}_{A_1}(K_4(x_t))$ & $0.9900$ & $	0.4500$ & $	0.4200$ & $	0.7889$ & $	0.6800$ & $	0.6800$ & $	0.6800$ & $	0.4938$ & $	0.6800$ & $	0.8700$ & $	0.4500$ & $	1$ & $	0.4285$ & $	0.6324$ & $	0.6800$\\
\rule{0pt}{16pt}
			$\overline{\mathfrak{C}}_{A_1}(K_5(x_t))$ & $0.6775$ & $	0.7937$ & $	0.8900$ & $	0.7422$ & $	0.7806$ & $	1$ & $	0.9500$ & $	0.8100$ & $	0.7500$ & $	0.4401$ & $	0.9300$ & $	0.6708$ & $	0.5669$ & $	0.7500$ & $0.9273$\\
\rule{0pt}{16pt}
			$\overline{\mathfrak{C}}_{A_1}(K_6(x_t))$ & $0.6400$ & $	0.6400$ & $	1	$ & $0.7300$ & $	0.6500$ & $	0.6013$ & $	0.7300$ & $	0.6400$ & $	0.7300$ & $	0.7300$ & $	0.6400$ & $	0.7106$ & $	0.5263$ & $	0.6324$ & $	0.6897$\\
\rule{0pt}{16pt}
			$\overline{\mathfrak{C}}_{A_1}(K_7(x_t))$ & $1$ & $	0.6428$ & $0.9400$ & $	0.7422$ & $	0.7800$ & $	0.7800$ & $	0.7800$ & $	0.6838$ & $	0.7000$ & $	0.3400$ & $	0.7700$ & $	0.6708$ & $	0.5263$ & $	0.5830$ & $	0.7800$\\
\rule{0pt}{16pt}
			$\overline{\mathfrak{C}}_{A_1}(K_8(x_t))$ & $0.7800$ & $0.7937$ & $	0.4811$ & $	0.8400$ & $	0.9200$ & $	0.7700$ & $0.7700$ & $	1$ & $	0.7000$ & $	0.6200$ & $	0.4900$ & $	0.6200$ & $	0.5669$ & $	1$ & $	0.7700$\\
\rule{0pt}{16pt}
			$\overline{\mathfrak{C}}_{A_1}(K_9(x_t))$ & $0.7100$ & $	0.7937$ & $	0.4811$ & $	0.9400$ & $	1$ & $	0.7300$ & $	0.7300$ & $	0.8400$ & $	0.7000$ & $	0.5811$ & $	0.4859$ & $	0.6324$ & $	0.8400$ & $ 0.6200$ & $	0.7300$\\
\rule{0pt}{16pt}
			$\overline{\mathfrak{C}}_{A_1}(K_{10}(x_t))$ & $0.5324$ & $	0.5900$ & $	0.8800$ & $	0.7889$ & $	0.7806$ & $	0.9600$ & $	1$ & $	0.5900$ & $	0.7000$ & $	0.8100$ & $	0.4859$ & $	0.7106$ & $	0.5669$ & $	0.9300$ & $	0.9273$\\
\rule{0pt}{16pt}
			$\overline{\mathfrak{C}}_{A_1}(K_{11}(x_t))$ & $0.6775$ & $	0.6428$ & $	0.9800$ & $	0.7200$ & $	0.8400$ & $	0.6013$ & $	0.6614$ & $	0.6838$ & $	0.6962$ & $	0.5811$ & $	0.9400$ & $	0.6708$ & $	1	$ & $ 0.9400$ & $	0.7370$\\
\rule{0pt}{16pt}
			$\overline{\mathfrak{C}}_{A_1}(K_{12}(x_t))$ & $0.6775$ & $	0.6700$ & $	0.9300$ & $	1	$ & $ 0.7600$ & $	0.7600$ & $	0.7600$ & $	0.6838$ & $	0.7000$ & $	0.8800$ & $	0.7600$ & $	0.7106$ & $	0.8800$ & $	0.6453$ & $	0.7600$\\
\rule{0pt}{16pt}
			$\overline{\mathfrak{C}}_{A_1}(K_{13}(x_t))$ & $0.5400$ & $	0.8100$ & $	0.4044$ & $	0.8100$ & $	0.5647$ & $	0.6013$ & $0.6600$ & $	0.5400$ & $	0.8100$ & $	0.4401$ & $	0.5400$ & $	0.8900$ & $	0.5200$ & $	1$ & $	0.7370$\\
\rule{0pt}{16pt}
			$\overline{\mathfrak{C}}_{A_1}(K_{14}(x_t))$ & $0.6775$ & $	0.7200$ & $	0.8200$ & $	0.7889$ & $	0.5502$ & $	0.4500$ & $	0.7059$ & $	0.7200$ & $	0.6100$ & $	1	$ & $0.7800$ & $	0.7106$ & $	0.7300$ & $	0.6324$ & $	0.5555$\\
\rule{0pt}{16pt}
			$\overline{\mathfrak{C}}_{A_1}(K_{15}(x_t))$ & $0.5588$ & $	0.5587$ & $	0.8300$ & $	0.7889$ & $	0.7806$ & $	0.8600$ & $	0.8164$ & $	0.4761$ & $	0.7000$ & $	0.8500$ & $	0.4477$ & $	0.7106$ & $	0.6600$ & $	0.9400$ & $	1$\\
			\hline
		
		\end{tabular}}
	\end{center}
	
\end{sidewaystable}
\begin{sidewaystable}[thp]
	\begin{center}\caption{The $t$-norm-based upper approximation of attribute fuzzy set $K_s(x_t)$}\label{tu}
\resizebox{.95\columnwidth}{!}{
\begin{tabular}{l|lllllllllllllll}
			\hline
			$\quad$ & $x_1$  & $x_2$ & $x_3$ & $x_4$ & $x_5$ & $x_6$ & $x_7$ & $x_8$ & $x_9$ & $x_{10}$ & $x_{11}$ & $x_{12}$ & $x_{13}$ & $x_{14}$ & $x_{15}$ \\
			\hline
			\rule{0pt}{10pt}
			$\overline{\mathcal{C}}_{A_1}(K_1(x_t))$ & $0.4600$ & $	1$ & $	1$ & $	0.4250$ & $	0.3000$ & $	0.1764$ & $	0.3000$ & $	0.6300$ & $	0.3181$ & $1	$ & $0.6700$ & $	0.7400$ & $	0.6000$ & $ 0.6900$ & $	0.2444$\\
			\rule{0pt}{16pt}
			$\overline{\mathcal{C}}_{A_1}(K_2(x_t))$ & $0.6800$ & $	0.3181$ & $	0.7200$ & $	0.2300$ & $	0.2300$ & $	0.6400$ & $	0.3300$ & $	0.7400$ & $	0.4900$ & $	0.8700$ & $	1	$ & $0.2702$ & $ 0.7700$ & $	0.4800$ & $	0.6400$ \\
		\rule{0pt}{16pt}
			$\overline{\mathcal{C}}_{A_1}(K_3(x_t))$ & $0.4545$ & $	0.9300$ & $	0.1900$ & $	0.4250$ & $	0.8000$ & $	0.4900$ & $	0.7700$ & $	0.6700$ & $	1	$ & $0.2465$ & $	0.6600$ & $	0.2702$ & $	0.7200$ & $	0.2500$ & $	0.4900$\\
		\rule{0pt}{16pt}
			$\overline{\mathcal{C}}_{A_1}(K_4(x_t))$ & $0.9900$ & $	0.2857$ & $	0.2600$ & $	0.4000$ & $	0.2934$ & $	0.6800$ & $	0.4000$ & $	0.3928$ & $	0.4900$ & $	0.8700$ & $	0.4500$ & $	1$ & $	0.4200$ & $	0.2682$ & $	0.6800$\\
		\rule{0pt}{16pt}
			$\overline{\mathcal{C}}_{A_1}(K_5(x_t))$ & $0.4545$ & $	0.6700$ & $	0.8900$ & $	0.5200$ & $	0.7000$ & $	1	$ & $0.9500$ & $	0.8100$ & $	0.7500$ & $	0.3300$ & $	0.9300$ & $	0.4200$ & $	0.2121$ & $	0.7500$ & $	0.9200$\\
\rule{0pt}{16pt}
			$\overline{\mathcal{C}}_{A_1}(K_6(x_t))$ & $0.6300$ & $	0.3600$ & $	1$ & $	0.3600$ & $ 0.6500$ & $	0.6000$ & $	0.3900$ & $	0.3928$ & $	0.7300$ & $	0.7300$ & $	0.6400$ & $0.6200$ & $	0.1800$ & $ 0.6300$ & $	0.6000$\\
\rule{0pt}{16pt}
			$\overline{\mathcal{C}}_{A_1}(K_7(x_t))$ & $1$ & $	0.3000$ & $	0.9400$ & $	0.2261$ & $	0.6800$ & $	0.7800$ & $	0.5300$ & $	0.3928$ & $	0.4900$ & $	0.2465$ & $	0.7700$ & $	0.2702$ & $	0.3400$ & $ 0.4100$ & $	0.7678$\\
\rule{0pt}{16pt}
			$\overline{\mathcal{C}}_{A_1}(K_8(x_t))$ & $0.7800$ & $	0.6900$ & $	0.1797$ & $	0.8400$ & $	0.9200$ & $	0.7700$ & $	0.2700$ & $	1$ & $	0.4900$ & $	0.6200$ & $	0.4900$ & $	0.2500$ & $	0.2121$ & $	1	$ & $ 0.7678$\\
\rule{0pt}{16pt}
			$\overline{\mathcal{C}}_{A_1}(K_9(x_t))$ & $0.7100$ & $	0.6200$ & $	0.3400$ & $	0.9400$ & $ 1	$ & $ 0.7300$ & $	0.6500$ & $	0.8400$ & $	0.4900$ & $	0.5500$ & $	0.3300$ & $	0.3500$ & $	0.8400$ & $	0.2800$ & $	0.7300$\\
\rule{0pt}{16pt}
			$\overline{\mathcal{C}}_{A_1}(K_{10}(x_t))$ & $0.5000$ & $	0.4900$ & $	0.8800$ & $	0.4250$ & $	0.4500$ & $	0.9600$ & $	1	$ & $0.5900$ & $	0.6800$ & $	0.8100$ & $0.2600$ & $	0.6900$ & $	0.3300$ & $	0.9300$ & $	0.7678$\\
\rule{0pt}{16pt}
			$\overline{\mathcal{C}}_{A_1}(K_{11}(x_t))$ & $0.4545$ & $	0.5400$ & $	0.9800$ & $	0.7200$ & $	0.8400$ & $	0.4300$ & $	0.4700$ & $	0.5000$ & $	0.6500$ & $	0.4800$ & $0.9400$ & $	0.5200$ & $	1	$ & $ 0.9400$ & $	0.5600$\\
\rule{0pt}{16pt}
			$\overline{\mathcal{C}}_{A_1}(K_{12}(x_t))$ & $0.4545$ & $	0.6700$ & $	0.9300$ & $	1	$ & $ 0.4800$ & $	0.7600$ & $	0.4200$ & $	0.3928$ & $	0.4900$ & $	0.8800$ & $	0.7600$ & $	0.6900$ & $	0.8800$ & $	0.2800$ & $	0.7600$\\
\rule{0pt}{16pt}
			$\overline{\mathcal{C}}_{A_1}(K_{13}(x_t))$ & $0.4545$ & $0.8100$ & $	0.3200$ & $	0.4250$ & $	0.3400$ & $	0.3800$ & $	0.5500$ & $	0.5200$ & $	0.6600$ & $	0.4000$ & $	0.5400$ & $	0.8900$ & $	0.2121$ & $	1$ & $	0.3800$\\
\rule{0pt}{16pt}
			$\overline{\mathcal{C}}_{A_1}(K_{14}(x_t))$ & $0.5000$ & $	0.6100$ & $	0.8200$ & $	0.5300$ & $	0.2934$ & $	0.4500$ & $	0.4300$ & $	0.7200$ & $	0.4500$ & $	1$ & $	0.7800$ & $	0.2702$ & $	0.7300$ & $	0.2682$ & $	0.4500$\\
\rule{0pt}{16pt}
			$\overline{\mathcal{C}}_{A_1}(K_{15}(x_t))$ & $0.2435$ & $	0.4400$ & $	0.8300$ & $	0.6200$ & $	0.5200$ & $	0.8600$ & $	0.3000$ & $	0.1900$ & $	0.5600$ & $	0.8500$ & $	0.1720$ & $	0.3700$ & $	0.6600$ & $	0.9400$ & $	1$\\
			\hline
		
		\end{tabular}}
	\end{center}
	
\end{sidewaystable}
Then the approximate precisions of the attribute set $K_s$, $\mathfrak{A}_{(\underline{\mathfrak{C}}_{A1},\overline{\mathfrak{C}}_{A1})}(K_s)$ and $\mathcal{A}_{(\underline{\mathcal{C}}_{A1},\overline{\mathcal{C}}_{A1})}(K_s)$, can be calculated, respectively. The results are depicted in the Table \ref{app}.
\begin{table}[H]
	\begin{center}\caption{The approximate precisions of the attribute set $K_s$}\label{app}
\resizebox{.95\columnwidth}{!}{
\begin{tabular}{l|lllllllllllllll}
			\hline
			$\quad$ & $K_1$  & $K_2$ & $K_3$ & $K_4$ & $K_5$ & $K_6$ & $K_7$ & $K_8$ & $K_9$ & $K_{10}$ & $K_{11}$ & $K_{12}$ & $K_{13}$ & $K_{14}$ & $K_{15}$ \\
			\hline
			\rule{0pt}{16pt}
			$\mathfrak{A}_{(\underline{\mathfrak{C}}_{A1},\overline{\mathfrak{C}}_{A1})}(K_s)$ & $0.4679$ & $	0.4662$ & $	0.4470$ & $	0.4248$ & $	0.5617$ & $	0.5368$ & $	0.3906$ & $	0.4421$ & $	0.5060$ & $	0.5076$ & $	0.6081$ & $	0.5381$ & $	0.5226$ & $	0.5338$ & $	0.4712$\\\hline
			\rule{0pt}{16pt}
			$\mathcal{A}_{(\underline{\mathcal{C}}_{A1},\overline{\mathcal{C}}_{A1})}(K_s)$ & $0.8329$ & $0.8322$ & $	0.7459$ & $	0.7447$ & $	0.8802$ & $	0.8951$ & $	0.7786$ & $	0.7100$ & $	0.8664$ & $	0.8400$ & $	0.8608$ & $	0.8737$ & $	0.8213$ & $	0.8977$ & $	0.8395$  \\
			
			\hline
		
		\end{tabular}}
	\end{center}
	
\end{table}
With the function \ref{wf1} and the function \ref{wf2}, two kinds of weights for each attribute can be given, and the details are shown in the Table \ref{weight}.
\begin{table}[H]
	\begin{center}\caption{The attribute weight vectors}\label{weight}
\resizebox{.95\columnwidth}{!}{
\begin{tabular}{l|lllllllllllllll}
			\hline
			$\quad$ & $K_1$  & $K_2$ & $K_3$ & $K_4$ & $K_5$ & $K_6$ & $K_7$ & $K_8$ & $K_9$ & $K_{10}$ & $K_{11}$ & $K_{12}$ & $K_{13}$ & $K_{14}$ & $K_{15}$ \\
			\hline
			\rule{0pt}{16pt}
			$\mathfrak{W}_s$ & $0.0630$ & $	0.0627$ & $	0.0602$ & $	0.0572$ & $	0.0756$ & $	0.0723$ & $	0.0526$ & $	0.0595$ & $	0.0681$ & $	0.0683$ & $	0.0818$ & $	0.0724$ & $	0.0703$ & $	0.0718$ & $	0.0634$\\
\hline
			\rule{0pt}{16pt}
			$\mathcal{W}_s$ & $0.0670$ & $	0.0670$ & $	0.0600$ & $	0.0599$ & $	0.0708$ & $	0.0720$ & $	0.0626$ & $	0.0571$ & $	0.0697$ & $	0.0676$ & $	0.0693$ & $	0.0703$ & $	0.0661$ & $	0.0722$ & $	0.0675$  \\
			
			\hline
		
		\end{tabular}}
	\end{center}
	
\end{table}
{\bf Step4:} For the final step, we need to calculate the closeness coefficients of alternative $x_t$, which can help us to determine the ranking of objects. By the function \ref{func14} and \ref{func15}, we obtain that $\mathfrak{H}^{\uparrow}=0.3106$, $\mathfrak{H}^{\downarrow}=0.5757 $ and $\mathcal{H}^{\uparrow}=0.3099$, $\mathcal{H}^{\downarrow}=0.5805 $. Then, the results are illustrated in the Table \ref{endd}.
\begin{table}[H]
	\begin{center}\caption{The closeness coefficient}\label{endd}
\resizebox{.95\columnwidth}{!}{
\begin{tabular}{l|lllllllllllllll}
			\hline
			$\quad$ & $x_1$  & $x_2$ & $x_3$ & $x_4$ & $x_5$ & $x_6$ & $x_7$ & $x_8$ & $x_9$ & $x_{10}$ & $x_{11}$ & $x_{12}$ & $x_{13}$ & $x_{14}$ & $x_{15}$ \\
			\hline
			\rule{0pt}{12pt}
			$\mathfrak{H}_t$ & $-1.0706$ & $	-0.6625$ & $	0$ & $	-0.9999$ & $	-0.7258$ & $	-0.2563$ & $	-0.9546$ & $	-0.7387$ & $	-0.9359$ & $	-0.2376$ & $	-0.3026$ & $ -1.1245$ & $-0.7663$ & $-0.5034$ & $	-1.3351$\\\hline
			\rule{0pt}{12pt}
			$\mathcal{H}_t$ & $-1.0194$ & $	-0.7060$ & $	0$ & $	-1.0500$ & $	-0.7609$ & $	-0.2525$ & $	-0.9768$ & $	-0.7731$ & $	-0.9848$ & $	-0.2281$ & $	-0.3290$ & $	-1.1547$ & $	-0.7781$ & $-0.5549$ & $	-1.3317$  \\
			
			\hline
		
		\end{tabular}}
	\end{center}
	
\end{table}
In line with the above steps, the rankings based $(\underline{\mathfrak{C}}_{A1}(X)$, $\overline{\mathfrak{C}}_{A1}(X))$ and $(\underline{\mathcal{C}}_{A1}(X)$, $\overline{\mathcal{C}}_{A1}(X))$ of all alternative targets can be acquired, respectively. The orders are shown as below. Meanwhile, the sorting results based on the other models are given in the Table \ref{ranking1}.

{\bf On the basis of $(\underline{\mathfrak{C}}_{A1}(X)$, $\overline{\mathfrak{C}}_{A1}(X))$:} $$x_3>x_{10}>x_6>x_{11}>x_{14}>x_2>x_5>x_8>x_{13}>x_{9}>x_{7}>x_4>x_1>x_{12}>x_{15}.$$

{\bf On the basis of $(\underline{\mathcal{C}}_{A1}(X)$, $\overline{\mathcal{C}}_{A1}(X))$:} $$x_3>x_{10}>x_6>x_{11}>x_{14}>x_2>x_5>x_8>x_{13}>x_{7}>x_{9}>x_1>x_4>x_{12}>x_{15}.$$

Thus, the third bone transplant replacement material which is the optimal alternative for the determination should be selected for the patient by the clinician.
\begin{table}[htbp]
	\begin{center}\caption{The rankings based on different model}\label{ranking1}
\resizebox{.95\columnwidth}{!}{
\begin{tabular}{l|l}
			\hline
			MODEL& THE RANKING OF ALTERNATIVES \\
			\hline
			\rule{0pt}{10pt}
			$(\underline{\mathfrak{C}}_{A1}(X)$, $\overline{\mathfrak{C}}_{A1}(X))$ & $x_3>x_{10}>x_6>x_{11}>x_{14}>x_2>x_5>x_8>x_{13}>x_{9}>x_{7}>x_4>x_1>x_{12}>x_{15}$\\
$(\underline{\mathfrak{C}}_{A2}(X)$, $\overline{\mathfrak{C}}_{A2}(X))$ & $x_3>x_{10}>x_6>x_{11}>x_{14}>x_2>x_5>x_{13}>x_{8}>x_{7}>x_{9}>x_4>x_1>x_{12}>x_{15}$\\
$(\underline{\mathfrak{C}}_{A3}(X)$, $\overline{\mathfrak{C}}_{A3}(X))$ & $x_3>x_{6}>x_{10}>x_{11}>x_{14}>x_2>x_5>x_{13}>x_{8}>x_{4}>x_{9}>x_7>x_{12}>x_{1}>x_{15}$\\

$(\underline{\mathfrak{C}}_{B}(X)$, $\overline{\mathfrak{C}}_{B}(X))$ & $x_3>x_{10}>x_6>x_{11}>x_{14}>x_2>x_5>x_8>x_{13}>x_{9}>x_{7}>x_4>x_1>x_{12}>x_{15}$\\
$(\underline{\mathfrak{C}}_{C}(X)$, $\overline{\mathfrak{C}}_{C}(X))$ & $x_3>x_{6}>x_{10}>x_{11}>x_{14}>x_2>x_5>x_{13}>x_{8}>x_{4}>x_{9}>x_7>x_{12}>x_{1}>x_{15}$\\
$(\underline{\mathfrak{C}}_{D}(X)$, $\overline{\mathfrak{C}}_{D}(X))$ & $x_3>x_{10}>x_6>x_{11}>x_{14}>x_2>x_5>x_{13}>x_{8}>x_{7}>x_{9}>x_4>x_1>x_{12}>x_{15}$\\
$(\underline{\mathfrak{C}}_{E}(X)$, $\overline{\mathfrak{C}}_{E}(X))$ & $x_3>x_{10}>x_6>x_{11}>x_{14}>x_5>x_{13}>x_2>x_{8}>x_{7}>x_{4}>x_9>x_{12}>x_{1}>x_{15}$\\
$(\underline{\mathfrak{C}}_{F1}(X)$, $\overline{\mathfrak{C}}_{F1}(X))$ & $x_3>x_{10}>x_6>x_{11}>x_{14}>x_2>x_5>x_8>x_{13}>x_{9}>x_{7}>x_4>x_1>x_{12}>x_{15}$\\
$(\underline{\mathfrak{C}}_{F2}(X)$, $\overline{\mathfrak{C}}_{F2}(X))$ & $x_3>x_{10}>x_6>x_{11}>x_{14}>x_2>x_5>x_8>x_{13}>x_{9}>x_{7}>x_4>x_1>x_{12}>x_{15}$\\
$(\underline{\mathfrak{C}}_{G}(X)$, $\overline{\mathfrak{C}}_{G}(X))$ & $x_3>x_{10}>x_6>x_{11}>x_{14}>x_5>x_{2}>x_{13}>x_{8}>x_{7}>x_{9}>x_4>x_{12}>x_{1}>x_{15}$\\
$(\underline{\mathfrak{C}}_{H1}(X)$, $\overline{\mathfrak{C}}_{H1}(X))$ & $x_3>x_{10}>x_6>x_{11}>x_{14}>x_5>x_{13}>x_2>x_{8}>x_{7}>x_{4}>x_9>x_{12}>x_{1}>x_{15}$\\
$(\underline{\mathfrak{C}}_{H2}(X)$, $\overline{\mathfrak{C}}_{H2}(X))$ & $x_3>x_{10}>x_6>x_{11}>x_{14}>x_5>x_{13}>x_2>x_{8}>x_{7}>x_{4}>x_9>x_{12}>x_{1}>x_{15}$\\
$(\underline{\mathfrak{C}}_{I}(X)$, $\overline{\mathfrak{C}}_{I}(X))$ & $x_3>x_{10}>x_6>x_{11}>x_{14}>x_{13}>x_{2}>x_{5}>x_{8}>x_{4}>x_{7}>x_9>x_{12}>x_{1}>x_{15}$\\
$(\underline{\mathfrak{C}}_{J}(X)$, $\overline{\mathfrak{C}}_{J}(X))$ & $x_3>x_{10}>x_6>x_{11}>x_{14}>x_5>x_{2}>x_{13}>x_{8}>x_{7}>x_{9}>x_4>x_{12}>x_{1}>x_{15}$\\
$(\underline{\mathfrak{C}}_{K}(X)$, $\overline{\mathfrak{C}}_{K}(X))$ & $x_3>x_{10}>x_6>x_{11}>x_{14}>x_{13}>x_{2}>x_{5}>x_{8}>x_{4}>x_{7}>x_9>x_{12}>x_{1}>x_{15}$\\
$(\underline{\mathfrak{C}}_{L}(X)$, $\overline{\mathfrak{C}}_{L}(X))$ & $x_3>x_{10}>x_6>x_{11}>x_{14}>x_5>x_{13}>x_2>x_{8}>x_{7}>x_{4}>x_9>x_{12}>x_{1}>x_{15}$\\
$(\underline{\mathfrak{C}}_{M}(X)$, $\overline{\mathfrak{C}}_{M}(X))$ & $x_3>x_{10}>x_6>x_{11}>x_{14}>x_5>x_{13}>x_2>x_{8}>x_{7}>x_{4}>x_9>x_{12}>x_{1}>x_{15}$\\\hline
			\rule{0pt}{10pt}
			$(\underline{\mathcal{C}}_{A1}(X)$, $\overline{\mathcal{C}}_{A1}(X))$ & $x_3>x_{10}>x_6>x_{11}>x_{14}>x_2>x_5>x_8>x_{13}>x_{7}>x_{9}>x_1>x_4>x_{12}>x_{15}$  \\
$(\underline{\mathcal{C}}_{A2}(X)$, $\overline{\mathcal{C}}_{A2}(X))$ & $x_3>x_{10}>x_6>x_{11}>x_{14}>x_2>x_5>x_8>x_{13}>x_{7}>x_{9}>x_1>x_4>x_{12}>x_{15}$  \\
			$(\underline{\mathcal{C}}_{B}(X)$, $\overline{\mathcal{C}}_{B}(X))$ & $x_3>x_{10}>x_6>x_{11}>x_{14}>x_2>x_5>x_8>x_{13}>x_{7}>x_{9}>x_1>x_4>x_{12}>x_{15}$  \\
$(\underline{\mathcal{C}}_{C}(X)$, $\overline{\mathcal{C}}_{C}(X))$ & $x_3>x_{10}>x_6>x_{11}>x_{14}>x_2>x_5>x_8>x_{13}>x_{7}>x_{9}>x_4>x_1>x_{12}>x_{15}$  \\
$(\underline{\mathcal{C}}_{D}(X)$, $\overline{\mathcal{C}}_{D}(X))$ & $x_3>x_{10}>x_6>x_{11}>x_{14}>x_2>x_5>x_8>x_{13}>x_{7}>x_{9}>x_1>x_4>x_{12}>x_{15}$  \\
$(\underline{\mathcal{C}}_{E}(X)$, $\overline{\mathcal{C}}_{E}(X))$ & $x_3>x_{6}>x_{10}>x_{11}>x_{14}>x_5>x_{13}>x_2>x_{8}>x_{7}>x_{4}>x_9>x_{12}>x_{1}>x_{15}$\\
$(\underline{\mathcal{C}}_{F1}(X)$, $\overline{\mathcal{C}}_{F1}(X))$ & $x_3>x_{10}>x_6>x_{11}>x_{14}>x_2>x_5>x_8>x_{13}>x_{7}>x_{9}>x_1>x_4>x_{12}>x_{15}$  \\
$(\underline{\mathcal{C}}_{F2}(X)$, $\overline{\mathcal{C}}_{F2}(X))$ & $x_3>x_{10}>x_6>x_{11}>x_{14}>x_2>x_5>x_8>x_{13}>x_{7}>x_{9}>x_1>x_4>x_{12}>x_{15}$  \\
$(\underline{\mathcal{C}}_{G}(X)$, $\overline{\mathcal{C}}_{G}(X))$ & $x_3>x_{6}>x_{10}>x_{11}>x_{14}>x_5>x_{13}>x_2>x_{8}>x_{7}>x_{9}>x_4>x_{12}>x_{1}>x_{15}$  \\
$(\underline{\mathcal{C}}_{H1}(X)$, $\overline{\mathcal{C}}_{H1}(X))$ & $x_3>x_{6}>x_{10}>x_{11}>x_{14}>x_5>x_{13}>x_2>x_{8}>x_{7}>x_{4}>x_9>x_{12}>x_{1}>x_{15}$\\
$(\underline{\mathcal{C}}_{H2}(X)$, $\overline{\mathcal{C}}_{H2}(X))$ & $x_3>x_{6}>x_{10}>x_{11}>x_{14}>x_5>x_{13}>x_2>x_{8}>x_{7}>x_{4}>x_9>x_{12}>x_{1}>x_{15}$\\
$(\underline{\mathcal{C}}_{I}(X)$, $\overline{\mathcal{C}}_{I}(X))$ & $x_3>x_{6}>x_{10}>x_{11}>x_{5}>x_{13}>x_{14}>x_4>x_{2}>x_{8}>x_{7}>x_9>x_{12}>x_{1}>x_{15}$  \\
$(\underline{\mathcal{C}}_{J}(X)$, $\overline{\mathcal{C}}_{J}(X))$ & $x_3>x_{6}>x_{10}>x_{11}>x_{14}>x_5>x_{13}>x_2>x_{8}>x_{7}>x_{9}>x_4>x_{12}>x_{1}>x_{15}$  \\
$(\underline{\mathcal{C}}_{K}(X)$, $\overline{\mathcal{C}}_{K}(X))$ & $x_3>x_{6}>x_{10}>x_{11}>x_{5}>x_{13}>x_{14}>x_4>x_{2}>x_{8}>x_{7}>x_9>x_{12}>x_{1}>x_{15}$  \\
$(\underline{\mathcal{C}}_{L}(X)$, $\overline{\mathcal{C}}_{L}(X))$ & $x_3>x_{6}>x_{10}>x_{11}>x_{14}>x_5>x_{13}>x_2>x_{8}>x_{7}>x_{4}>x_9>x_{12}>x_{1}>x_{15}$\\
$(\underline{\mathcal{C}}_{M}(X)$, $\overline{\mathcal{C}}_{M}(X))$ & $x_3>x_{6}>x_{10}>x_{11}>x_{14}>x_5>x_{13}>x_2>x_{8}>x_{7}>x_{4}>x_9>x_{12}>x_{1}>x_{15}$\\
			\hline
		
		\end{tabular}}
	\end{center}
	
\end{table}
\subsection{Comparisons and analyses}\label{ss63}
According to the results we obtained before, the effectiveness and the reasonableness of new approaches are discussed in this subsection.
\subsubsection{The analyses among new models}
In the light of the ranking outcomes exhibited in Table \ref{ranking1}, although the results based on different models are high similarity, and all models give the same optimal solution, we also find that changes in the selection of neighborhood operators can lead to subtle changes in the sorting results. The essence of these variations is that for approximation operators based on the same pair of logical operators, the different neighborhood operator implies the change in the order relation among pairs of approximation operators, which in turn leads to the changes in assigning a weight of each attribute. Especially for the TNRFRS models, on the basis of the partially order relation defined in Remark \ref{rere}, a model is at the bottom of the Hasse diagram implying the largest lower approximation set and the smallest upper approximation set, which represents, to some extent, the most accurate approximation of the target set. Thus, with the underlying logical operator unchanged, the sorting result based the model $(\underline{\mathcal{C}}_{A1}(X)$, $\overline{\mathcal{C}}_{A1}(X))$ is the best representative of TNRFRS models. The situation is slightly different for ONRFRS models because the relations among most models are incomparable. The reasons for this change are the disappearance of the associative law for the overlap functions and the change of the boundary  conditions, however, this does not prevent ONRFRS models and TNRFRS models from making similar decision outcomes. Furthermore, for more complicated cases in practical applications, such as image processing, the pending data cannot be arbitrarily combined with others, the ONRFRS models are more suitable and the incomparable neighborhood operators provide more options for decision-makers.
\begin{figure}[H]\caption{Comparison of partial consistency rate}\label{OT}
  \centering
  \includegraphics[width=14cm]{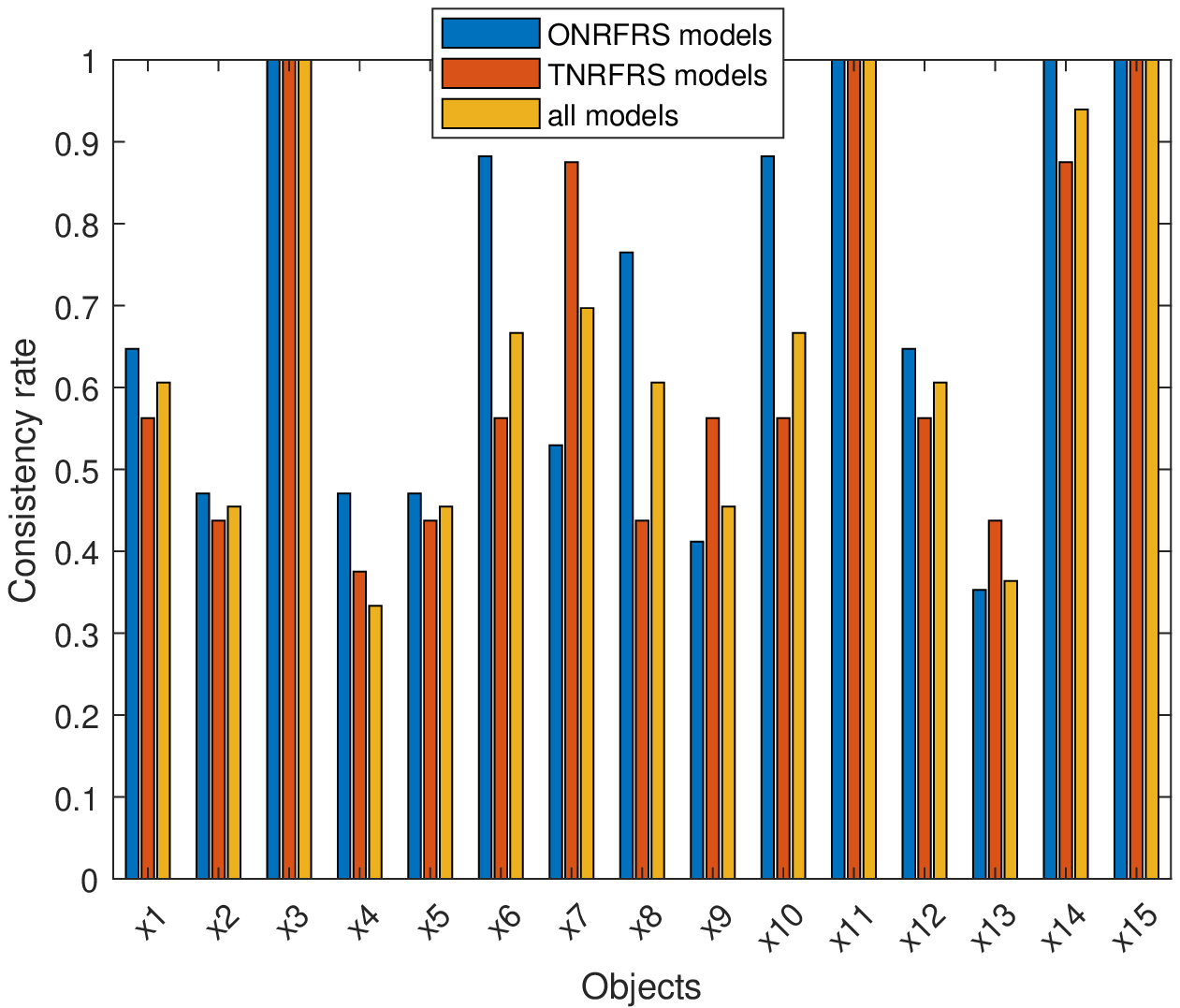}\\

\end{figure}

Comparing ONRFRS models with TNRFRS models, we find that the sorting results obtained by ONRFRS models for each alternative are more consistent according to Figure \ref{OT}. For ONRFRS models, the objects' consistency rate more than $60\%$ are a total of nine, which are $x_1$, $x_3$, $x_6$, $x_8$, $x_{10}$, $x_{11}$, $x_{12}$, $x_{14}$ and $x_{15}$, respectively. Yet, the number of consistency rate more than $60\%$  is just five. Overall, the overlap function based-models are superior to the $t$-norm based models in terms of the consistency of decision-making outcomes. This does not mean that TNRFRS models are not valid methodologies, as we discussed previously, because of the lattice relations among TNRFRS models, the result of the model $(\underline{\mathcal{C}}_{A1}(X)$, $\overline{\mathcal{C}}_{A1}(X))$ is the best representative for this type of models. Thus, for decision-makers, $(\underline{\mathcal{C}}_{A1}(X)$, $\overline{\mathcal{C}}_{A1}(X))$ model is a good choice, if they would like to obtain a conclusion rapidly by just one model. On the other hand, due to the high consistency of ranking results, the series of models based on overlap function are better choices for decision-makers who do not want to recklessly make decisions with only one model. Certainly, combining the results from two kinds of models is also a brilliant idea, for the number of alternatives whose consistency rate is more than $60\%$ is even higher than ONRFRS models by means of this way.
\subsubsection{The analyses with the models based on different logical operators}
Through the comparison with the conclusions from different types of models which are constructed by diverse types of logical operators, we surmise that the models established by different logical operators of the same type also induce slightly different decision results but do not affect the validity of the models. In order to verity our deduce, we construct the new models which are based on the non-associative overlap function ${O_{m_{2}}}(x,y)=\min\{x^2,y^2\}$ that satisfies the condition (O7) and its $R_O$-implication $I_{O_{m_2}}=\left\{\begin{matrix}
		1, & \text{if}$\ $ x^2\leq y;\\
		\sqrt{y}, & \text{if}$\ $ x^2>y.
	\end{matrix}\right.$, and the left-continuous $t$-norm $\mathscr{T}_{m}=\min\{x,y\}$ and its $R$-implication $\mathscr{I}_{\mathscr{T}_{m}}=\left\{\begin{matrix}
		1, & \text{if}$\ $ x\leq y;\\
		y, & \text{if}$\ $ x>y.
	\end{matrix}\right.$ Without losing generality, we only need to give the results based on $(\underline{\mathfrak{C}}_{A1}(X)$, $\overline{\mathfrak{C}}_{A1}(X))$, $(\underline{\mathcal{C}}_{A1}(X)$, $\overline{\mathcal{C}}_{A1}(X))$, $(\underline{\mathfrak{C}}_{H1}(X)$, $\overline{\mathfrak{C}}_{H1}(X))$ and $(\underline{\mathcal{C}}_{H1}(X)$, $\overline{\mathcal{C}}_{H1}(X))$ that separately represent the models defined on the original covering $\mathscr{C}$ which are based on overlap functions and their $R_O$-implications, and $t$-norms and their $R$-implications, respectively. The rankings are shown in Table \ref{ranking2}.
\begin{table}[t]
	\begin{center}\caption{The rankings based on models with different logical operators}\label{ranking2}
\resizebox{.95\columnwidth}{!}{
\begin{tabular}{l|l|l}
			\hline
			MODEL& BASIS&THE RANKING OF ALTERNATIVES \\
			\hline
			\rule{0pt}{10pt}
			$(\underline{\mathfrak{C}}_{A1}(X)$, $\overline{\mathfrak{C}}_{A1}(X))$ & $I_{O_D}$& $x_3>x_{10}>x_6>x_{11}>x_{14}>x_2>x_5>x_8>x_{13}>x_{9}>x_{7}>x_4>x_1>x_{12}>x_{15}$\\
\quad & $I_{O_{m_2}}$& $x_3>x_{10}>x_6>x_{11}>x_{14}>x_2>x_5>x_{8}>x_{13}>x_{7}>x_{9}>x_1>x_4>x_{12}>x_{15}$\\

			\hline
		\rule{0pt}{10pt}
			$(\underline{\mathfrak{C}}_{H1}(X)$, $\overline{\mathfrak{C}}_{H1}(X))$ & $O_D$& $x_3>x_{10}>x_6>x_{11}>x_{14}>x_5>x_{13}>x_2>x_{8}>x_{7}>x_{4}>x_9>x_{12}>x_{1}>x_{15}$\\
\quad & $O_{m_2}$& $x_3>x_{10}>x_6>x_{11}>x_{14}>x_5>x_{13}>x_{2}>x_{8}>x_{7}>x_{4}>x_9>x_{12}>x_1>x_{15}$\\

			\hline\rule{0pt}{10pt}
			$(\underline{\mathcal{C}}_{A1}(X)$, $\overline{\mathcal{C}}_{A1}(X))$ & $\mathscr{I}_{\mathscr{T}_D}$& $x_3>x_{10}>x_6>x_{11}>x_{14}>x_2>x_5>x_8>x_{13}>x_{7}>x_{9}>x_1>x_4>x_{12}>x_{15}$\\
\quad & $\mathscr{I}_{\mathscr{T}_{m_2}}$& $x_3>x_{10}>x_6>x_{11}>x_{14}>x_2>x_5>x_{8}>x_{13}>x_{7}>x_{9}>x_1>x_4>x_{12}>x_{15}$\\

			\hline\rule{0pt}{10pt}
			$(\underline{\mathcal{C}}_{H1}(X)$, $\overline{\mathcal{C}}_{A1}(X))$ & $\mathscr{T}_D$& $x_3>x_{6}>x_{10}>x_{11}>x_{14}>x_5>x_{13}>x_2>x_{8}>x_{7}>x_{4}>x_9>x_{12}>x_{1}>x_{15}$\\
\quad & $\mathscr{T}_{m_2}$& $x_3>x_{6}>x_{10}>x_{11}>x_{14}>x_5>x_{13}>x_{2}>x_{8}>x_{7}>x_{4}>x_9>x_{12}>x_{1}>x_{15}$\\

			\hline
		\end{tabular}}
	\end{center}
	
\begin{figure}[H]
  \centering\caption{The rankings shown in Table \ref{ranking2}}\label{AC}
  \includegraphics[width=14cm]{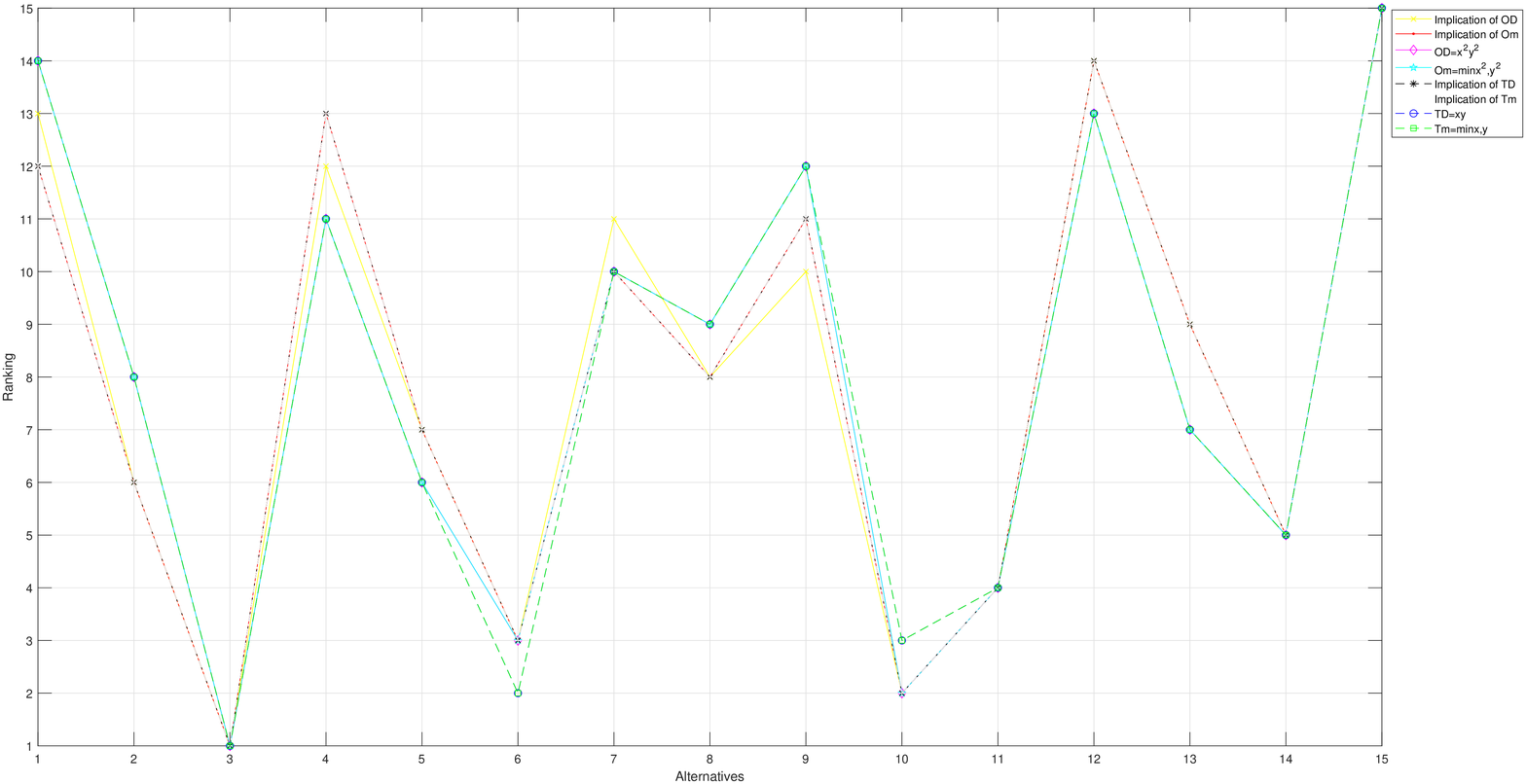}\\

\end{figure}
\end{table}

\begin{figure}[t] \caption{The rankings based on different fuzzy logical operators}\label{OITTIO}
  \centering
  \includegraphics[width=6cm]{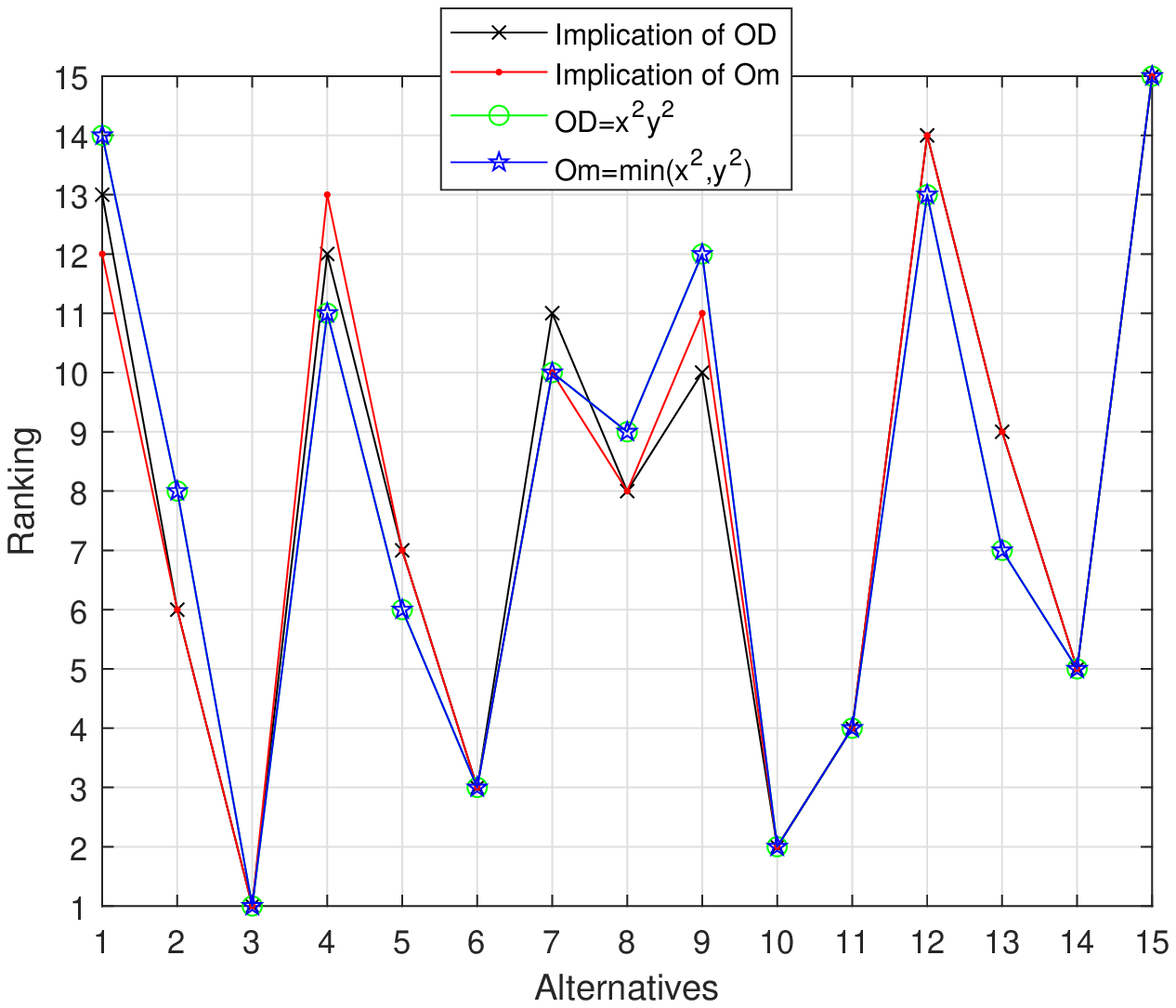}\includegraphics[width=6cm]{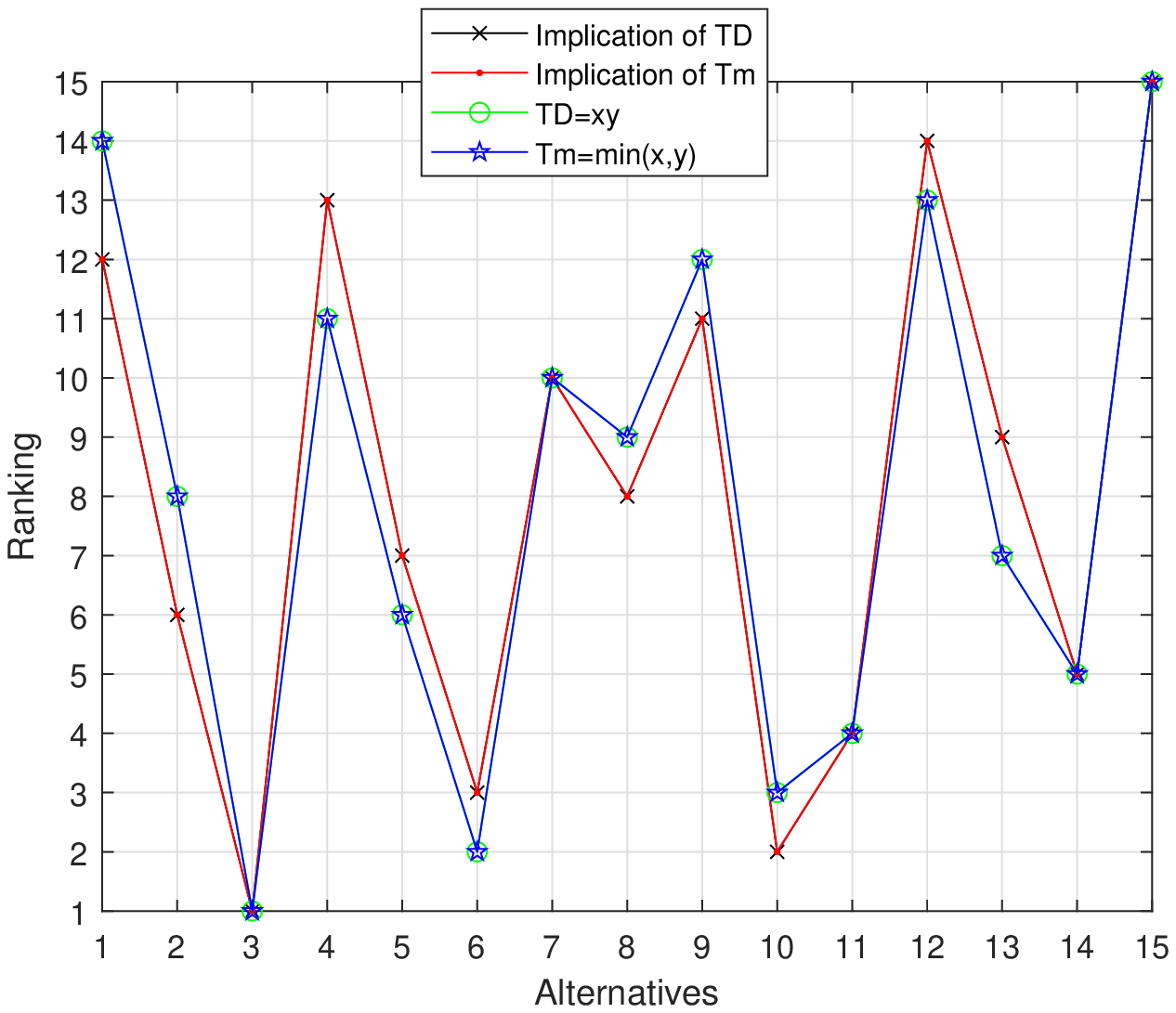}\\
  \includegraphics[width=6cm]{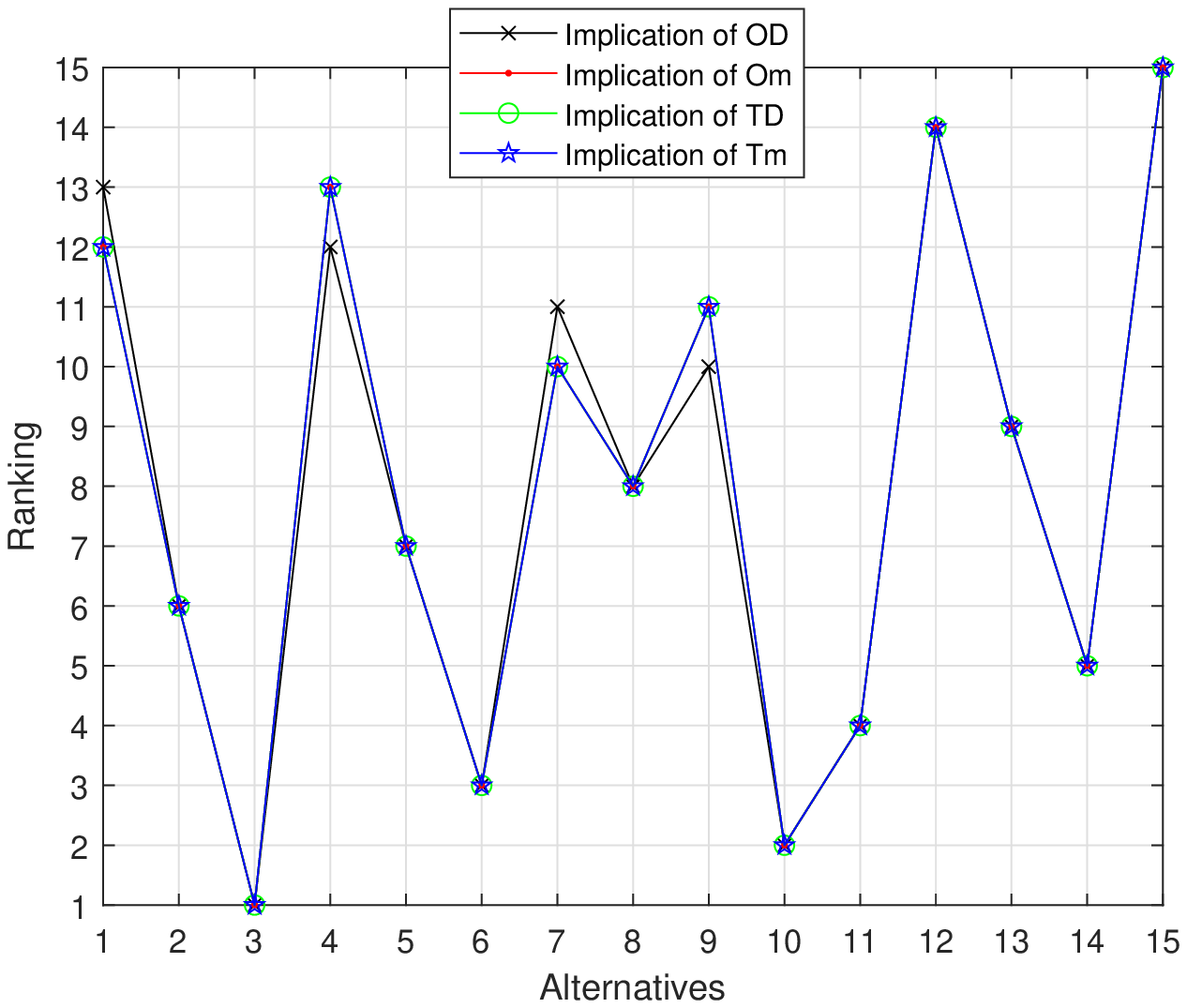}\includegraphics[width=6cm]{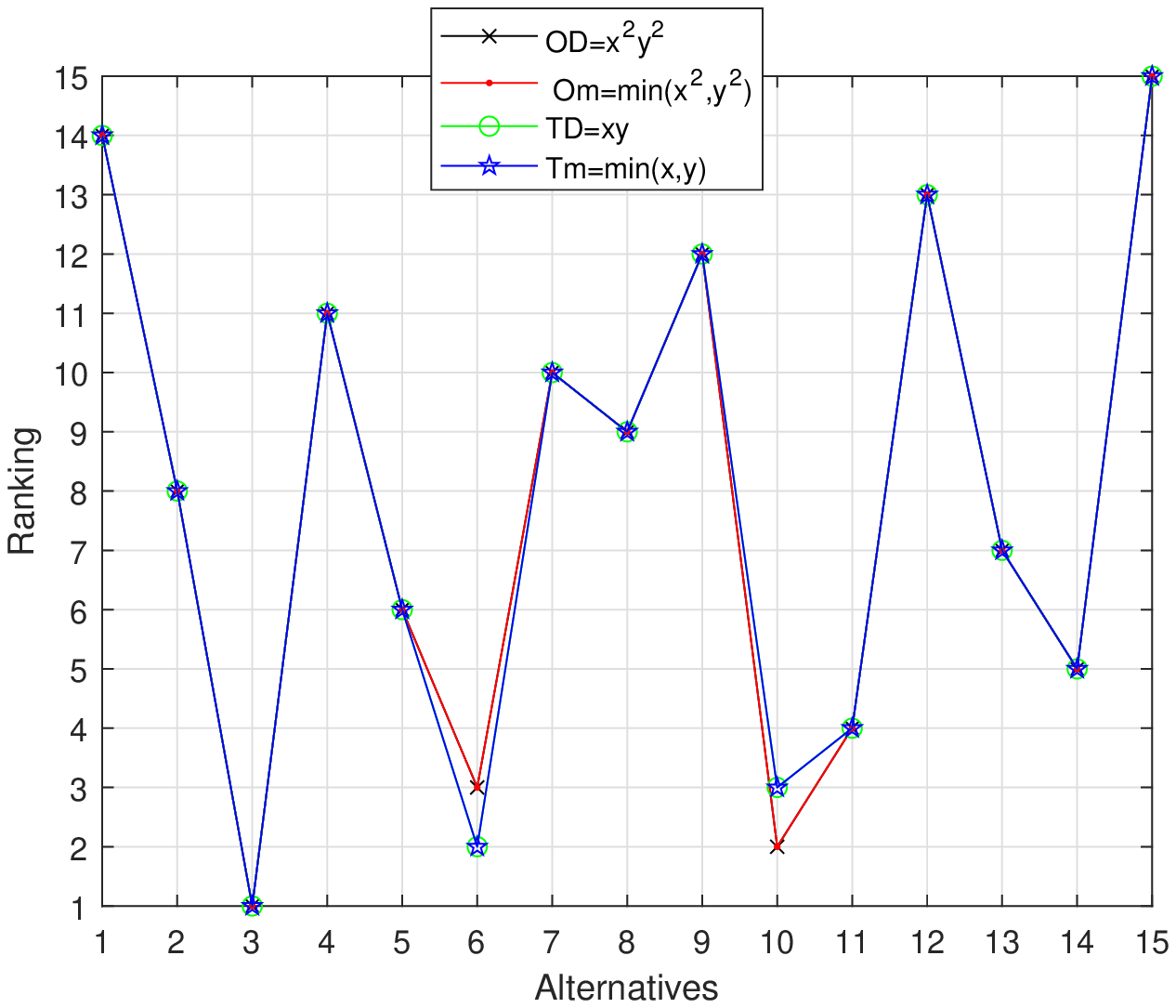}\\

\end{figure}

According to Figure \ref{AC} and Figure \ref{OITTIO}, we find that although the fuzzy logical operators which compose the models are different, the sorting results still show a high degree of consistency. Considering about Figure \ref{AC}, this picture shows that the impact of eight different logical operators respectively on ranking results that the line graphs formed based on the sorting results all reflect close trends. In Figure \ref{OITTIO}, the sorting results which are displayed in Table \ref{ranking2} are discussed in four categories, they are respectively between different types of fuzzy logical operators and fuzzy logical operators, between different implications and implications, and between logical operators and their corresponding implications,  all giving the same optimal solution and the direction of sorting approximation. The above means the decision-making results of our approach do not change significantly with the change of the fuzzy logical operators, which further illustrates the effectiveness and objectivity of our methodology. In comparison to the different fuzzy logical operators of the same type, the diverse type of fuzzy logical operators lead to bigger changes in the rankings, this subtle difference is due to the fact that the overlap function is inherently different from the $t$-norm. Besides, the selection of different fuzzy logical operators can reflect the preferences of decision-makers and adapt to the practical application background.
\subsubsection{The analyses with the existing methods}
In previous work, we put forward a new fuzzy TOPSIS methodology for the MADM issues without known weight based on two newly proposed models, which are in light of the advantages of fuzzy rough set theory in solving uncertain information. In order to further verify the effectiveness of our new method, the results of us are compared with the existing approaches in this part. The sorting results are shown in Table \ref{ranking3}.

\begin{table}[H]
	\begin{center}\caption{The rankings based on different mothods}\label{ranking3}
\resizebox{.95\columnwidth}{!}{
\begin{tabular}{l|l}
			\hline
			METHOD& THE RANKING OF ALTERNATIVES \\
			\hline
			\rule{0pt}{10pt}
			Our method based on $(\underline{\mathfrak{C}}_{A1}(X)$, $\overline{\mathfrak{C}}_{A1}(X))$ & $x_3>x_{10}>x_6>x_{11}>x_{14}>x_2>x_5>x_8>x_{13}>x_{9}>x_{7}>x_4>x_1>x_{12}>x_{15}$\\

Our method based on $(\underline{\mathfrak{C}}_{H1}(X)$, $\overline{\mathfrak{C}}_{H1}(X))$ & $x_3>x_{10}>x_6>x_{11}>x_{14}>x_5>x_{13}>x_2>x_{8}>x_{7}>x_{4}>x_9>x_{12}>x_{1}>x_{15}$\\

		Our method based on	$(\underline{\mathcal{C}}_{A1}(X)$, $\overline{\mathcal{C}}_{A1}(X))$ & $x_3>x_{10}>x_6>x_{11}>x_{14}>x_2>x_5>x_8>x_{13}>x_{7}>x_{9}>x_1>x_4>x_{12}>x_{15}$  \\

Our method based on $(\underline{\mathcal{C}}_{H1}(X)$, $\overline{\mathcal{C}}_{H1}(X))$ & $x_3>x_{6}>x_{10}>x_{11}>x_{14}>x_5>x_{13}>x_2>x_{8}>x_{7}>x_{4}>x_9>x_{12}>x_{1}>x_{15}$\\
The WAA method \cite{yager1988ordered} & $x_3>x_{6}>x_{10}>x_{11}>x_{14}>x_2>x_{5}>x_{13}>x_{8}>x_{9}>x_{7}>x_1>x_{4}>x_{12}>x_{15}$\\
The OWA method \cite{yager1988ordered} & $x_3>x_{6}>x_{10}>x_{11}>x_{14}>x_2>x_{5}>x_{13}>x_{8}>x_{9}>x_{7}>x_1>x_{4}>x_{12}>x_{15}$\\
The OWGA method \cite{xu2003overview} & $x_6>x_{11}>x_{10}>x_{3}>x_{2}>x_{14}>x_{8}>x_{5}>x_{13}>x_{9}>x_{7}>x_4>x_{1}>x_{12}>x_{15}$\\
The PROMETHEE II method $(p=0.5)$ \cite{brans1986select} & $x_3>x_{6}>x_{10}>x_{11}>x_{14}>x_{13}>x_{2}>x_{8}>x_{5}>x_{1}>x_{7}>x_4>x_{9}>x_{12}>x_{15}$\\
The TOPSIS method \cite{hwang1981methods,kacprzak2019doubly,sun2018innovative} & $x_3>x_{6}>x_{10}>x_{11}>x_{14}>x_{2}>x_{5}>x_{13}>x_{8}>x_{9}>x_{7}>x_1>x_{4}>x_{12}>x_{15}$\\
The TOPSIS method based on an F$\beta$CAS\cite{zhan2019covering} & $x_3>x_{10}>x_{6}>x_{11}>x_{14}>x_{2}>x_{13}>x_{5}>x_{8}>x_{1}>x_{4}>x_7>x_{9}>x_{12}>x_{15}$\\
The TOPSIS method based on a VFCAS \cite{jiang2018covering} & $x_3>x_{6}>x_{10}>x_{11}>x_{14}>x_{2}>x_{5}>x_{8}>x_{13}>x_{9}>x_{7}>x_1>x_{4}>x_{12}>x_{15}$\\
The TOPSIS method based on a $\lambda$AS \cite{yu2019lambda} & $x_4>x_{6}>x_{1}>x_{8}>x_{9}>x_{7}>x_{15}>x_{2}>x_{5}>x_{10}>x_{13}>x_{12}>x_{11}>x_{14}>x_{3}$\\
The TOPSIS method based on an FCAS \cite{zhang2019topsis} & $x_3>x_{6}>x_{10}>x_{11}>x_{14}>x_{2}>x_{5}>x_{13}>x_{8}>x_{9}>x_{7}>x_1>x_{4}>x_{12}>x_{15}$\\
			\hline
		
		\end{tabular}}
	\end{center}
	
\end{table}

The value of the correlation coefficient is often used to reflect the degree of relevance between two sets of variables. Typically, these values fluctuate between $-1$ and $1$, where the positive and negative values represent positive and negative correlations, respectively. Besides, when the absolute value of the correlation coefficient is closer to $1$, the stronger the correlation between the two sides of the comparison, correspondingly, when the absolute value of the correlation coefficient is closer to $0$, the weaker the correlation between the two sides of the comparison. In order to present the relations  more visually among the sorting results based on different methods shown in Table \ref{ranking3}, we have the aid of the Spearman rank correlation coefficient \cite{myers2013research} which is one of the methods commonly used in statistics to calculate the dependence between two variables to describe them. Through numericalizing the ranking results, we can use the above method to calculate the correlations among the results obtained by each decision-making methodology. Then, the results of the calculation are presented in Table \ref{cscs} and their corresponding visualized representations are shown in Figure \ref{csss}. Where M1 is our method based on $(\underline{\mathfrak{C}}_{A1}(X)$, $\overline{\mathfrak{C}}_{A1}(X))$, M2 is our method based on $(\underline{\mathfrak{C}}_{H1}(X)$, $\overline{\mathfrak{C}}_{H1}(X))$, M3 is our method based on $(\underline{\mathcal{C}}_{A1}(X)$, $\overline{\mathcal{C}}_{A1}(X))$, M4 is our method based on $(\underline{\mathcal{C}}_{H1}(X)$, $\overline{\mathcal{C}}_{H1}(X))$, M5 is the WAA method, M6 is the OWA method, M7 is the OWGA method, M8 is the PROMETHEE II method, M9 is the TOPSIS method, M10 is the TOPSIS method based on an F$\beta$CAS, M11 is the TOPSIS method based on a VFCAS, M12 is the TOPSIS method based on a $\lambda$AS and M13 is the TOPSIS method based on an FCAS.
\begin{table}[t]
	\begin{center}	\caption{Spearman rank correlation coefficient analysis}\label{cscs}
\resizebox{.95\columnwidth}{!}{
\begin{tabular}{l|lllllllllllll}
			\hline
			$\rho$ & $M_1$  & $M_2$ & $M_3$ & $M_4$ & $M_5$ & $M_6$ & $M_7$ & $M_8$ & $M_9$ & $M_{10}$ & $M_{11}$ & $M_{12}$ & $M_{13}$  \\
			\hline
			
			$M_1$ & $1$ & $	0.9678$ & $0.9928$ & $	0.9642$ & $	0.9892$ & $	0.9892$ & $	0.9607$ & $	0.9392$ & $	0.9892$ & $	0.9535$ & $	0.9928$ & $	-0.4107$ & $	0.9892$\\
		
			$M_2$ & $0.9678$ & $	1$ & $	0.9642$ & $	0.9964$ & $0.9607$ & $	0.9607$ & $	0.9142$ & $	0.9392$ & $	0.9607$ & $	0.9464$ & $	0.9535$ & $	-0.4714$ & $	0.9607$\\
		
			$M_3$ & $0.9928$ & $	0.9642$ & $	1$ & $	0.9607$ & $	0.9892$ & $	0.9892$ & $	0.9535$ & $	0.9535$ & $	0.9892$ & $	0.9607$ & $	0.9928$ & $	-0.4214$ & $	0.9892$\\
		
			$M_4$ & $0.9642$ & $	0.9964$ & $	0.9607$ & $	1	$ & $0.9642$ & $	0.9642$ & $	0.9214$ & $	0.9428$ & $	0.9642$ & $	0.9428$ & $	0.9571$ & $	-0.4428$ & $	0.9642$\\
		
			$M_5$ & $0.9892$ & $	0.9607$ & $	0.9892$ & $	0.9642$ & $	1	$ & $1	$ & $0.9571$ & $	0.9571$ & $	1$ & $	0.9607$ & $	0.9964$ & $	-0.4142$ & $	1$\\

			$M_6$ & $0.9892$ & $	0.9607$ & $	0.9892$ & $	0.9642$ & $	1	$ & $1	$ & $0.9571$ & $	0.9571$ & $	1$ & $	0.9607$ & $	0.9964$ & $	-0.4142$ & $	1$\\

			$M_7$ & $0.9607$ & $0.9142$ & $	0.9535$ & $	0.9214$ & $	0.9571$ & $	0.9571$ & $1$ & $	0.9142$ & $	0.9571$ & $	0.9142$ & $	0.9642$ & $	-0.2821$ & $	0.9571$ \\

			$M_8$ & $0.9392$ & $	0.9392$ & $	0.9535$ & $	0.9428$ & $	0.9571$ & $	0.9571$ & $ 0.9142$ & $	1$ & $	0.9571$ & $	0.9857$ & $	0.9500$ & $	-0.3857$ & $	0.9571$ \\
			
			$M_9$ & $0.9892$ & $	0.9607$ & $	0.9892$ & $	0.9642$ & $	1$ & $	1$ & $	0.9571$ & $	0.9571$ & $	1$ & $	0.9607$ & $	0.9964$ & $	-0.4142$ & $	1$ \\

			$M_{10}$ & $0.9535$ & $	0.9464$ & $	0.9607$ & $	0.9428$ & $	0.9607$ & $	0.9607$ & $	0.9142$ & $	0.9857$ & $	0.9607$ & $1$ & $	0.9535$ & $	-0.4035$ & $	0.9607$ \\

			$M_{11}$ & $0.9928$ & $	0.9535$ & $	0.9928$ & $	0.9571$ & $	0.9964$ & $	0.9964$ & $	0.9642$ & $	0.9500$ & $	0.9964$ & $	0.9535$ & $	1$ & $	-0.3892$ & $	0.9964$ \\

			$M_{12}$ & $-0.4107$ & $	-0.4714$ & $	-0.4214$ & $	-0.4428$ & $	-0.4142$ & $	-0.4142$ & $	-0.2821$ & $	-0.3857$ & $	-0.4142$ & $	-0.4035$ & $	-0.3892$ & $	1$ & $	-0.4142$ \\

			$M_{13}$ & $0.9892$ & $	0.9607$ & $	0.9892$ & $	0.9642$ & $	1$ & $	1$ & $	0.9571$ & $	0.9571$ & $	1$ & $	0.9607$ & $	0.9964$ & $	-0.4142$ & $	1$ \\

			\hline
		
		\end{tabular}}
	\end{center}

\begin{figure}[H]
  \centering\caption{Visual Spearman rank correlation coefficient analysis}\label{csss}
  \includegraphics[width=10cm]{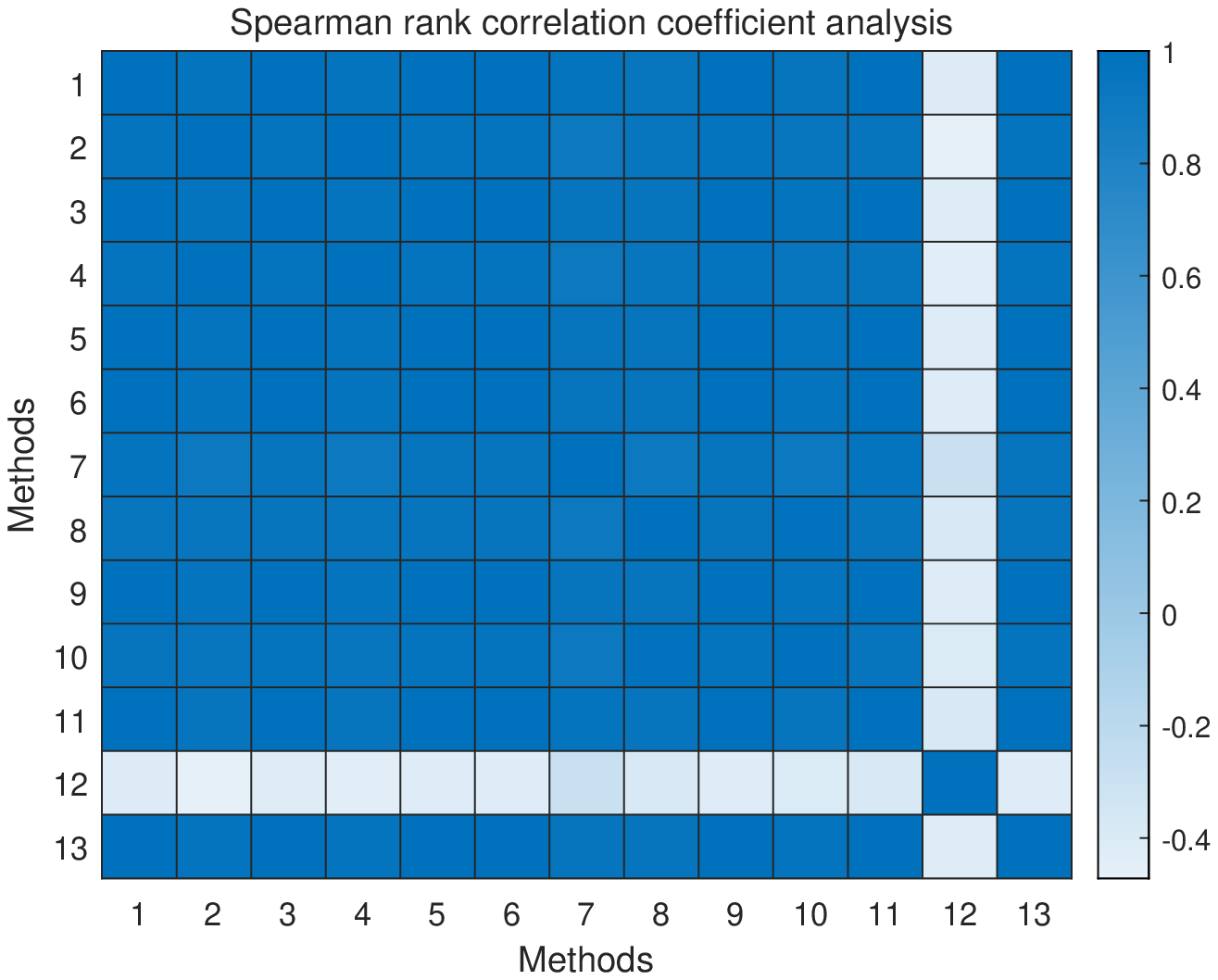}\\

\end{figure}
\end{table}

According to Table \ref{ranking3}, we find that both the results based on our four different models and the results based on other existing methods have the same optimal alternative (except for the method in \cite{xu2003overview} and the method in \cite{yu2019lambda}). In light of Figure \ref{csss}, it is more intuitive to show the high correlation, except for ranking result of the TOPSIS method based on a $\lambda$AS, between the results of our methods based on different models and other methods, and the Spearman rank correlation coefficient $\rho$ between each method is more than $0.9142$ as shown in Table \ref{cscs}. Thus, the effectiveness of the new models have been confirmed again. However, the results of our four models are quite different from M12 and the largest value of the correlation coefficients is just $-0.4107$. This does not indicate that our results are incorrect, because the other methods also reflect a weak correlation with M12 by comparing with their sort results.The case of M7 is different, although this method gives an entirely different optimal solution from others, the lowest value of dependence for M7 is more than $0.9$. In the practical sense, the ranking results by the diverse methods are different, and even giving distinct optimal alternatives is unsurprising. Therefore, the difference in the choice of decision methods reflects, to some extent, the decision maker's decision preferences.
 \section{Conclusion and future work}\label{Section 6}
 In order to extend the traditional covering based rough set to the fuzzy environment and give a different generalization from \cite{d2017fuzzy}, we have defined four original overlap function-based fuzzy operators. In addition, we also give two types of neighborhood-related fuzzy covering-based rough set models which are based on fuzzy neighborhood operators given in \cite{d2017fuzzy} and the new ones defined in the above text. First of all, six new derived fuzzy coverings have been given from an original covering, and for a finite fuzzy covering with an overlap function which satisfies (O7), we have obtained twenty four new fuzzy neighborhood operators via combining its six derived fuzzy coverings with four original operators. Furthermore, the equivalence relationships among all of the overlap function based fuzzy neighborhood operators have been discussed and they have also been grouped into seventeen groupings. In light of the classification, the partially order relations, $\leq$, among the different groups of operators have been discussed. By means of comparing with groupings and Hasse diagram of $t$-norm-based fuzzy neighborhood operators, we have concluded that the huge differences between these two kinds of fuzzy neighborhood operators are dependent on whether the constructing functions have exchange principle and their boundary conditions. Then on the basis of the fuzzy neighborhood operators proposed in \cite{d2017fuzzy} and the new fuzzy neighborhood operators investigated in this paper, we obtain two types of neighborhood-related fuzzy covering-based rough set models, and the properties and classifications of new models have been discussed and proved in detail. Finally, a new fuzzy TOPSIS method based on our models are proposed to deal with a biosynthetic nanomaterials selection issue and the effectiveness of  new models is verified in the final subsection of this article. The decision-making processes confirm that for the traditional decision-making issues, overlap functions are no less than conventional fuzzy logical operators.

However, there is a problem need to deal with, which is the relationship between the grouping $C$ and the grouping $A1$. In this paper, we fail to prove or find an example to verify the relation between them, so that we have assumed they are incomparable when we drew the Hasse diagram. We expect this problem can be approached in the future. Furthermore, there are some essential future researches about the decision-making model based on overlap functions. Comparing with $t$-norms, the researches about establishing the fuzzy rough set models based on overlap function are not many, although overlap functions are more sophisticated, more in line with reality, after all, in real-world applications, not all data can be arbitrarily combined. Therefore, the overlap function-based fuzzy rough set models are of high application values in data mining techniques, image processing techniques, and so on. In summary, the applications of overlap function still have great research prospects.
 \section*{Acknowledgements}
 This research was supported by the National Natural Science Foundation of China (Grant no. 12101500) and the Chinese Universities Scientific Fund (Grant no. 2452018054).


\end{document}